\documentclass[twoside,a4paper,12pt,centertags]{amsart}
\usepackage[T1]{fontenc}
\usepackage[utf8]{inputenc}
\usepackage{subfiles}
\usepackage{multicol}
\makeatletter
\usepackage[section]{placeins}
\AtBeginDocument{%
	\expandafter\renewcommand\expandafter\subsection\expandafter{%
		\expandafter\@fb@secFB\subsection
	}%
}
\makeatother
\makeatletter
\@namedef{subjclassname@2020}{%
	\textup{2020} Mathematics Subject Classification}
\makeatother
\usepackage{amsfonts,amssymb,amsmath,amsthm}
\usepackage{mathrsfs}
\usepackage{nicefrac}
\usepackage{latexsym}
\usepackage{mathtools}
\usepackage[noadjust]{cite}
\usepackage{paralist}
\usepackage{longtable}
\usepackage{float}
\usepackage{tikz, tikz-cd}
\usetikzlibrary{patterns}
\usepackage{hhline}
\usepackage{import}
\usepackage{enumitem}
\usepackage[colorlinks=true,citecolor=magenta,linkcolor=cyan,backref]{hyperref}
\AtBeginDocument{}

\numberwithin{equation}{section}

\theoremstyle{plain}
\newtheorem{thm}{Theorem}[section]
\newtheorem{lem}[thm]{Lemma}
\newtheorem{cor}[thm]{Corollary}
\newtheorem{prop}[thm]{Proposition}

\theoremstyle{definition}
\newtheorem{defn}[thm]{Definition}
\newtheorem{rem}[thm]{Remark}

\usepackage{color}

\definecolor{gr}{rgb}   {0., 0.8, 0. } 
\definecolor{bl}{rgb}   {0., 0.5, 1. } 
\definecolor{mg}{rgb}   {0.7, 0., 0.7}

\DeclareMathOperator*{\esssup}{ess\,sup}
\renewcommand{\div}{\operatorname{div}}

\newcommand{\IC}{\mathbb{C}}

\newcommand{\IN}{\mathbb{N}}

\newcommand{\R}{\mathbb{R}}

\newcommand{\IZ}{\mathbb{Z}}

\newcommand{\cD}{\mathcal{D}} %Distributions
\newcommand{\cH}{\mathcal{H}}

\newcommand{\cE}{\mathcal{E}}

\newcommand{\cL}{\mathcal{L}}
\newcommand{\cM}{\mathcal{M}}

\newcommand{\cS}{\mathcal{S}} %Schwartz space

\newcommand{\cV}{\mathcal{V}}
\newcommand{\dom}{\operatorname{dom}}
\renewcommand{\L}{\operatorname{L}} %Lebesgue spaces
\newcommand{\Lloc}{\L_{\operatorname{loc}}} %local Lebesgue spaces
\newcommand{\C}{\operatorname{C}} %spaces of continuous functions
\renewcommand{\H}{\operatorname{H}} %Sobolev spaces, spaces of holomorphic functions
\newcommand{\W}{\operatorname{W}}% Sobolev spaes
\newcommand{\E}{\operatorname{E}}% 

\newcommand{\D}{\operatorname{D}}

\newcommand{\J}{\operatorname{J}}
\newcommand{\I}{\operatorname{I}}

\DeclareRobustCommand{\Hdot}{\dot{\H}\protect{\vphantom{H}}} %homogeneous space
\DeclareRobustCommand{\cVdot}{\dot{\cV}\protect{\vphantom{V}}} 
\DeclareRobustCommand{\Edot}{\dot{\E}\protect{\vphantom{E}}} 
\DeclareRobustCommand{\Deldot}{\dot{\Delta}\protect{\vphantom{\Delta}}} %homogeneous \Delta spaces
\newcommand{\LL}{\L^2_{t}\L^2_{x}}% L^2L^2 spaces.
\newcommand{\BMO}{\mathrm{BMO}}
\newcommand{\e}{\mathrm{e}} %Euler's number

\renewcommand{\d}{\mathrm{d}} %differential
\newcommand{\eps}{\varepsilon} %epsilon
\renewcommand\Re{\operatorname{Re}}

\def\angle#1#2{\langle #1,#2 \rangle} %brackets in x,
\def\anglep#1#2{( #1,#2 )} %parentheses in x,
\def\Angle#1#2{\langle\langle #1,#2 \rangle\rangle} %brackets in (t,x)
\def\Anglep#1#2{(( #1,#2 ))} %double parentheses in (t,x)
\def\Hstarinverse{{\cH^*}^{-1}}
\def\Lt{\cL}
\def\Ltstar{\Lt ^*}
\def\tphi{\tilde \phi}
\def\Ga{G}
\def\tGa{\widetilde G}
\def\P{\Pi}
\def\tP{\widetilde\Pi}
\def\tpsi{\tilde\psi}
\def\tf{\tilde f}

\def\Del{\Delta}
\newcommand\Htheta{\Hdot_{t}^{{-\theta/2}}\Hdot^{{\theta-1}}_{\vphantom{t} x}}
\newcommand\Hmtheta{\Hdot_{t}^{-\theta/2}\Hdot_{\vphantom{t} x}^{m(\theta-1)}}
\newcommand\HthetaV{\Hdot_{t}^{-\theta/2}V^{\theta-1}}
\def\oa{\mathbf{a}}
\def\ob{\mathbf{b}}
\newcommand\Id{\operatorname{Id}}

\def\Xint#1{\mathchoice
{\XXint\displaystyle\textstyle{#1}}%
{\XXint\textstyle\scriptstyle{#1}}%
{\XXint\scriptstyle\scriptscriptstyle{#1}}%
{\XXint\scriptscriptstyle%
\scriptscriptstyle{#1}}%
\!\int}
\def\XXint#1#2#3{{\setbox0=\hbox{$#1{#2#3}{%
\int}$ }
\vcenter{\hbox{$#2#3$ }}\kern-.6\wd0}}
\def\barint{\,\Xint -} % \, corrects the \! used in the definition
\def\bariint{\barint_{} \kern-.4em \barint}
\def\bariiint{\bariint_{} \kern-.4em \barint}
\renewcommand{\iint}{\int_{}\kern-.34em \int} %\, minor space between the integrals
\renewcommand{\iiint}{\iint_{}\kern-.34em \int} %\, minor space between the integrals
\title[A variational framework for parabolic systems]{A universal variational framework for parabolic equations and systems}

\author{Pascal Auscher}
\address{Universit\'e Paris-Saclay, CNRS, Laboratoire de Math\'{e}matiques d'Orsay, 91405 Orsay, France}
\email{pascal.auscher@universite-paris-saclay.fr}

\author{Moritz Egert}
\address{TU Darmstadt, Fachbereich Mathematik, Schlossgartenstr.\ 7, 64289 Darmstadt, Germany}
\email{egert@mathematik.tu-darmstadt.de}
\address{Data sharing not applicable to this article as no datasets were generated or analysed during the current study.}
\thanks{The authors were supported by the ANR project RAGE ANR-18-CE40-0012. A CC-BY 4.0 \url{https://creativecommons.
org/licenses/by/4.0/} public copyright license has been applied by the authors to the present document and will be applied to all subsequent
versions up to the Author Accepted Manuscript arising from this
submission, in accordance with the grant's open access conditions. The authors also acknowledge support and hospitality from the Hausdorff Research Institute for Mathematics funded by the Deutsche Forschungsgemeinschaft (DFG, German Research Foundation) under Germany's Excellence Strategy -- EXC-2047/1 -- 390685813, where part of this material was developed. The authors would like to thank Sylvie Monniaux for a very careful reading of earlier versions of their manuscript. }
\date{September 5, 2023}

\begin{document}

\maketitle

\begin{abstract} 
We propose a variational approach to solve  Cauchy problems for parabolic equations and  systems  independently of regularity theory for solutions. This produces a  universal and conceptually simple construction of fundamental solution operators (also called propagators)  for which we prove   $\operatorname{L}^2$ off-diagonal estimates, which is new under our assumptions. In the special case of systems for which pointwise local bounds hold for weak solutions, this provides Gaussian upper bounds for the corresponding  fundamental solution. In particular, we obtain a new proof of Aronson's estimates for real equations.  The scheme is general enough to allow systems with higher order elliptic parts on full space or second order elliptic parts on Sobolev spaces with boundary conditions.  Another new feature is that the control on lower order coefficients is within critical mixed time-space Lebesgue spaces or even  mixed Lorentz spaces.
\end{abstract}

\tableofcontents

\section{Introduction}

The classical treatment of parabolic problems begins with solving the Cauchy problem with or without forcing terms and  representing 
solutions by what is called fundamental solutions. Here, we consider operators of the form $\partial_{t}+\Lt $, where $\Lt $ is an elliptic operator in divergence form with possibly complex-valued coefficients.  Coefficients depend on all space and time variables. We assume strongly (G\aa rding) elliptic  {and bounded} higher order coefficients and   unbounded lower order coefficients  controlled in mixed Lebesgue and even Lorentz norms that are compatible with Sobolev embeddings for solutions. In particular, our treatment includes parabolic Schr\"odinger operators with Coulomb like potentials.  

When the coefficients are regular, several methods are possible to construct the fundamental solution and the most efficient one is via a parametrix, using the so-called  freezing point technique, which reduces the situation to space-independent coefficients for which fundamental solutions  are explicit kernels $\Gamma(t,x, s,y)$ with exponential decay in  $(|x-y|^{2m}/|t-s|)^{1/(2m-1)}$, where $2m$ is the order of the elliptic operator~\cite{Fr}. When $m=1$, this is the Gaussian decay.

When coefficients become irregular (measurable, unbounded), one goes through the theory of weak solutions that was developed in the 1950's-60's, culminating in the treatise by Lady{\v{z}}enskaja, Solonnikov and Ural'ceva \cite{LSU}. In parallel, when the coefficients are real-valued and the elliptic operator has order two, Aronson constructed generalized fundamental solutions, using Riesz representation theorems as a consequence of well-posedness of Cauchy problems  that generate bounded solutions. He also proved Gaussian upper and lower bounds \cite{Ar67, Ar68}. This supposedly closed the topic but here we shall reveal some new phenomena. 

Our starting point is the {guiding principle that many results on elliptic problems} have counterparts in parabolic world, taking into account evolution with respect to time. However, elliptic problems are tackled using a coercive variational formulation, while parabolic problems are  attacked via the Cauchy problem as mentioned above. The presence of the first order time derivative seems to forbid any possibility of coercivity as in the elliptic case. 

Nevertheless,  the heat kernel  $$(t,x,s,y) \mapsto 1_{\{t>s\}} \frac{\e^{-\frac{|x-y|^2}{4(t-s)}}}{(4\pi (t-s))^{n/2}} $$
can be seen as the kernel of the operator $(\partial_{t}-\Delta)^{-1}$. Here, the inverse can be computed using Fourier transform, but for more general parabolic operators this is not  possible. Our question is whether  some form of invertibility can still be implemented.

We show that indeed there is a variational formulation in the parabolic setting, too. That is, we find a variational space $\cV$ such that if $\partial_{t}+\Lt \colon \cV \to \cV'$, $\cV'$ being its dual, is invertible, then one can represent the inverse by  Green operators that  eventually become  the fundamental solution operators for the Cauchy problem (and whose kernels, whenever they exist in a pointwise sense, give a generalized fundamental solution). Invertibility and causality are then checked under appropriate coercivity requirements.  In other words, we are reversing the order of the usual arguments. Our main conclusion for Cauchy problems in the case of coefficients in mixed Lebesgue spaces  is in Theorem~\ref{thm:Cauchy}.
 
One may think this is a matter of cosmetic changes in the theory, but it is not.  For instance, in the case of second order elliptic part, the usual energy space of weak solutions $\L^2(\I; \H^1)\cap \L^\infty(\I; \L^2)$ or the smaller Lions' space $\L^2(\I; \H^1)\cap \H^1(\I; \H^{-1})$\footnote{Here, we do not define all spaces precisely as we just want to explain the spirit of our results.} cannot play the  role of a variational space as above, as the dual is either not handy or too big. Thus, we have to renounce to \textit{a priori} boundedness in $\L^2$. Also,  for symmetry reasons it is  easier to let $\I=\R$ as this avoids boundary conditions for the time derivative. If one looks for a variational space candidate, the space $\L^2(\R; \H^1)$ is unavoidable for weak solutions since it is mapped to its dual by the leading terms.  Another Hilbert space,  $\H^{1/2}(\R; \L^2)$,  is mapped to its dual by the time derivative. This space already appeared in the theory \cite{LSU, Lions} but rather in the regularity theory  than with an instrumental role. 
Hence, the space $$\cV = \L^2(\R; \H^1)\cap \H^{1/2}(\R; \L^2),$$ or its homogeneous version $\cVdot$, is a natural candidate and we are going to  assume from the start  the alluded $\cVdot \to \cVdot'$ invertibility of the parabolic operator. Even though $\H^{1/2}(\R) \subset \L^\infty(\R)$ fails,  the homogeneous versions of these two spaces have the same scale invariance and therefore the homogeneous versions of the spaces $\cV$ and  $\L^2(\R; \H^1)\cap \L^\infty(\R; \L^2)$  have  the same embeddings into the mixed spaces $\L^r(\R;\L^q)$ except for the endpoint exponents $r=\infty, q=2$. Regularity theory  based on improvements of Lions's embedding theorem  allows us to introduce  a class of solutions where one can uniquely solve
$$\partial_{t}u+\Lt u=\delta_{s}\otimes \psi$$ for  $\psi\in \L^2$ and $\delta_{s}$ the Dirac mass at $s$,  and show that such solutions are continuously $\L^2$-valued  except at $s$. This turns out to be precisely what is needed to define Green operators with the expected properties that can be used to represent solutions of the Cauchy problem.  All boils down to proving invertibility of the parabolic operator, which uses an idea going back to \cite{Kaplan} that has been rediscovered several times since. There, Kaplan showed for the first time, {in absence of lower order coefficients},
where coercivity of the parabolic operator $\partial_{t}+\Lt$ hides  even though $\partial_{t}$ alone is not coercive in any sense since $\Re \int_{\R}{\partial_{t}u} \, {\overline u}\, \d t=0$ holds for any reasonable function $u$.

In summary, solutions being in $\L^\infty(\L^2)$ is an \textit{a priori} requirement  in most references to develop the theory and that the solutions belong to $\C(\L^2)$ and $\H^{1/2}(\L^2)$ is an \textit{a posteriori}  gain. Here, we use the invertibility on a space involving $\H^{1/2}(\L^2)$ to construct (unique) solutions that  are proved to be $\C(\L^2)\cap\L^\infty(\L^2)$  by a regularity argument, hence that are usual weak solutions in the end.

Let us next describe the new findings that emerge from these conceptual changes.

\medskip

\paragraph{\itshape Weaker assumptions on the coefficients} An advantage of using the variational space $\mathcal V$, as opposed to classical energy spaces, is that it  not only embeds into mixed Lebesgue spaces  $\L^r(\R;\L^q)$ for pairs $(r,q)$ of exponents that we call admissible,  but also into mixed Lorentz spaces $\L^{r,2}(\R; \L^{q,2})$. As a consequence, this allows us to relax assumptions from $\L^{\tilde r}(\R; \L^{\tilde q})$ to $\L^{\tilde r,\infty}(\R; \L^{\tilde q,\infty})$ for the lower order coefficients for pairs $(\tilde r, \tilde q)$ that we call compatible, as far as invertibility is concerned. Also causality can be proved under a weaker assumption, namely $\L^{\tilde r}(\R; \L^{\tilde q,\infty})$. This is explained in Section~\ref{sec:Lorentz}.

\medskip

\paragraph{\itshape Adaptability of the approach} 
%Our Hilbert space approach of parabolic equations is both surprising and interesting. 
The   ``hidden coercivity''  using the space  $\mathcal V$ discovered in \cite{Kaplan} had been explicitly appeared in several instances for other questions~\cite{HL, ABES,  AEN, AE, DZ} concerning local regularity, maximal regularity or boundary value problems. {The heart of the matter are Sections~\ref{sec:preliminaries} -  \ref{sec:puresecondorder}. Once the framework is set up correctly, numerous, otherwise non-trivial extensions, will come effortlessly:  Lower order coefficients in Lorentz spaces (Section~\ref{sec:Lorentz}), unbounded leading coefficients in $\BMO$ (Section~\ref{sec:BMO}), higher order systems on full space with integrability varying over the coefficients (Section~\ref{sec:HO}), second order equations and systems with lateral boundary conditions (Section~\ref{sec:lateral}). We provide full details for the first two extensions and restrict ourselves to sketching the strategies for the latter two as the article is already quite long.}

\medskip

\paragraph{\itshape A self-contained theory with simpler proof techniques  and improvement of Lions' embedding theorem} Many results we prove here could seem ``well-known'' to experts at first glance but we produce all details of the second order case in full space in order to show that the method we develop is self-contained with no recourse to older literature. Some arguments require new techniques of proof, hopefully simpler and without using Steklov averages. In particular, our approach is a consequence of  $\L^2$ continuity in time  (up to constant) of solutions of  the heat operator $\partial_{t}u-\Delta u=f$ or its adjoint when $u$  \textit{a priori} belongs to $\L^2(\R;\H^1)$ or its homogeneous version and $f$ belongs to sums of  mixed Sobolev spaces of $\L^2$ type with negative indices. This seems new.  In the end, this yields an improvement of  Lions' embedding theorem (Lemma~\ref{lem:maxreg}).  

\medskip

\paragraph {\itshape A universal construction without approximation of coefficients}

Lastly, our construction of propagators or fundamental solution operators avoids density arguments from operators with smooth coefficients or Galerkin methods.  Uniqueness implies that our construction agrees with others under common hypotheses. {In this sense it is universal and also constructive}.  In particular, we obtain a new proof of the Gaussian upper bound of Aronson {as a consequence of  $\L^2$ off-diagonal estimates for fundamental solution operators, which hold in full generality (see Theorem~\ref{thm:ODE}). The latter is a new result in its own right.}  

\medskip

Further details and precise assumptions are given in the course of the article.

\section{Second order problems on full space}
\label{sec: full space}

In what follows, we use $\L^2$  spaces in both $\R^n$, $n\ge 1$, and $\R^{n+1 }$, equipped with Lebesgue measures. We denote by $\angle \psi {\tilde\psi}$ the complex inner product in the variable $x \in \R^n$ and by $\Angle \phi \tphi $ the complex inner product in the variables $(t,x) \in \R\times \R^n \eqqcolon \R^{n+1}$. (We prefer this order for the variables for practical reasons.)  For a function $f$ of the two variables, we set  $f(t): x\mapsto f(t,x)$ for any  $t$. 

We use the notation $\cD$ and $\cD'$ for the spaces of $\C^\infty$ functions with compact support and of distributions, respectively, and  $\cS$ and $\cS'$ for the spaces of Schwartz functions and tempered distributions, respectively.  Variables will be indicated at the time of use. Duality brackets extend inner products on $\L^2$ spaces, hence they are sesquilinear.

\subsection{Variational space}
\label{sec:variationalspace}

As said in the introduction, we need a space, which can be thought of as $\L^2_{t}\Hdot^{1}_{x}\cap \Hdot^{1/2}_{t}\L^2_{x}$. However, some care must be taken because we use homogeneous norms.  

Let the Fourier transform on $\R^{n+1}$ be the usual extension to tempered distributions  of the integral defined on $\L^1$ functions by
$$
\hat \varphi(\tau,\xi)= \iint_{\R^{n+1}} \varphi(t,x)\e^{-i(t\tau+x\cdot \xi)}\, \d t\d x, \qquad (\tau,\xi)\in \R\times \R^n.
$$ 
Remark that $(\tau,\xi)\mapsto (|\xi|^2+|\tau|)^{-1/2}$ is locally square integrable in $\R\times \R^{n}$ with $$\iint_{|\xi|^2+|\tau|\le R^2} (|\xi|^2+|\tau|)^{-1}\, \d\tau\d \xi =C(n) R^n.$$ Thus, for any $g\in \LL$ we have that $(|\xi|^2+|\tau|)^{-1/2}g$ is in both $\Lloc^1(\R^{n+1})$ and $\cS'(\R^{n+1})$. We define
$$\cVdot \coloneqq \Big\{u\in  \cS'(\R^{n+1})\, : \, \hat u= (|\xi|^2+|\tau|)^{-1/2}g \text{\ for\ some\ (unique)}\ g\in \LL \Big\}.
$$
Equipped with the norm $$\|u\|_{\cVdot} \coloneqq (2\pi)^{-(n+1)/2} \|g\|_{\LL},$$ this is a Hilbert space.  It is easy to see that it contains $\cS(\R^{n+1})$ as a dense subspace. {Note also that constants do not belong to $\cVdot$}.

\begin{rem} For $\varphi\in \cS(\R^{n+1})$, we have, using Plancherel's formula,
$$
\Big(\|\nabla \varphi\|_{\LL}^2+ \|\D_{t}^{1/2}\!\varphi\|^2_{\LL}\Big)^{1/2}= \|\varphi\|_{\cVdot}.
$$
Here, $\D_{t}^\alpha$ is the Fourier multiplier with symbol $|\tau|^\alpha$.
In fact, the closure of $\cS(\R^{n+1})$ for the norm  defined by the left-hand side is $\cVdot+\IC$, seen as a subspace of $\cS'(\R^{n+1})/\IC$. For a proof see Lemma~3.11 in \cite{AEN}; this closure is denoted by $\Edot(\R^{n+1})$ there. Hence $\cVdot$ is nothing but the realization of this closure within $\cS'(\R^{n+1})$ that eliminates constants.  In  particular, whenever $u\in \cVdot$, then $\nabla u$ and $\D_{t}^{1/2}u$ exist as tempered distributions, belong to $\LL$, and the identity above holds.
\end{rem}
 
We let $\cVdot'$ be the dual of $\cVdot$ with respect to $\Angle{\cdot}{\cdot}$. Thus, it is a subspace of $\cS'(\R^{n+1})$ and a distribution $w \in \cS'(\R^{n+1})$ belongs to $\cVdot'$ if and only if $(|\tau|+ |\xi|^2)^{-1/2}\hat w\in \LL$.  It {follows from Plancherel's formula} that  $w\in \cVdot'$ if and only if  there exists a decomposition $\hat w= |\xi| g_{1}+|\tau|^{1/2}g_{2}$ with $g_{1},g_{2}\in \LL$ and that $\|w\|_{\cVdot'}\sim \inf(\|g_{1}\|_{\LL}+\|g_{2}\|_{\LL})$ taken over all such decompositions. In {this sense we write} $\cVdot'=\L^2_{t}\Hdot^{-1}_{x}+ \Hdot^{-1/2}_{t}\L^2_{x}$ with equivalent norms.

\subsection{Embeddings}
\label{sec:preliminaries}

 Recall that the homogeneous Sobolev space $\Hdot^{1/2}(\R)=\Hdot^{1/2}_{t}$, that is, the closure of $\cS(\R)$ for the norm $\|\D_{t}^{1/2}\!\varphi\|_{2}$, has the same scaling properties as $\L^\infty(\R)$. This results in continuous inclusions into mixed normed Lebesgue spaces for $\cVdot$ that, except for endpoints, are {the same as} for $\L^2_{t}\Hdot^1_{x}\cap \L^\infty_{t}\L^2_{x}$. We describe them next. We mention that {$\Hdot^{1/2}_{t}$ has an equivalent (semi-)norm, using difference quotients that is often used in  the literature on this topic, see for instance \cite{LSU}. We  do not use them here.}

We need the following mixed spaces. For pairs $(r,q)\in [1,\infty]^2$ of exponents, intervals $\I\subset \R$, and open sets $\Omega\subset \R^n$, $n\ge 1$,  we write $\L^r(\I; \L^q(\Omega))$ for the mixed norm space of measurable functions $u:\I\times \Omega\to \IC$ with
$$
\|u\|_{\L^r(\I; \L^q(\Omega))} \coloneqq \bigg(\int_{\I}\Big(\int_{\Omega}|u(t,x)|^q\, \d x\Big)^{r/q}\, \d t\bigg)^{1/r} <\infty$$
and the usual changes if either $r=\infty$ or $q=\infty$. We set  $\L^{r}_{t}\L^{q}_{\vphantom{t} x}\coloneqq \L^r(\R; \L^q(\R^n))$ with dummy variables in indices when $t\in \R$ and $x\in \R^n$. 

We introduce the Banach space of tempered distributions
$$
\Deldot^{r, q}: = \Big\{u\in \L^{r}_{t}\L^{q}_{\vphantom{t} x}\,  : \, \nabla u\in \LL \Big\} \ (\eqqcolon \L^2_{t}\Hdot^{1}_{x} \cap \L^{r}_{t}\L^{q}_{\vphantom{t} x}),
$$
where the gradient is taken in the sense of distributions,\footnote{From now on, we choose not to indicate the target vector space in the notation.} with norm 
$$
\|u\|_{\Deldot^{r, q}}: =\|\nabla u \|_{\LL}+ \|u\|_{\L^{r}_{t}\L^{q}_{\vphantom{t} x}}.
$$
{Duality theory for $\Deldot^{r,q}$ is easily understood by identifying $\Deldot^{r,q}$ with a closed subspace of $\L^{r}_{t}\L^{q}_{\vphantom{t} x} \times \L^{2}_{t}\L^{2}_{x}$ through the map $u\mapsto (u, \nabla u)$. As usual, the  H\"older conjugate  of $q\in [1,\infty]$ is $q'=\frac q {q-1}$. Thus, for the duality $\Angle{\cdot\,}{\cdot}$ we have that if $(r,q)\in [1,\infty)^2$, then $(\Deldot^{r, q})' = \L^2_{t}\Hdot^{-1}_{x} + \L^{r'}_{t}\L^{q'}_{\vphantom{t} x}$, the space of elements $\div F + g$, with vector field $F\in \LL$ and scalar function $g\in \L^{r'}_{t}\L^{q'}_{\vphantom{t} x}$, equipped with usual infimum norm. In the same manner,} when $(r,q)\in (1,\infty]^2$,  then $\Deldot^{r, q}$ is identified with the dual space of $\L^2_{t}\Hdot^{-1}_{x} + \L^{r'}_{t}\L^{q'}_{\vphantom{t} x}$ and when $(r,q)\in (1,\infty)^2$, then it is reflexive. {The duality theory implies that $\cS(\R^{n+1})$ is dense in $\Deldot^{r,q}$ when $(r,q) \in [1,\infty)^2$. }

\begin{defn} 
\label{def:admissible} 
A pair $(r,q)$ is admissible if   $\frac 1 {r}+\frac {n}{2q}= \frac n4$ with $2\le r,q<\infty$.
\end{defn}

\begin{lem} 
\label{ref:embedding}
 If $(r,q)$ is an admissible pair, then $\cVdot \hookrightarrow \Deldot^{r, q}$ {with continuous inclusion.}
In particular, elements of $\cVdot$ are locally integrable functions. By duality, this yields  the continuous inclusion $ (\Deldot^{r, q})' \hookrightarrow \cVdot'$ 
\end{lem}

\begin{proof} 
By density it suffices to work with  $\varphi \in \cS(\R^{n+1})$ {and show the first inclusion}.  The proof  relies on two ingredients. First, for $\theta\in [0,1]$, using the convexity inequality
 $$
 |\tau|^{\theta}|\xi|^{2(1-\theta)} \le \theta |\tau|+ (1-\theta) |\xi|^2,
 $$
Fourier transform in the $(t,x)$-variable shows that 
\begin{equation}
\label{eq:convexity}
\|\D_{t}^{\theta/2}(-\Delta)^{(1-\theta) /2} \varphi\|_{\LL}^2
\le \theta \|\D_{t}^{1/2} \!\varphi\|_{\LL}^2 + (1-\theta)\|(-\Delta)^{1 /2} \varphi\|_{\LL}^2\le  \|\varphi\|_{\cVdot}^2.
\end{equation}
Next, Sobolev embeddings in $\R^n$ and $\R$ give us 
\begin{equation}
\label{eq:embed}
\|\varphi\|_{\L^{r}_{t}\L^{q}_{\vphantom{t} x}}  \le c(n,q) \|(-\Delta)^{(1-\theta) /2} \varphi\|_{\L^{r}_{t}\L^{2}_{\vphantom{t} x}}
\le c(n,q) c(1,r)\|\D_{t}^{\theta/2}(-\Delta)^{(1-\theta) /2} \varphi\|_{\LL},
\end{equation}
where the first inequality holds exactly when $\frac 1 2 - \frac{1-\theta}{n} =\frac 1 q$ and $2\le q<\infty$, and the second one exactly when $\frac \theta 2- \frac 1 2= -\frac 1 r$ and $2\le r<\infty$. Note that the second embedding is the vector-valued extension of the scalar embedding on $\R$. 

We can solve for $\theta\in [0,1)$ if  $\frac 1 {r}+\frac {n}{2q}= \frac n4$ with $2\le r<\infty$ and $2\le q <\infty$, which is  the definition of an admissible pair. 
Since the other part   $\|\nabla \varphi\|_{\LL}$ of the $\Deldot^{r, q}$ norm  is controlled by $\|\varphi\|_{\cVdot}$, we are done. 
\end{proof}

\begin{rem}
 As $\Hdot^{1/2}(\R)$ does not embed into $\L^\infty(\R)$, the result fails for $(r,q)=(\infty,2)$ and we  excluded this pair  from our admissible range  although it satisfies the relation $\frac 1 {r}+\frac {n}{2q}= \frac n4$. Similarly, the embedding $\Hdot^{1-\theta}_{x}\subset \L^q_{x}$ never holds when $q=\infty$ and requires  $1-\theta < \frac n 2$. 
\end{rem}

\subsection{Variational approach}
 \label{sec:homogeneousparabolicsystems}
 
We study parabolic equations on $\R^{n+1} = \R\times \R^n$, namely
$$\partial_{t}u+\Lt u=f$$ 
and its adjoint equation 
$$-\partial_{t}\tilde u+\Ltstar   \tilde u=\tilde f,$$  
where $\Lt $ and its adjoint are second order elliptic operators in divergence form perturbed with unbounded lower order terms. The equalities are taken in the sense of distributions, provided $\Lt u$ and $\Ltstar   \tilde u$ are well-defined.  To be precise, we consider
\begin{equation}
\label{eq:Lt}
\Lt u= -\div (A \nabla u + \oa\,   u) + \ob  \cdot \nabla u + a u,
\end{equation}
where the coefficients $A,\oa  , \ob  , a$ depend on $(t,x)$. Sometimes, we consider $\Lt$ as an operator acting on functions of the $x$-variable by freezing $t$: context will make things clear. The leading coefficient $A=(a_{ij})$ is an $n\times n$ matrix of  bounded, possibly complex-valued,  measurable functions on $\R^{n+1}$. Thus, the sesquilinear form corresponding to the leading part in \eqref{eq:Lt} satisfies
\begin{equation}
\label{eq:Abdd}
\int_{\R} |\angle{A(t)\nabla u(t)}{\nabla v(t)}|\, \d t \le \Lambda  \|\nabla u\|_{\LL}\|\nabla v\|_{\LL}
\end{equation}
with $\Lambda\coloneqq \|A\|_{\infty}$. The lower order coefficients $\oa  , \ob   $ are $n$-vectors of  complex-valued, measurable functions on $\R^{n+1}$,  and $a$ is a complex-valued, measurable functions on $\R^{n+1}$. The {formal} complex adjoint of $\Lt$ corresponds to 
\begin{equation}
\label{eq:Ltstar}
\Ltstar u= -\div (A^* \nabla u + \overline\ob\,   u) + \overline\oa  \cdot \nabla u + \overline a u.
\end{equation}

We introduce the following quantity. 

\begin{defn}
\label{def:Pr1q1}
For  a  pair  $(\tilde r_{1},\tilde q_{1})\in [1,\infty]^2$ let
\begin{align}
	\label{eq:P}
	P_{\tilde r_{1},\tilde q_{1}} \coloneqq \||\oa  |^2\|_{\L^{\tilde r_{1}}_{t}\L^{\tilde q_{1}}_{\vphantom{t} x}}^{1/2}+ \||\ob  |^2\|_{\L^{\tilde r_{1}}_{t}\L^{\tilde q_{1}}_{\vphantom{t} x}}^{1/2}+ \|a\|_{\L^{\tilde r_{1}}_{t}\L^{\tilde q_{1}}_{\vphantom{t} x}}
\end{align}
and define  $(r_{1}, q_{1})\in [2,\infty]^2$ through the relations
\begin{equation}
	\label{eq:associatedpair}
	r_{1}\coloneqq 2(\tilde r_{1})' \quad \& \quad q_{1}\coloneqq 2 (\tilde q_{1})'.
\end{equation}
\end{defn}

\begin{rem}
In principle, we could let $(\tilde r_{1},\tilde q_{1})$ be different for each entry of $\oa  $ and $\ob  $, and $a$. At this point we do not go into this  in order to simplify the exposition of ideas but we shall come back to the general version later on in Section~\ref{sec:HO}.
\end{rem}

Next, we introduce the sesquilinear pairings corresponding to the lower order terms in \eqref{eq:Lt}. We set 
$$
\beta u \coloneqq  -\div (\oa\,   u) + \ob  \cdot \nabla u + a u,
$$
so that {for appropriate $u,v$} and
for (almost) each $t \in \R$, we write 
\begin{align}
\label{eq:beta}
	\angle {\beta u(t)}{v(t)}= \angle{\oa(t) u(t)}{  \nabla v(t)} + \angle{\ob(t)\cdot \nabla u(t)}{  v(t)} + \angle{a(t) u(t)}{ v(t)}
\end{align}
and 
\begin{align}
\label{eq:Beta}
	\Angle {\beta u}{v}= \int_{\R} \angle{\beta u(t)}{v(t)}\, \d t.
\end{align}
The relation \eqref{eq:associatedpair} guarantees that the formal pairings above are absolutely convergent Lebesgue integrals, as becomes apparent from the next lemma.

\begin{lem}
\label{lem:beta} Let $(\tilde r_{1}, \tilde q_{1}) \in [1,\infty]^2$ and let $(r_{1}, q_{1})$ given by \eqref{eq:associatedpair}. {Suppose that $P_{\tilde r_{1},\tilde q_{1}}$ is finite.} If $u, v\in \Deldot^{r_{1},  q_{1}}$, then  
\begin{align*}
|\angle {\beta u(t)}{v(t)}| \leq 
& \||\oa(t)|^2\|_{\L^{\tilde q_{1}}_{x}}^{1/2}\|u(t)\|_{\L^{q_{1}}_{x}}\|\nabla v(t)\|_{\L^2_{x}} \\
&+  \||\ob(t)|^2\|_{\L^{\tilde q_{1}}_{x}}^{1/2}\|v(t)\|_{\L^{ q_{1}}_{x}}\|\nabla u(t)\|_{\L^2_{x}} \\
&+ \|a(t)\|_{\L^{\tilde q_{1}}_{x}}\|u(t)\|_{\L^{ q_{1}}_{x}}^2.
\end{align*}
In particular, $\angle {\beta u(t)}{v(t)} \in \L^1_{t}$ and
$$
|\Angle {\beta u}{v}| \le P_{\tilde r_{1},\tilde q_{1}} \|u\|_{\Deldot^{ r_{1}, q_{1}}}\|v\|_{\Deldot^{r_{1}, q_{1}}}.
$$
\end{lem}

\begin{proof}
Use H\"older inequalities in the $x$ and then $t$-variables, taking into account the  relations  
\begin{align*}
	\frac 1 {2\tilde r_{1}}+ \frac 1 {r_{1}}+ \frac 1 {2} = 1, &\quad 
	\frac 1 {2\tilde q_{1}}+ \frac 1 {q_{1}}+ \frac 1 {2}= 1,\\
	\frac 1 {\tilde r_{1}}+ \frac 2 { r_{1}}= 1, &\quad
	\frac 1 {\tilde q_{1}}+ \frac 2 {q_{1}}= 1. \qedhere
\end{align*}
\end{proof}

It is of course natural to relate the choice of pairs to Sobolev embeddings.

\begin{defn}
\label{def:compatiblepair}
A pair $(\tilde r_{1}, \tilde q_{1})$ is said  {compatible} for lower order coefficients if  $\frac 1 { \tilde r_{1}}+\frac {n}{2 \tilde q_{1}}= 1$ and $1< \tilde r_{1}, \tilde q_{1} \le \infty$. {In this case, $(r_{1},q_{1})$ given by \eqref{eq:associatedpair} is its {admissible conjugate} pair.}
\end{defn}
 
This terminology {is motivated by} the following principle.

\begin{lem}
\label{lem:compadm} A pair $(\tilde r_{1}, \tilde q_{1})$ is compatible for lower order coefficients if and only if $(r_{1},q_{1})$ is admissible.
\end{lem}

\begin{proof}
We can see that 
\begin{align*}
 \frac 1 {\tilde r_{1}}+\frac {n}{2\tilde q_{1}}= 1 \Longleftrightarrow   
 \frac 1 {r_{1}}+\frac {n}{2q_{1}} = \frac n 4, \ \ \\
 1<\tilde r_{1}, \tilde q_{1} \le \infty \Longleftrightarrow 2\le  r_{1}, q_{1}<\infty. &\qedhere
\end{align*}
\end{proof}

\begin{rem} The compatibility and admissibility conditions already appear in  Chapter 3 of \cite{LSU}. As in there (see p.~137), we include the case $\tilde r_{1}=\infty$, $\tilde q_{1}=\frac n 2$, $n\ge 3$,  but not the case $\tilde r_{1}=1$ as the variational space is not contained in  $ \L^\infty_{t}\L^2_{x}$.  
\end{rem}
  
We now introduce the  variational setup. We use the Hilbert space $\cVdot$ and its dual $\cVdot'$ for $\Angle{\cdot}{\cdot}$. Since $\cS$ is dense in $\cVdot$, this pairing is consistent with the sesquilinear pairing of tempered distributions and Schwartz functions. We have seen in Section~\ref{sec:preliminaries} that $\cVdot \subset \Deldot^{r,q} $ and $\L^{r'}_{t}\L^{q'}_{\vphantom{t} x} \subset \cVdot'$ if $(r,q)$ is admissible. For $u \in \cVdot$ and $v \in \L^{r'}_{t}\L^{q'}_{\vphantom{t} x}$ the pairing $\Angle{u}{v}$ is therefore the Lebesgue integral $\int_{\R^{n+1}}u(t,x)\overline v(t,x)\, \d x\d t$. This observation will be tacitly used throughout the section. 

\begin{prop}
\label{prop:Lt} 
Assume that $P_{\tilde r_{1}, \tilde q_{1}}<\infty$ for some pair $(\tilde r_{1}, \tilde q_{1})$  compatible for lower order coefficients.
Define the operator
\begin{align*}
\cH\colon \cVdot \to \cVdot',  \quad \cH u= \partial_{t}u+\Lt u
\end{align*}
through
\begin{align}
\label{eq:defcH}
\Angle {\cH u}v =  \Angle {\partial_{t}u} {v}  +
	\int_{\R} \angle{A(t)\nabla u(t)}{\nabla v(t)}+ \angle{\beta u(t)}{v(t)}\, \d t
\end{align}
for $u,v \in \cVdot$. In the same fashion, define the dual operator
\begin{align*}
	\cH^*\colon \cVdot \to \cVdot', \quad \cH^* \tilde u = -\partial_{t}\tilde u+\Ltstar   \tilde u.
\end{align*}
Then $\cH, \cH^*: \cVdot \to \cVdot'$ are well-defined, bounded and adjoint to one another.
\end{prop}

\begin{proof}  For $u\in \cVdot$ we have  $(\partial_{t}u) \, \hat{}= |\tau|^{1/2} (i\tau |\tau|^{-1/2}\hat u)$ and $i\tau |\tau|^{-1/2}\hat u\in \LL$, so that $\partial_{t}u\in \cVdot'$ with $\|\partial_{t}u\|_{\cVdot'}\le \|u\|_{\cVdot}$. It follows that $\partial_{t}$, seen as an operator acting on tempered distributions in the two variables $(t,x)$, maps $\cVdot$ into $\cVdot'$.  

Next, for the admissible conjugate pair $(r_{1},q_{1})$ and  $u,v\in \Deldot^{ r_{1},  q_{1}}$, the pairing $\Angle {\Lt u} {v}$ is defined as 
\begin{equation}
\label{eq:defLt}
\Angle {\Lt u} {v} = \int_{\R} \angle{A(t)\nabla u(t)}{\nabla v(t)}+ \angle{\beta u(t)}{v(t)}\, \d t,
\end{equation}
so that by Lemma~\ref{lem:beta},  
\begin{align*}
 |\Angle {\Lt u} {v}| \le \|A\|_{\infty} \|\nabla u\|_{\LL}\|\nabla v\|_{\LL} + P_{ \tilde r_{1}, \tilde q_{1}} \|u\|_{\Deldot^{ r_{1}, q_{1}}}\|v\|_{\Deldot^{ r_{1}, q_{1}}}.
\end{align*}
By the embedding $\cVdot\hookrightarrow  \Deldot^{ r_{1}, q_{1}}$  for the admissible pair $(r_{1},q_{1})$,  we conclude that $\Angle {\Lt u} {v}$ is defined on $\cVdot\times \cVdot$ and  that with $C= C(n, r_{1},q_{1})$ we have
\begin{align*}
|\Angle {\Lt u} {v}| \le ( \|A\|_{\infty}+ C P_{ \tilde r_{1}, \tilde q_{1}}) \|u\|_{\cVdot}\|v\|_{\cVdot}. 
\end{align*}
Eventually, for $u,v\in \cVdot$, it follows using Fourier transform that
\begin{equation*}
\Angle {\partial_{t}u}v= -\overline{\Angle  {\partial_{t}v} u}
\end{equation*}
and by inspection that 
\begin{equation*}
\Angle {\Lt u} {v}= \overline{\Angle {\Ltstar   v} u} .
\end{equation*}
Hence, $\cH^*$ is the adjoint of $\cH$ and its boundedness follows.
\end{proof}

\begin{rem}
\label{rem:Delta} We shall often write $\Angle {\Lt u} {v}= \int_{\R} \angle {\Lt u(t)} {v(t)}\, \d t$ to mean \eqref{eq:defLt}. When $u,v\in \Deldot^{ r_{1}, q_{1}}$, we have also shown integrability in $t$ together with  the intermediate estimate
 $$|\Angle {\Lt u} {v}|
  \le (\|A\|_{\infty}  + P_{ \tilde r_{1}, \tilde q_{1}}) \|u\|_{\Deldot^{ r_{1},  q_{1}}}\|v\|_{\Deldot^{ r_{1}, q_{1}}}.
  $$
Hence, $\Lt u\in (\Deldot^{r_{1}, q_{1}})'$ with $\|\Lt u\|_{(\Deldot^{r_{1}, q_{1}})'}\le (\|A\|_{\infty}  + P_{ \tilde r_{1}, \tilde q_{1}}) \|u\|_{\Deldot^{ r_{1},  q_{1}}}$. This implies in particular that $\Lt u$ is defined as a distribution when $u\in \Deldot^{ r_{1}, q_{1}}$, and so is $\partial_{t}u + \Lt u$.
\end{rem}

This remark suggests the following notion of solution. {We try to be very explicit in this regard} in order not to confuse the reader by the versatile terminology of weak solution,
and also because we work on $\R^{n+1}$. 

\begin{defn}
\label{def:Delta-sol}
We say that $u$ is a $\Deldot^ {r_{1},q_{1}}$-solution of 
$ \partial_{t}u+\Lt u= f$ in $\R^{n+1}$ if $u\in \Deldot^{r_{1},q_{1}}$ and the equation is satisfied in the sense of distributions on $\R^{n+1}$, that is  for all $\tphi\in \cD(\R^{n+1})$,
\begin{align}
\label{eq:Deltasol}
-\Angle {u} {\partial_{t}\tphi}  +
\int_{\R} \angle{A(t)\nabla u(t)}{\nabla \tphi(t)}+ \angle{\beta u(t)}{\tphi(t)}\, \d t= \Angle {f}{\tphi},
\end{align}
where $\Angle {u} {\partial_{t}\tphi}= \iint_{\R^{n+1}} u(t,x)  \overline{{\partial_{t}\tphi}(t,x)} \, \d x\d t.$
\end{defn}
%Recall that the first term can be expressed as  $-  \iint_{\R^{n+1}} u(t,x)  \overline{{\partial_{t}\tphi}(t,x)} \, \d x\d t.$ 
There is no regularity in time attached to this definition.  For appropriate right hand side, we shall show that in fact a $\Deldot^{r_{1},q_{1}}$-solution is (up to a constant)  continuous  and bounded in time valued in $\L^2_{x}$, so that in the end we will be able to identify with weak solutions when considering the Cauchy problem, see Section  \ref{sec:CP}.

%Weak solutions for the Cauchy problem on $[0,T]\times \R^n$ are those  %functions in the class $\L^2(0,T; \H^1(\R^n)) \cap \L^{\infty}(0,T; \L^{2}(\R^n))$, 
%which embeds  into $ \L^{r_{1}}(0,T; \L^{\vphantom{r_{1}} q_{1}}(\R^n))$, 
%see Proposition~\ref{prop:GN} for a quick proof. We shall show that in %fact a $\Deldot^{r_{1},q_{1}}$-solution is continuous in time valued in $\L^2_{x}$, 
%so that in the end we will be able to identify solutions from both methods. 

\subsection{Main regularity estimates}
\label{sec:mainregularityestimates}

Our approach builds on the results of the next two sections.  Note that the assumptions have a homogeneous flavor, which is necessary when the time interval is infinite.

We begin with results providing existence and uniqueness of specific solutions for the heat operator $\partial_{t}-\Delta$ in $\R^{n+1}$. (We could also use $\partial_{t}+\Delta$ as the choice of forward or backward time is irrelevant when $t\in \R$.) This relies on Fourier transform arguments with tempered distributions.

\begin{lem}
\label{lem:uniqueness} Let $u\in \cD'(\R^{n+1})$ be a solution of $\partial_{t}u - \Delta u=0$ in $\cD'(\R^{n+1})$ with $\nabla u\in \LL$.  Then $u$ is constant.  
 
\end{lem}

\begin{proof}  
For each $j\in \{1,\ldots, n\}$ we have $u_{j} \coloneqq \partial_{x_{j}}u\in \LL\subset  \cS'(\R^{n+1})$ and $\partial_{t}u_{j}-\Delta u_{j}=0$ in $\cD'(\R^{n+1})$, hence also    in $\cS'(\R^{n+1})$. This implies that  $\Angle{u_{j}}{\partial_{t}\phi+\Delta\phi}=0$ for all  $\phi \in  
 \cS_{0}(\R^{n+1})$, the space of Schwartz functions whose Fourier transforms vanish to infinite order at 0. Now, $\partial_{t}+\Delta$, whose  Fourier symbol is the polynomial $i\tau-|\xi|^2$, is an automorphism of $ 
 \cS_{0}(\R^{n+1})$. We obtain $\Angle{u_{j}}{\phi}=0$ for all  $\phi \in  
 \cS_{0}(\R^{n+1})$, so that  $u_{j}$ is a polynomial. As $u_{j}\in \LL$, this polynomial vanishes. We have shown that $\nabla u=0$, hence $\Delta u=0$, and $\partial_{t}u=0$ follows from the equation. We conclude that $u$ is constant.
 \end{proof}

\begin{prop}
\label{prop:maxregDeldot} Let $(r,q)$ be an admissible pair and $w\in \L^2_{t}\Hdot^{-1}_{x}+ \L^{r'}_{t}\L^{q'}_{\vphantom{t} x}$. Then there exists   $v\in \C_{0}(\L^2_{x})\cap  \cVdot$, unique in the {class of distributions} with $\nabla v\in \LL$, solution of $\partial_{t}v-\Delta v=w$ 
in {$\cD'(\R^{n+1})$}. Moreover, 
$$
\sup_{t\in \R} \|v(t)\|_{\L^2_{x}}+ \|v\|_{\cVdot} \le c(n,q,r)\|w\|_{\L^2_{t}\Hdot^{-1}_{x}+ \L^{r'}_{t}\L^{q'}_{\vphantom{t} x}}.
$$
\end{prop}

Uniqueness is provided by Lemma~\ref{lem:uniqueness}. The main work is to produce this solution.  Using $\R$ as the time interval  allows us to use embeddings for homogeneous Sobolev space on $\R$. For $\theta\in [0,1)$, we introduce the space $$\Htheta \coloneqq \Big\{\D_{t}^{\theta/2}(-\Delta)^{(1-\theta) /2} g \, : \,  g\in \LL\Big\}$$ equipped with the norm  $\|w\|_{\Htheta} \coloneqq \|g\|_{\LL}$. We recall that $\D_{t}^\alpha$ is the Fourier multiplier with symbol $|\tau|^\alpha$. For $\theta=0$, we find 
$\L^2_{t}\Hdot^{-1}_{x}$. Elements in $\Htheta$ are tempered distributions with locally square integrable Fourier transforms.
    
\begin{lem}
\label{lem:H-theta} 
If $(r,q)$ is an admissible pair, then
$\L^{r'}_{t}\L^{q'}_{\vphantom{t} x} \hookrightarrow \Htheta$ for $\theta=1-\frac 2 r$.
\end{lem}

\begin{proof}   
Set
\begin{align*}
	M_{\theta} \coloneqq \D_{t}^{\theta/2}(-\Delta)^{(1-\theta) /2} \quad \text{and} \quad G\coloneqq\{M_{\theta}\varphi\, : \, \varphi\in  \cS(\R^{n+1})\}.
\end{align*} 
As $\cS(\R^{n+1})$ is dense in $\LL$,  we see that $G$ is a dense subspace of $\Htheta$. For any $v\in \L^{r'}_{t}\L^{q'}_{\vphantom{t} x}$ and $g\in G$, we get using \eqref{eq:embed},
$$
|\Angle v {M_{\theta}^{-1}g}| \le \|v\|_{\L^{r'}_{t}\L^{q'}_{\vphantom{t} x}}\|M_{\theta}^{-1}g\|_{\L^{r}_{t}\L^{q}_{\vphantom{t} x}} \le  c\|v\|_{\L^{r'}_{t}\L^{q'}_{\vphantom{t} x}}\|g\|_{\LL}.
$$
By density of $G$ and the Riesz representation theorem there exists a unique $w\in \LL$ with 
$\|w\|_{\LL}\le c\|v\|_{\L^{r'}_{t}\L^{q'}_{\vphantom{t} x}}$ such that 
$\Angle v {M_{\theta}^{-1}g}=\Angle w {g}$ for all $g\in G$,  and we can see that $v=M_{\theta}w$ in $\cS'(\R^{n+1})$. This concludes the proof. 
\end{proof}

Armed with this embedding, it suffices to prove the following stronger statement in purely $\L^2$-based mixed Sobolev spaces.

\begin{lem}
\label{lem:maxreg} 
Let $\theta\in [0,1)$ and  $w\in \Htheta$. Then there exists $v\in \C_{0}(\L^2_{x})\cap  \cVdot$, solution of $\partial_{t}v-\Delta v=w$ 
in the sense of tempered distributions in $\R^{n+1}$, with
$$
\sup_{t\in \R} \|v(t)\|_{\L^2_{x}}+ \|v\|_{\cVdot} \le c(\theta)\|w\|_{\Htheta}.
$$
\end{lem}

\begin{proof} Guided by the Duhamel formula, the function $v(t)$ is formally defined by 
\begin{equation}
\label{eq:v(t)}
v(t)\coloneqq\int_{{-\infty}}^t \e^{(t-s)\Delta}w(s)\, \d s
\end{equation}
and we need to make sense of this integral. To this end, we introduce a smaller dense subspace of $\Htheta$. 

{Let $M_\theta = \D_{t}^{\theta/2}(-\Delta)^{(1-\theta)/2}$ as before and} set $G_{0}\coloneqq\{M_{\theta}g\, : \, g\in  \cS_{00}(\R^{n+1})\}$, 
where $\cS_{00}(\R^{n+1})$ is the space of Schwartz functions whose Fourier transforms are supported away from  $\tau=0$ and $\xi=0$. 
Note that $M_{\theta}$ preserves  $\cS_{00}(\R^{n+1})$
{and that} $G_{0}$ 
is a dense subspace of $\Htheta$. Call $\tau_{t}$ the translation by $t\in \R$: $\tau_{t}g(s)\coloneqq g( s+t)$, not indicating the $x$-variable as usual. 
It commutes with $M_{\theta}.$ 

\medskip
  
 \paragraph{\itshape Step 1: $v(t) \in \L^2_{x}$ with $\|v(t)\|_{\L^2_{x}}\le c(\theta)\|w\|_{\Htheta}$}

We begin with a preliminary estimate. Let $\varphi\in \L^2_{x}$ and define
$$h(t,x) \coloneqq 1_{(-\infty,0)}(t)\,(\e^{-t\Delta}\varphi)(x),$$ where $1_{A}$ denotes the indicator function of $A$.  
A classical calculation shows that $\hat h(\tau,\xi)= (-i\tau+|\xi|^2)^{-1} \hat \varphi(\xi)$, where $\hat \varphi$ is the Fourier transform 
of $\varphi$ on $\R^n$.
Using Plancherel's formula in $\R^{n}$ and $\R$,  Fubini's theorem and the change of variables  $\tau=\sigma|\xi|^2$ when $\xi\ne 0$, we find
\begin{align*}
\|M_{\theta}(h)\|_{\LL}
&=  {{(2\pi)}^{-\frac{n+1}{2}}}  \bigg(\iint_{\R^{n+1}}\frac{ |\tau|^\theta |\xi|^{2-2\theta}}{|- i\tau+|\xi|^2|^{2}}| \hat\varphi(\xi)|^2 \, 
\d \xi\d \tau\bigg)^{1/2}
\\
&
= c(\theta)  \|\varphi\|_{\L^2_{x}}
\end{align*}
with $c(\theta) \coloneqq (2\pi)^{-1/2}\big(\int_{-\infty}^\infty \frac {|\sigma|^\theta}{1+\sigma^2}\, \d\sigma\big)^{1/2}<\infty$ since $\theta\in [0,1)$. 
It follows that for any $g\in \LL$,
\begin{align}
\label{eq:LinftyL2bound} 
\int^{\infty}_{-\infty} |\angle {\tau_{t}g(\sigma)}{M_{\theta} (h)(\sigma)}|\, \d \sigma  
\le c(\theta)\|\tau_{t}g\|_{\LL} \|\varphi\|_{\L^2_{x}}.
\end{align}

Now assume $g\in \cS_{00}(\R^{n+1})$, set $w \coloneqq M_{\theta}g$ and fix $t\in \R$.
Then $w\in \cS_{00}(\R^{n+1})$, so in particular $w \in \L^1_{t}\L^2_{x}$ and the integral in \eqref{eq:v(t)} converges as a Bochner integral in $\L^2_{x}$ thanks to the contractivity of the heat semigroup on $\L^2_{x}$. 
To get the appropriate estimate on $v(t)$, we calculate
\begin{align}
\label{eq:v(t)-duality estimate}
\begin{split}
\angle{v(t)}{\varphi}&= \int_{{-\infty}}^t \angle{\e^{(t-s)\Delta}w(s)}{\varphi}\, \d s
 =   \Angle{\tau_t w}{h}
 \\
 &= \Angle{\tau_t g}{M_{\theta}(h)}
 =\int^{\infty}_{-\infty} \angle {\tau_{t}g(\sigma)}{M_{\theta} (h)(\sigma)}\, \d \sigma
\end{split}
\end{align}
 and \eqref{eq:LinftyL2bound} yields   
 \begin{equation}
 \label{eq:cMbdd}
 \|v(t)\|_{\L^2_{x}}\le c(\theta)\|\tau_{t}g\|_{\LL}= c(\theta)\|g\|_{\LL} = c(\theta) \|w\|_{\Htheta}.
 \end{equation} 
 In total, we have produced a bounded map 
 \begin{align*}
 G_{0}\to \L^\infty_{t}\L^2_{x}, \; w\mapsto v \quad \text{with} \quad \|v\|_{\L^\infty_{t}\L^2_{x}}\le c(\theta)\|w\|_{\Htheta}.
 \end{align*}
By density, it has a bounded extension $\cM:\Htheta\to \L^\infty_{t}\L^2_{x}$. Due to \eqref{eq:v(t)-duality estimate} this extension is defined weakly by
 $$
\angle{\cM w(t)}{\varphi}
 =\int^{\infty}_{-\infty} \angle {\tau_{t}g(\sigma)}{M_{\theta} (h)(\sigma)}\, \d \sigma, 
 $$
where as before $w=M_{\theta}g\in \Htheta$ and $h(t,x)=1_{(-\infty,0)}(t)(e^{-t\Delta}\varphi)(x)$ with $\varphi\in \L^2_{x}$.

\medskip
 
\paragraph{\itshape Step 2: $v \in \C_{0}(\L^2_{x})$}

By density of $G_{0}$ in $\Htheta$ and closedness of   $\C_{0}(\L^2_{x})$ in $\L^\infty_{t}\L^2_{x}$, it is enough by \eqref{eq:cMbdd}  to show that $v=\cM w \in C_{0}(\L^2_{x})$ for all $ w\in G_{0}$. In that case we have seen that $w\in \L^1_{t}\L^2_{x}$.  For the continuity in time, write \eqref{eq:v(t)} as $v(t)= \int_{-\infty}^0 \e^{-s\Delta}w(s+t)\, \d s$, so that continuity follows right away from the continuity of time translations in $\L^1_{t}\L^2_{x}$ and contractivity of the heat semigroup on $\L^2_{x}$. 
For the limits $v(t) \to 0$ as $|t| \to \infty$, we use dominated convergence in \eqref{eq:v(t)} as follows. By contractivity of the heat semigroup, the integrand is bounded in $\L^2_{x}$ by $\|w(s)\|_{\L^2_{x}}$ and this function is integrable with respect to $s$. The limit as $t \to - \infty$ follows immediately, whereas for $t \to \infty$ we additionally use that the heat semigroup tends to $0$ strongly in $\L^2_{x}$.

\medskip

 \paragraph{\itshape Step 3:  $v=\cM w$ is a solution of $\partial_{t}v-\Delta v=w$ in $\cS'(\R^{n+1})$.} 
 
 Assume that $w= M_\theta g \in G_0$ with $g\in \cS_{00}(\R^{n+1})$. Since $w, \Delta w \in \L^1_{t}\L^2_{x}$ and $t\mapsto w(t)$ is continuous as an $\L^2_{x}$-valued function, we obtain $v'(t)=w(t)+\Delta v(t)$ in $\L^2_{x}$ for all $t \in \R$. From this we conclude that $\Angle v {-\partial_{t}\phi -\Delta \phi}=\Angle {w}{\phi}$ for all $\phi\in \cS(\R^{n+1})$. It remains to argue by density, letting $g$ approximate any element in $\L^2_{t}\L^2_{x}$ in this equality for fixed $\phi$. 

\medskip

 \paragraph{\itshape Step 4: ${v=\cM w\in \cVdot}$ with $\|v\|_{\cVdot}\lesssim \|w\|_{\Htheta}$}
 
 Again, it is enough to proceed by density after proving the claim for $w = M_\theta g$ when
 $g\in \cS_{00}(\R^{n+1})$ with a constant that does not depend on this assumption.  By Fourier transform from the equation,  $(i\tau+|\xi|^2)\hat v= |\tau|^{\theta/2}|\xi|^{1-\theta}\hat g$ as tempered distributions so that $\hat v= (i\tau+|\xi|^2)^{-1} |\tau|^{\theta/2}|\xi|^{1-\theta}\hat g$. Here, the right-hand side is again a tempered distribution of the form  $(|\tau|+|\xi|^2)^{-1/2}m\hat g$, so that 
 \begin{align*}
 	\|v\|_{\cVdot} \le (2\pi)^{-(n+1)/2} \|m\hat g\|_{\LL}\le \| m\|_{\infty}\|g\|_{\LL}=\| m\|_{\infty}\|w\|_{\Htheta},
 \end{align*}
where
\begin{align*}
	m=(|\tau|+|\xi|^2)^{1/2} (i\tau+|\xi|^2)^{-1} |\tau|^{\theta/2}|\xi|^{1-\theta}. &\qedhere
\end{align*}
\end{proof}

We observe that there is a substitute result for the non admissible pair $(\infty,2)$ that does not involve the variational space $\cVdot$.  

\begin{prop}
\label{prop:maxregL1L2} 
Let $w\in \L^2_{t}\Hdot^{-1}_{x}+ \L^{1}_{t}\L^{2}_{x}$. Then there exists $v\in \C_{0}(\L^2_{x})\cap\L^2_{t}\Hdot^1_{x}$,  unique in the {class of distributions} with $\nabla v\in \LL$, solution of $\partial_{t}v-\Delta v=w$ 
in {$\cD'(\R^{n+1})$}. It has the estimates
$$
\|\nabla v\|_{\LL} \le \|w\|_{\L^2_{t}\Hdot^1_{x}+ \L^{1}_{t}\L^{2}_{x}}  \quad \& \quad \sup_{t\in \R} \|v(t)\|_{\L^2_{x}} \le \|w\|_{\L^2_{t}\Hdot^1_{x}+ \L^{1}_{t}\L^{2}_{x}}.
$$ 
\end{prop}

\begin{proof} 
Uniqueness is again provided by Lemma~\ref{lem:uniqueness}.
The case where $w=-\div F\in \L^2_{t}\Hdot^{-1}_{x}$ is given by the solution $v$ in   Lemma~\ref{lem:maxreg} with $\theta=0$, and  checking the constants we have
$$
\|v\|_{\L^\infty_{t}\L^2_{x}} \le 2^{-1/2} \|F\|_{\LL} \quad \& \quad \|\nabla v\|_{\LL}\le \|F\|_{\LL}.
$$
Assuming now that $w\in \L^1_{t}\L^2_{x}$,  Steps 1 and 2 of the proof of Lemma~\ref{lem:maxreg} show that $v$ given by  \eqref{eq:v(t)} belongs to $\C_{0}(\L^2_{x})$ with  $\|v(t)\|_{\L^2_{x}} \le \|w\|_{\L^{1}_{t}\L^{2}_{x}}$. To show that $\nabla v\in  \LL$, we take a vector field
$\widetilde \Phi\in \cS(\R^{n+1})$  and $\tilde v$ given by 
$\tilde v(s)=\int^{\infty}_{s} \e^{(t-s)\Delta}(\div\widetilde \Phi)(t)\, \d t$, $s\in \R$, is the solution of Lemma~\ref{lem:maxreg} in the case $\theta=0$ for the backward equation 
$-\partial_{s}\tilde v-\Delta \tilde v= \div \tilde \Phi$. In particular $\tilde v \in 
\C_{0}(\L^2_{x})$ with $\|\tilde v\|_{\L^\infty_{t}\L^2_{x}} \le c(0) \|\tilde \Phi\|_{\LL}$ and $c(0) =  2^{-1/2}$. As 
$$
\Angle{ v}{\div\tilde \Phi}= \Angle {w}{\tilde v},
$$
we deduce that 
$ \|\nabla v\|_{\LL}\le 2^{-1/2} \|w\|_{\L^1_{t}\L^2_{x}}$. 
{Eventually, we get $\partial_{t}v-\Delta v=w$ in $\cD'(\R^{n+1})$ as before.}
\end{proof}

Let us state  consequences of these two propositions. The first one shows a set of lower bounds for the heat operator.

\begin{cor} 
\label{cor:maxreg}
For any distribution $u$ with $\nabla u \in \LL$, there is a constant $c\in \IC$ such that  $u-c\in \C_{0}(\L^2_{x})$ with
$$
\|u-c\|_{\L^\infty_{t}\L^2_{x}} \le c(n,r,q)\|\partial_{t}u-\Delta u\|_{\L^2_{t}\Hdot^{-1}_{x}+ \L^{r'}_{t}\L^{q'}_{\vphantom{t} x} },
$$
provided that $(r,q)$ is admissible or $(r,q)=(\infty,2)$ and that the right hand side is finite.
\end{cor}

\begin{proof}  Let  $v$ be the solution of $\partial_{t}v-\Delta v=\partial_{t}u-\Delta u$ given by  Propositions~\ref{prop:maxregDeldot} or \ref{prop:maxregL1L2}.  By uniqueness, $v-u$ is a constant $c$ so that 
$u-c\in \C_{0}(\L^2_{x})$. 
Optimizing over all possible decompositions of $\partial_{t}u-\Delta u$ in $\L^2_{t}\Hdot^{-1}_{x}+ \L^{r'}_{t}\L^{q'}_{\vphantom{t} x}$
yields the estimate.  
\end{proof}
 
The second one {will be our fundamental regularity estimate in the following}.

 \begin{cor}
\label{cor:mainreg}
Let $u\in \cD'(\R^{n+1})$. Assume $\nabla u\in \LL$ and $\partial_{t}u\in\L^2_{t}\Hdot^{-1}_{x}+ \L^{r'}_{t}\L^{q'}_{\vphantom{t} x}$ for $(r,q)$ an admissible pair  or $(r,q)=(\infty,2)$. Then
there is a constant $c\in \IC$ such that $u-c\in   \C_{0}(\L^2_{x})$ with
$$
 \sup_{t\in \R}\|u(t)-c\|_{\L^2_{x}} \le c(n,q,r)(\|\nabla u\|_{\LL}+ \|\partial_{t}u\|_{\L^2_{t}\Hdot^{-1}_{x}+ \L^{r'}_{t}\L^{q'}_{\vphantom{t} x}}).
$$
Moreover, if $(r,q)$ is admissible, then $u-c\in \cVdot$ with the same estimate {on $\|u-c\|_{\cVdot}$}.
\end{cor}

\begin{proof}
We see that  
$\partial_{t}u-\Delta u  \in \L^2_{t}\Hdot^{-1}_{x} +  \L^{r'}_{t}\L^{q'}_{\vphantom{t} x}$ with 
$$\|\partial_{t}u-\Delta u\|_{\L^2_{t}\Hdot^{-1}_{x}+ \L^{r'}_{t}\L^{q'}_{\vphantom{t} x} }\le \|\nabla u\|_{\LL}+ \|\partial_{t}u\|_{\L^2_{t}\Hdot^{-1}_{x}+ \L^{r'}_{t}\L^{q'}_{\vphantom{t} x}}.$$ We apply Corollary~\ref{cor:maxreg} to get the estimate. That $u-c\in \cVdot$ when $(r,q)$ is admissible is already in Proposition~\ref{prop:maxregDeldot}.
\end{proof}

\begin{rem}
\label{rem:mainreg2}
\begin{enumerate}
\item 
The case $\theta=0$ of Lemma~\ref{lem:maxreg} on the  half-space $(0,\infty)\times \R^{n}$ appears in \cite{AMP}. It can be seen 
as the homogeneous version of Lions' embedding theorem, {which would  apply had we assumed in addition  $u\in \L^2_{t}\L^2_{x}$}. An important point is that the interval is infinite, otherwise 
the homogeneous version is wrong. Here, the statement with the time interval being the real line simplifies some matters 
of the proof in \cite{AMP} when $\theta=0$ and we extend it to $\theta<1$. However, the statement does not hold when $\theta=1$. When $\theta>0$, the situation is not as symmetric as Lions' embedding theorem since the hypothesis cannot be put into a form 
$u\in E$, $\partial_{t}u\in E'$, where $E$ is a Banach space and $E'$ its dual for the $\LL$ duality. We relax on $u$ since we do not want 
to impose more than $u\in \L^2_{t}\Hdot^1_{x}$ for the spatial regularity.  
In any case, the conditions $u\in \cVdot$ and $\partial_{t}u\in \cVdot'$ do not imply $u$ continuous into $\L^2_{x}$. Indeed, we saw that the time derivative maps $\cVdot$ into $\cVdot'$, but $\cVdot $ is not contained in $\C(\L^2_{x})$. 

\item  
In the hypotheses of  Corollaries~\ref{cor:maxreg} and \ref{cor:mainreg}, one can take a finite sum of {spaces} of the same types with different pairs of exponents. {This is a consequence of our constructive proof: $u-c$ is the sum of solutions to the heat equation with right-hand sides equal to the  various different components.}
\end{enumerate}
\end{rem}

\subsection{Integral equalities}
 \label{sec:integralequalities}
 
Integral identities are well-known in our context when the time interval is finite~\cite{LSU}. When the time interval is infinite, they require additional care in both the assumptions and the proofs, and Lemma~\ref{lem:maxreg} becomes the essential instrument. Therefore, the next lemma is not a straightforward generalization of known results.

\begin{lem}
\label{lem:energy} Let $u\in \cD'(\R^{n+1})$. Assume $\nabla u\in \LL$ and $\partial_{t}u= -\div F+ g$ with $F\in \L^2_{t}\L^2_{x}$, $g\in \L^{r'}_{t}\L^{q'}_{\vphantom{t} x} $ and $(r,q)$ an admissible pair or $(r,q)=(\infty,2)$. Then, up to a constant, $u\in  \C_{0}(\L^2_{x})$, 
$t\mapsto \|u(t)\|_{\L^2_{x}}^2 $ is absolutely continuous on $\R$ and  for all $\sigma<\tau$,
\begin{equation}
\label{eq:energy}
\|u(\tau)\|_{\L^2_{x}}^2-\|u(\sigma)\|_{\L^2_{x}}^2= 2\Re \int_{\sigma}^\tau \angle{F(t)}{\nabla u(t)} + \angle {g(t)}{u(t)}\, \d t.
\end{equation}
\end{lem}

\begin{rem} In the spirit of Remark~\ref{rem:mainreg2}~(ii),  one can  replace $g$ by linear combinations of functions of the same type  with different admissible pairs and the pair $(\infty,2)$. 
\end{rem}

\begin{proof} From  Corollary~\ref{cor:mainreg} we  know that  $u-c\in \C_{0}(\L^2_{x})\cap \L^{r}_{t}\L^{q}_{\vphantom{t} x}$ for some constant $c$. Adjusting the constant to $0$,
the integrand of \eqref{eq:energy} is well-defined and integrable on $\R$. It remains to prove the integral identity. 
	
We begin with assuming that $(r,q)$ is an admissible pair. Let $\varphi_{\varepsilon}=\frac 1 \varepsilon \varphi(\frac \cdot \varepsilon)$, $\varepsilon>0$, be a mollifying sequence of  $\R$ with 
$\varphi$  smooth, compactly supported in $[-1,1]$ and $\int_\R \varphi=1$, and let $u_{\varepsilon} \coloneqq  \varphi_{\varepsilon}\star u$, where 
convolution is in the $t$-variable. Clearly, $t\mapsto u_{\varepsilon}(t)$ is  of class $\C^1$ as an $\L^2_{x}$-valued function. Thus,
\begin{equation}
\label{eq:eps}
\|u_{\varepsilon}(\tau)\|_{\L^2_{x}}^2-\|u_{\varepsilon}(\sigma)\|_{\L^2_{x}}^2= 2\Re \int_{\sigma}^\tau \angle{u_{\varepsilon}'(t)}{ u_{\varepsilon}(t)}  \, \d t.
\end{equation}
Since $u\in \C_{0}(\L^2_{x})$, the left-hand side converges to the one of \eqref{eq:energy} as $\varepsilon\to 0$. 
Next, to justify the convergence of the integral in \eqref{eq:eps} to 
the one in \eqref{eq:energy}, a computation yields
$$
\angle{u_{\varepsilon}'(t)}{ u_{\varepsilon}(t)} = \angle{F_{\varepsilon}(t)}{\nabla u_{\varepsilon}(t)} + \angle {g_{\varepsilon}(t)}{u_{\varepsilon}(t)},
$$
where $F_{\varepsilon} \coloneqq   \varphi_{\varepsilon} \star F$ and $g_{\varepsilon} \coloneqq  \varphi_{\varepsilon} \star g$.  Observe that $ (h_{1},h_{2})\mapsto \int_{\sigma}^\tau \angle {h_{1}(t)}{h_{2}(t)}\, \d t$ is a sesquilinear continuous form on $X\times Y$, where $(X,Y)=(\LL,\LL)$ or $(\L^{r'}_{t}\L^{q'}_{\vphantom{t} x}, \L^{r}_{t}\L^{q}_{\vphantom{t} x})$ {and that the pairings above are of this type} as $(r,q)$ is admissible. In each factor the convolution with $\varphi_{\varepsilon}$ is uniformly bounded and  converges strongly to the identity operator. The convergence follows.

If $(r,q)=(\infty,2)$, then we repeat the above argument  with $(X,Y)=(\L^{1}_{t}\L^{2}_{x},\C_{0}(\L^2_{x}))$.
\end{proof}

The lemma above can be localized to a half-infinite time interval as follows.

\begin{cor}
\label{cor:energy} 
Let $\I$ be an open half-infinite interval of $\R$ and $u\in \cD'(\I\times \R^{n})$. Assume $\nabla u\in \L^2(\I;\L^2_{x})$ and 
$\partial_{t}u= -\div F+ g$ with  $F\in \L^2(\I;\L^2_{x})$, $g\in \L^{r'}(\I;\L^{q'}_{x} )$ and $(r,q)$ an admissible pair or $(r,q)=(\infty,2)$. Then, up to a constant,  $u\in \C_{0}(\overline \I; \L^2_{x})$ and the function 
$t\mapsto \|u(t)\|_{\L^2_{x}}^2 $ is absolutely continuous on $\overline \I$ with 
\begin{equation}
\label{eq:energy2}
\|u(\tau)\|_{\L^2_{x}}^2-\|u(\sigma)\|_{\L^2_{x}}^2= 2\Re \int_{\sigma}^\tau \angle{F(t)}{\nabla u(t)} + \angle  {g(t)}{u(t)}\, \d t
\end{equation}
for all  $\sigma, \tau \in \overline \I$,  $\sigma<\tau$.
\end{cor}

\begin{proof}
It is enough to consider the case $\I=(a,\infty)$, $a\in \R$.   The strategy is to extend $u$ by even reflection at $t=a$  and  $F$ and $g$ by odd reflection at $t=a$ so that the assumptions of Lemma~\ref{lem:energy} apply to these extensions. The conclusion follows by restricting back to $\overline \I$.
However, some care needs to be taken since we do not assume \emph{a priori} that $u$ is a locally integrable function and we provide details for the convenience of the reader. Take $a=0$ for simplicity. We construct a distribution $v \in \cD'(\R^{n+1})$ (the even extension of $u$) verifying the hypothesis of Lemma~\ref{lem:energy} with the odd extensions of $F$ and $g$, and whose restriction to $(0,\infty)\times \R^{n}$ is $u$.  
 	
For any $\psi \in \cD(\R^n)$ we define the distribution $\angle{u(t)}{\psi}$ on $(0,\infty)$ by
	\begin{align*}
		\angle{\angle{u(t)}{\psi}}{f} \coloneqq \Angle{u}{f \otimes \psi},
	\end{align*}
where $f \otimes \psi$ is the tensor product. The equation $\partial_{t}u= -\div F+ g$ implies that for any $ \psi\in \cD(\R^n)$ we have $\angle{u(t)}{\psi}'=  \angle{F(t)}{\nabla \psi} +\angle{g(t)}{\psi}$ in $\cD'(0,\infty)$. As the right-hand side is locally integrable from the assumptions on $F$ and $g$, this shows that $\angle{u(t)}{\psi}$ can be identified with a continuous function on $(0,\infty)$ that extends continuously to $0$. We continue to use the suggestive notation $t \mapsto \angle{u(t)}{\psi}$. For $\psi\in \cD(\R^n)$ and $f\in \cD(\R)$ the formula 
$$
\Angle{v}{f\otimes \psi}=\int_{0}^\infty \angle{u(t)}{\psi}(\overline f(t)+\overline f(-t))\, \d t
$$
defines a distribution $v \in \cD'(\R^{n+1})$. (Recall that distributions in $\R^{n+1}=\R\times \R^n$ are uniquely determined on tensor products.) Taking $f$ supported in $(0,\infty)$ gives that $v=u$ in $ \cD'((0,\infty)\times \R^{n})$. Next, integration by parts shows that 
\begin{align*}
\Angle{v}{f'\otimes \psi} &= - \int_{0}^\infty (\angle{F(t)}{\nabla \psi} +\angle{g(t)}{\psi})(\overline f(t)-\overline f(-t))\, \d t
\\
&
= -\int_{-\infty}^\infty (\angle{F_{o}(t)}{\nabla \psi} +\angle{g_{o}(t)}{\psi})\overline f(t)\, \d t,
\end{align*}
where $F_{o}$ and $g_{o}$ are the odd extensions of $F$ and $g$, respectively. Thus,  
$\partial_{t} v= -\div  F_{o}+  g_{o}$ in  $\cD'(\R^{n+1})$. 
Lastly, 
\begin{align*}
\Angle{v}{f\otimes -\div \psi} &=  \int_{0}^\infty \angle{\nabla u(t)}{ \psi} (\overline f(t)+\overline f(-t))\, \d t
\\
&
= \int_{-\infty}^\infty \angle{(\nabla u (t))_{e}}{ \psi} \overline f(t)\, \d t,
\end{align*}
where $(\nabla u)_{e}$ is the even extension of $\nabla u$, so that $\nabla v=  (\nabla u)_{e}$ in $\cD'(\R^{n+1})$. We have proved our claim. 
\end{proof}

The conclusion of Lemma~\ref{lem:energy} can be polarized, given two functions $u, \tilde{u}$ that verify the assumptions {(with possibly different pairs $(r,q)$)}. Due to the extendability that we have seen in  the previous proof, the same also works with open, half-infinite intervals and the conclusion reads as follows.

\begin{cor}
\label{cor:energy2}
If $u,\tilde u$ satisfy the same assumptions as in Corollary~\ref{cor:energy} on two open half-infinite intervals $\I$ and $\J$, then  
after eliminating a constant from each of $u$ and $\tilde u$, the function $t\mapsto \angle {u(t)}{\tilde u(t)}$ is absolutely continuous on $\overline{\I\cap \J}$ and
\begin{align}
\label{eq:energypolarisation}
\begin{split}
\angle {u(\tau)}{\tilde u (\tau)}-\angle {u(\sigma)}{\tilde u (\sigma)}&=  \int_{\sigma}^\tau \angle{F(t)}{\nabla \tilde u(t)} + \angle {g(t)}{\tilde u(t)}\, \d t\\ 
&\qquad + \int_{\sigma}^\tau \angle{\nabla  u(t)}{\tilde F(t)} + \angle {u(t)}{\tilde g(t)}\, \d t,
\end{split}
\end{align}
whenever $\sigma,\tau\in \overline{\I\cap \J}$, $\sigma<\tau$.
\end{cor}

 \subsection{Existence and uniqueness results}
\label{sec:existenceanduniqueness}

We come back to the study of parabolic equations.  Up until Section~\ref{sec:invertibility} included,  we
\begin{center}
\textbf{fix a single compatible pair $\boldsymbol{(\tilde r_{1}, \tilde q_{1})}$ for lower order coefficients \\ and its corresponding admissible conjugate pair  $\boldsymbol{(r_{1},q_{1})}$.} 
\end{center}
Recall that Proposition~\ref{prop:Lt} yields boundedness of an operator $\cH:\cVdot\to \cVdot'$, which acts as $\partial_{t}+\Lt $.  We develop the existence and uniqueness theory under the hypothesis 
$$
\mathbf{(H_{0}) \qquad \boldsymbol{\cH} \ and \ \boldsymbol{\cH^*} \ are \ invertible.}
$$
This means that the constants in our estimates will also depend on $\Lambda=\|A\|_{\infty}$, $P_{\tilde r_{1}, \tilde q_{1}}$ and the norm of the inverse of $\cH$. We do not make this dependence explicit until Section~\ref{sec:invertibility}, where we discuss a sufficient condition for invertibility. {So, we write for instance $C(n,q,r)$ to mean dependence on $n,q,r$ and possibly the aforementioned quantities.}
 
In the following, we only state results involving the operator $\partial_{t}+\Lt $. All results apply \textit{mutatis mutandis} to $-\partial_{t}+\Ltstar$ since  both operators are indistinguishable from the assumptions at this stage.  We are going to prove uniqueness and existence in the class of $\Deldot^{r_{1},q_{1}}$-solutions that we introduced at the end of Section~\ref{sec:homogeneousparabolicsystems}.

\begin{prop}
\label{prop:regularityofDeldotr1q1solutions} Any $\Deldot^{r_{1},q_{1}}$-solution of $\partial_{t}u+\Lt u=f\in  \L^2_{t}\Hdot^{-1}_{x}+ \L^{r'}_{t}\L^{q'}_{\vphantom{t} x}$ for $(r,q)$ an admissible pair  or $(r,q)=(\infty,2)$ belongs to $\C_{0}(\L^2_{x})$ with
$$\|u\|_{\L^\infty_{t}\L^2_{x}} \le C(n,q_{1},r_{1})\|u\|_{\Deldot^{r_{1}, q_{1}}}+ C(n,q,r)\|f\|_{\L^2_{t}\Hdot^{-1}_{x}+ \L^{r'}_{t}\L^{q'}_{\vphantom{t} x}}.$$
 \end{prop}
 
\begin{proof} 
We have that $\Lt u \in (\Deldot^{r_{1},q_{1}})'=\L^2_{t}\Hdot^{-1}_{x}+ \L^{r_{1}'}_{t}\L^{q_{1}'}_{\vphantom{t} x}$ with bound controlled by $\|u\|_{\Deldot^{r_{1},q_{1}}}$ according to Remark~\ref{rem:Delta}, so that $\partial_{t}u=-\Lt u+ f \in \L^2_{t}\Hdot^{-1}_{x}+ \L^{r_{1}'}_{t}\L^{q_{1}'}_{\vphantom{t} x}+ \L^{r'}_{t}\L^{q'}_{\vphantom{t} x}$.   Remark~\ref{rem:mainreg2}~(ii) yields  the desired conclusion up to a constant for $u$ and the constant  must be 0 as $u\in \Deldot^{r_{1},q_{1}}$.  
\end{proof}

\begin{thm}
\label{thm:uniqueness}
Assume $\mathbf{(H_0)}$.   If $u$ is a $\Deldot^{r_{1},q_{1}}$-solution of $\partial_{t}u+\Lt u= 0$, then $u=0$.
\end{thm}

\begin{proof}
We may apply Corollary~\ref{cor:mainreg}, so that $u-c\in \cVdot\cap \C_{0}(\L^2_{x})$. The constant $c$ vanishes as   $u\in \C_{0}(\L^2_{x})$ from the above proposition. Thus, we have $u\in \cVdot$ with  $\cH u=0$ and $u=0$ follows by $\mathbf{(H_0)}$.
 \end{proof}

\begin{thm}
\label{thm:Fg}
Assume $\mathbf{(H_0)}$. Let $(r,q)$ be any admissible pair, $F\in \LL$ and $g\in \L^{r'}_{t}\L^{q'}_{\vphantom{t} x}$. Then   $u \coloneqq \cH^{-1}(-\div F+g)\in \cVdot$ is the unique   $\Deldot^{r_{1},q_{1}}$-solution  of 
 \begin{equation}
 \label{eq:Fg} 
\partial_{t}u+\Lt u= -\div F+g.
\end{equation}
 Moreover, it belongs to $\C_{0}(\L^2_{x})$ with  
\begin{equation}
 \label{eq:uLinftyL2}
\|u\|_{\L^\infty_{t}\L^2_{x}} \le C(n,q,r)(\|F\|_{\LL}+ \|g\|_{\L^{r'}_{t}\L^{q'}_{\vphantom{t} x}}).
\end{equation}
\end{thm}

\begin{proof}
As $-\div F +g \in \cVdot'$ by Lemma~\ref{ref:embedding}, we see that $u$ is well-defined in $\cVdot$ and belongs to $\Deldot^{r_{1},q_{1}}$  thanks to the same lemma. It is thus a $\Deldot^{r_{1},q_{1}}$-solution of  \eqref{eq:Fg}.
By Theorem~\ref{thm:uniqueness} it is unique and by Proposition~\ref{prop:regularityofDeldotr1q1solutions} it belongs to $\C_{0}(\L^2_{x})$. The estimate \eqref{eq:uLinftyL2} follows from that. 
\end{proof}

The previous theorem does not apply when $g \in \L^{1}_{t}\L^{2}_{x}$ since $(r,q)=(\infty,2)$ is not admissible and $\cH^{-1}g$ does not make sense. Yet, we can construct a $\Deldot^{r_1,q_1}$-solution that falls outside of the variational $\cVdot - \cVdot'$ setting.

\begin{thm}
\label{thm:L1L2} Assume $\mathbf{(H_0)}$. For every $g\in \L^1_{t}\L^2_{x}$, there exists a unique $\Deldot^{r_{1},q_{1}}$-solution of 
\begin{equation}
 \label{eq:L1L2}
\partial_{t}u+\Lt u= g.
\end{equation}
 Moreover, this solution belongs to  $\L^{r}_{t}\L^{q}_{\vphantom{t} x}$ for any admissible pair $(r,q)$ and to $\C_{0}(\L^2_{x})$ with 
\begin{equation}
 \label{eq:uL1L2toGrq}
\|u\|_{\Deldot^{r,q}} \le C(n,q,r)\|g\|_{\L^1_{t}\L^2_{x}}
 \end{equation}
 and  
\begin{equation}
 \label{eq:uLinftyL1L2}
\|u\|_{\L^\infty_{t}\L^2_{x}} \le C(n,q_{1},r_{1}) \|g\|_{\L^{1}_{t}\L^{2}_{x}}.
\end{equation}
\end{thm}

\begin{proof}
 The uniqueness follows from Theorem~\ref{thm:uniqueness}.  For the existence, let $T:\L^1_{t}\L^2_{x} \to \cD'(\R^{n+1})$ defined by
 \begin{equation}
 \label{eq:defT}
 \Angle {Tg}{\tphi} \coloneqq \Angle {g}{\Hstarinverse\tphi}, \quad g\in \L^1_{t}\L^2_{x}, \ \tphi \in \cD(\R^{n+1}).
 \end{equation}
 From  Theorem~ \ref{thm:Fg} applied to $\cH^*$, we have that the restriction
 $\Hstarinverse\colon (\Deldot^{ r, q})' \to \C_{0}(\L^2_{x})$ is bounded for any admissible pair $(r,q)$, so that 
  \begin{equation}
 \label{eq:L1L2toGrq}
\|Tg\|_{\Deldot^{r,q}} \le C(n,q,r)\|g\|_{\L^1_{t}\L^2_{x}}.
\end{equation}
We next show that   $u \coloneqq Tg$ is a $\Deldot^{r_{1},q_{1}}$-solution of  \eqref{eq:L1L2}. Observe that $u$  agrees with $\cH^{-1}g$ for  $g\in (\Deldot^{r,q})' \cap \L^1_{t}\L^2_{x}$, hence for $g$ in a dense subspace, for example $\cD(\R^{n+1})$. For those functions $g$, we have that $u$ is a $\Deldot^{r_{1},q_{1}}$-solution of  \eqref{eq:L1L2}. 
Secondly,  let $g\in \L^1_{t}\L^2_{x}$ and $(g_{k})$ be a sequence in $\cD(\R^{n+1})$ that converges to $g$ in $ \L^1_{t}\L^2_{x}$. Testing the equation for $u_{k} \coloneqq Tg_{k}$ against a function $\tphi \in \cD(\R^{n+1})$, we have
$$
-\Angle {u_{k}} {\partial_{t}\tphi} + \Angle {\Lt u_{k}}{\tphi} = \Angle {g_{k}}{\tphi}.
$$
The estimate \eqref{eq:L1L2toGrq} shows that $u_{k}$ converges to $u$ in the norm of $\Deldot^{r_{1},q_{1}}$, and this allows us to pass to the limit on the left-hand side, using Remark~\ref{rem:Delta} for the second term. This proves \eqref{eq:L1L2} and \eqref{eq:uL1L2toGrq}. 

Eventually, $u\in \C_{0}(\L^2_{x})$ follows from Proposition~\ref{prop:regularityofDeldotr1q1solutions} and we obtain the estimate  \eqref{eq:uLinftyL1L2}  from the estimate in that proposition if we plug in \eqref{eq:uL1L2toGrq} with $(r,q)=(r_{1},q_{1})$.
\end{proof}

{The next result is central for our constructive approach to fundamental solutions and Green operators.}

\begin{thm}
\label{thm:deltasL2} Assume $\mathbf{(H_0)}$. For all $s\in \R$ and $\psi\in \L^2_{x}$ there exists a unique $\Deldot^{r_{1},q_{1}}$-solution  of 
\begin{equation}
 \label{eq:deltasL2}
\partial_{t}u+\Lt u= \delta_{s}\otimes \psi,
\end{equation}
 where $\delta_{s}$ is the Dirac mass at $t=s$ and $\otimes$ denotes the tensor product. Moreover, 
this solution belongs to  $\Deldot^{r,q}$ for any admissible pair $(r,q)$ with 
\begin{equation}
 \label{eq:udeltasL2toGrq}
\|u\|_{\Deldot^{r,q}} \le C(n,q,r)\|\psi\|_{\L^2_{x}}.
\end{equation}
Furthermore, $u\in \C_{0}(\R\setminus \{s\}; \L^2_{x})$, it has $\L^2_{x}$ limits $u(s^\pm)$ when $t\to s^\pm$ with the jump relation
\begin{equation}
\label{eq:split}
u(s^+)-u(s^-)=\psi
\end{equation}
and $t\mapsto \|u(t)\|_{\L^2_{x}}^2$ has an absolutely continuous extension to each of $[s,\infty)$ and $(-\infty,s]$ with estimate
 \begin{equation}
 \label{eq:uLinftydeltas}
\sup_{t\ne s}\|u(t)\|_{\L^2_{x}} \le C(n,q_{1},r_{1})\|\psi\|_{\L^2_{x}},
\end{equation}
where $C(n,q_{1},r_{1})$ does not depend on $s$ and $\psi$.  
Eventually, on each of the intervals $(s,\infty)$ and $(-\infty,s)$, the solution $u$ is the restriction of a function in $\cVdot$.
\end{thm}

\begin{proof}
The uniqueness follows from Theorem~\ref{thm:uniqueness}. For the existence,  given $s\in \R$, 
let $T_{s}:\L^2_{x} \to \cD'(\R^{n+1})$ be defined by
\begin{equation}
\label{eq:defTs}
\Angle {T_{s}\psi}{\tphi} \coloneqq \angle {\psi}{(\Hstarinverse\tphi)(s)}, \quad \psi\in \L^2_{x}, \ \tphi \in \cD(\R^{n+1}).
\end{equation}
For any    admissible pair $(r,q)$,  Theorem~ \ref{thm:Fg} applies to $\Hstarinverse$ so that   
$\Hstarinverse\colon (\Deldot^{r,q})'\to \C_{0}(\L^2_{x})$ is bounded and it follows that
\begin{equation}
\label{eq:deltasL2toGrq}
\|T_{s}\psi\|_{\Deldot^{r,q}} \le C(n,q,r) \|\psi\|_{\L^2_{x}}.
\end{equation} 
In particular,   $u \coloneqq T_{s}\psi$ satisfies \eqref{eq:udeltasL2toGrq}. 
It is no loss of generality to assume $s=0$ as the general argument is the same (or can be deduced from $s=0$ by a translation in time, which preserves all assumptions with uniform constants).

\medskip

\noindent \emph{Step 1: $u$ is a   $\Deldot^{r_{1},q_{1}}$-solution of  \eqref{eq:deltasL2}.}
 Let $(\varphi_{\varepsilon})$ be a mollifying sequence  as in the proof of Lemma~\ref{lem:energy}. We apply Theorem~\ref{thm:L1L2} to $g_{\varepsilon} \coloneqq \varphi_{\varepsilon}\otimes \psi \in  \L^1_{t}\L^2_{x}$. We test the equation for $u_{\varepsilon} \coloneqq Tg_{\varepsilon}$ against a  function $\tphi \in \cD(\R^{n+1})$ and pass to the limit in this equation as $\eps \to 0$. 
First 
\begin{align*}
\Angle {u_{\varepsilon}} {\tphi} = \Angle {g_{\varepsilon}} {\Hstarinverse\tphi} &= \int_{\R} \varphi_{\varepsilon}(s) \angle {\psi}{(\Hstarinverse \tphi)(s)}\, \d s
\\
&\to \angle {\psi}{(\Hstarinverse\tphi)(0)}=
\Angle {T_{0}\psi}{\tphi},
 \end{align*}
where we used  $\Hstarinverse\tphi\in \C_{0}(\L^2_{x})$. Replacing $\tphi$ by ${\partial_{t}\tphi}$, we have that
$$
\Angle {u_{\varepsilon}} {\partial_{t}\tphi} \to \Angle {T_{0}\psi} {\partial_{t}\tphi}
$$
and similarly
$$
\Angle {g_{\varepsilon}}{\tphi}= \int_{\R} \varphi_{\varepsilon}(s) \angle {\psi}{\tphi(s)}\, \d s\to \angle {\psi}{\tphi(0)}=
\Angle {\delta_{0}\otimes \psi}{\tphi}.
$$
Eventually, we have seen that $u_{\varepsilon}$ converges to $u \coloneqq T_{0}\psi$ in $\cD'(\R^{n+1})$. 
The estimate \eqref{eq:L1L2toGrq} shows that $(u_{\varepsilon})$ is uniformly bounded in $\Deldot^{r_{1},q_{1}}$, which is  a dual space (it is even reflexive). Hence, it converges also weakly-star in $\Deldot^{r_{1},q_{1}}$ to $u$ and so Remark~\ref{rem:Delta} shows that
$$\Angle {\Lt u_{\varepsilon}}{\tphi} \to \Angle {\Lt u}{\tphi}.$$ {This proves \eqref{eq:deltasL2} and \eqref{eq:udeltasL2toGrq}.}

\medskip

\noindent \emph{Step 2: Proof of \eqref{eq:uLinftydeltas}.} We can apply  the integral equality \eqref{eq:energy2} of Corollary~\ref{cor:energy},  in which the right-hand side is $-2\Re \int_{\sigma}^\tau \angle{\Lt u(t)}{ u(t)} \, \d t$, when $\sigma,\tau\in (-\infty,0)$ or $\sigma,\tau \in (0,\infty)$.  Thus, letting $\sigma\to -\infty$ in the first case or $\tau\to \infty$ in the second case and using the estimate in Remark~\ref{rem:Delta}, we obtain 
\begin{equation*}
 \sup_{t\ne 0}\|u(t)\|_{\L^2_{x}}^2 \le (\|A\|_{\infty}  + P_{ \tilde r_{1}, \tilde q_{1}}) \|u\|_{\Deldot^{ r_{1},  q_{1}}}^2.
\end{equation*}
We conclude, using  \eqref{eq:udeltasL2toGrq}.

\medskip

\noindent \emph{Step 3: Proof of \eqref{eq:split}}. Let $\tphi \in \cD(\R^{n+1})$
 and $\theta\colon \R\to \R$ even, smooth, $0$ on $[0,1]$ and $1$ on $[2,\infty)$. Set $\theta_{\varepsilon}(t) \coloneqq \theta(t/\varepsilon)$ when $\varepsilon>0$. Using $\tphi\, \theta_{\varepsilon}$ as a test function for the equation,
$$
-\Angle {u} {\partial_{t}(\tphi\, \theta_{\varepsilon})} + \Angle {\Lt u}{\tphi\, \theta_{\varepsilon}} = 0.
$$
Expanding the first term, we obtain 
$$
-\Angle {u} {(\partial_{t}\tphi)\, \theta_{\varepsilon})} + \Angle {\Lt u}{\tphi \,\theta_{\varepsilon}} = \Angle {u} {\tphi \, (\theta_{\varepsilon})'}.
$$
We now pass to the limit as $\varepsilon\to 0$. Using again  the duality between $\Deldot^{r_{1},q_{1}}$ and its pre-dual for the $\LL$ duality, 
$\Angle {\Lt u}{\tphi\, \theta_{\varepsilon}}\to \Angle {\Lt u}{\tphi}$ by dominated convergence. Similarly, $\Angle {u} {(\partial_{t}\tphi)\, \theta_{\varepsilon}} \to \Angle {u} {\partial_{t}\tphi }$. Hence, the left-hand side above converges to $\angle {\psi}{\tphi(0)}$. 
The right-hand side rewrites after a change of variable as
$$
\int_{\R} \theta'(t)\angle {u(\varepsilon t)}{\tphi(\varepsilon t)}\, \d t=
\int_{-\infty}^0 \theta'(t)\angle {u(\varepsilon t)}{\tphi(\varepsilon t)}\, \d t 
+ \int_{0}^\infty \theta'(t)\angle {u(\varepsilon t)}{\tphi(\varepsilon t)}\, \d t,
$$
which, by dominated convergence and the existence of limits from the left and the right at 0 from Corollary~\ref{cor:energy}, tends to 
$$
\int_{-\infty}^0 \theta'(t)\angle {u(0^-)}{\tphi(0)}\, \d t 
+ \int_{0}^\infty \theta'(t)\angle {u(0^+)}{\tphi(0)}\, \d t = - \angle {u(0^-)}{\tphi(0)}+ \angle {u(0^+)}{\tphi(0)}.
$$
This proves \eqref{eq:split}. 

\medskip

\noindent \emph{Step 4: On the left and right of $0$, $u$ is a restriction of an element in $\cVdot$.}
It remains to see this last point.
Consider $w$  the even  extension across $t=0$ of the restriction of $u$ to $(0,\infty)\times \R^n$. Using the same  functions $\tphi$ and $\theta$ as above, we have 
$$
\Angle {w} {\partial_{t}\tphi } = \lim_{\varepsilon\to 0} \Angle {w} {(\partial_{t}\tphi)\, \theta_{\varepsilon} }.
$$
Since $w$ and $\theta_{\varepsilon}$ are even in $t$, the only contribution of $\tphi$ is through its odd part $\tphi_{o}(t) =\frac 1 2(\tphi(t)-\tphi(-t))$ and 
$$\Angle {w} {(\partial_{t}\tphi)\, \theta_{\varepsilon} }= 2\Angle {u} {(\partial_{t}\tphi_{o})\, \theta^+_{\varepsilon} }= 2 \Angle {u} {\partial_{t}(\tphi_{o}\, \theta^+_{\varepsilon}) }- 2\Angle {u} {\tphi_{o}\, (\theta^+_{\varepsilon})' },
$$
where $\theta^+_{\varepsilon}$ is the restriction of $\theta_{\varepsilon}$ to 
$(0,\infty)$. 
The first term on the right-hand side equals $2  \Angle {\Lt u} {\tphi_{o}\, \theta^+_{\varepsilon} }$, which converges to $2  \Angle {\Lt u} {\tphi_{o} }=   \Angle {v} {\tphi }$, where $v$ is the odd extension of $\Lt u $ restricted to $(0,\infty)\times \R^n$. It is clearly an element of $(\Deldot^{r_{1},q_{1}})'$. 
For the  second term,  
$$
|\Angle {u} {\tphi_{o}\, (\theta^+_{\varepsilon})' }|\le  \int_{\varepsilon}^{2\varepsilon} |(\theta^+_{\varepsilon})'(t)| |\angle {u(t)}{\tphi_{o}(t)}|\, \d t \le \int_{\varepsilon}^{2\varepsilon} \frac {\|\theta'\|_{\infty}} \varepsilon\, \|u(t)\|_{\L^2_{x}}\|\tphi_{o}(t)\|_{\L^2_{x}}\, \d t.
$$
As $\|u(t)\|_{\L^2_{x}}$ is bounded and  $\|\tphi_{o}(t)\|_{\L^2_{x}} \le Ct$, we obtain a bound on the order of $\varepsilon$ and this term tends to $0$. 

In conclusion, we proved $\partial_t w \in (\Deldot^{r_{1},q_{1}})'$. Since $w \in \Deldot^{r_{1},q_{1}}\subset \L^2_{t}\Hdot^1_{x}$, we obtain $w \in \cVdot$ from Corollary~\ref{cor:mainreg} and since $w|_{(0,\infty)} = u$, we are done.

The same argument can be done with the restriction of $u$ to $(-\infty,0)\times \R^n$.
\end{proof}

\subsection{Green operators}
\label{sec:propagators}

Theorem~\ref{thm:deltasL2} can be used to construct Green operators for the parabolic equation and its adjoint. We borrow this terminology from \cite{Lions}.

\begin{defn} Assume $\mathbf{(H_0)}$ and let $s,t \in \R$ and $\psi,\tpsi \in \L^2_x$.
\begin{enumerate}
\item For $t\ne s$, define $\Ga(t,s)\psi$ as the value at time $t$  in $\L^2_{x}$ of the  $\Deldot^{r_{1},q_{1}}$-solution $u$  of 
$\partial_{t}u+\Lt u= \delta_{s}\otimes \psi$ in Theorem~\ref{thm:deltasL2}. 
\item For $s\ne t$, define $\tGa(s,t)\tpsi$ as the value at time $s$  in $\L^2_{x}$ of the  $\Deldot^{r_{1},q_{1}}$-solution $\tilde u$ of
$-\partial_{s}\tilde u+\Lt^*\tilde u= \delta_{t}\otimes \tpsi $
in Theorem~\ref{thm:deltasL2}. 
\end{enumerate}
The operators $\Ga(t,s)$ and $\tGa(s,t)$ are called the \emph{Green operators} for the parabolic operator $\partial_{t}+\Lt $ and the (adjoint) parabolic operator $-\partial_{t}+\Ltstar   $, respectively.
\end{defn}

We recall that the orbit  $G(\cdot,s)\psi$ was defined as $T_s \psi$   in \eqref{eq:defTs}, which reads as the ``double duality formula''
\begin{equation}
\label{eq:G vs Ts}
\Angle{G(\cdot,s)\psi}{\tphi} = \angle{\psi}{(\Hstarinverse \tphi)(s)}.
\end{equation}
Indeed, $G(\cdot,s)\psi$ and $\Hstarinverse \tphi$ are solutions of parabolic problems for adjoint operators.

Rephrasing parts of Theorem~\ref{thm:deltasL2} in terms of Green operators yields the following result.

\begin{prop} 
\label{prop:Gamma}
Assume $\mathbf{(H_0)}$.
\begin{enumerate}
\item Let $s\in \R$ and  $\psi \in \L^2_x$. The function $t\mapsto \Ga(t,s)\psi$ is in  $\C_{0}(\R\setminus \{s\}; \L^2_{x})$ and the following limits exist in $\L^2_{x}$:
\begin{align*}
	\P_{s}^\pm \psi\coloneqq \lim_{t\to s^\pm} \Ga(t,s)\psi.
\end{align*}
\item Let $t \in \R$ and  $\tpsi \in \L^2_x$.  The function $s\mapsto \tGa(s,t)\tpsi$ is in $\C_{0}(\R\setminus \{t\}; \L^2_{x})$ 
and the following limits exist in $\L^2_x$: 
\begin{align*}
	\tP_{t}^\pm \tpsi= \lim_{s\to t^\pm} \tGa(s,t)\tpsi.
\end{align*}
\item The operators $\Ga(t,s), \tGa(s,t)$ are uniformly bounded on $\L^2_{x}$ with respect to $(s,t)$ with $ t\ne s$.
\end{enumerate}
\end{prop}

Next, we list a number of properties involving the Green operators and their limits.

\begin{thm} Assume $\mathbf{(H_0)}$. 
\label{thm:Gammats} 
\begin{enumerate}

\item For each $s$,  
\begin{align*}
	\P_{s}^+ - \P_{s}^-=\Id, \quad \tP_{s}^+ - \tP_{s}^-=-\Id
\end{align*}
and  $\P_{s}^\pm$ and $\tP_{s}^\mp$ are adjoint operators, respectively.
\item For $s\ne t$, $\Ga(t,s)$ and $\tGa(s,t)$ are adjoint operators.
\item If $t>s$, then 
\begin{align*}
	\P_{t}^+\Ga(t,s) &= \Ga(t,s), \quad  \P_{t}^-\Ga(t,s) =0, \\
	 \Ga(t,s)\P_{s}^+&= \Ga(t,s), \quad \Ga(t,s)\P_{s}^-=0.
\end{align*}
\item If $s>t$, then 
\begin{align*}
	\P_{t}^-\Ga(t,s) &= \Ga(t,s), \quad \P_{t}^+\Ga(t,s)=0,  \\
	\Ga(t,s)\P_{s}^-&= \Ga(t,s), \quad \Ga(t,s)\P_{s}^+=0.
\end{align*}
\item For $s,r,t$ distinct reals with $r$ between $s$ and $t$, 
\begin{align*}
	\Ga(t,s)=\Ga(t,r)\Ga(r,s).
\end{align*}
\end{enumerate}
\end{thm}

\begin{rem}
\label{rem:no-causality-yet}
\begin{enumerate}
	\item The reader familiar with parabolic problems expects causality, that is, $\Ga(t,s)=0$ when $t<s$, and recovery of initial data, that is $\Ga(t,s)\psi\to \psi$ when $t\to s^+$, meaning that $\P_{s}^+=\Id$. At this stage however, there is yet no reason to expect these  properties since all assumptions apply indifferently to the equation and its adjoint, going backward in time. In view of this, it is remarkable that we can prove the adjointness property (ii) and the Chapman--Kolmogorov formula (v) under the mere assumption $\mathbf{(H_0)}$. 
	
	\item Properties (iii) and (iv) mean that the Green operators $\Ga(t,s)$  propagate the ranges of the limit operators $\P_{t}^+$ when $t>s$ and $\P_{t}^-$ when $t<s$ even though we do not have further information on the range and kernel of these operators at this point.
\end{enumerate}
\end{rem}

\begin{proof}

We proceed as follows.

\medskip

\noindent{\emph{Proof of (i).}} We may assume $s=0$ as usual. The property $\P_{0}^+ - \P_{0}^-=\Id$ is a rephrasing of the jump relation \eqref{eq:split} and similarly $\tP_{0}^+ - \tP_{0}^-=-\Id$, the negative sign coming from the fact that  $-\partial_{t}+\Ltstar   $ is backward in time.  Fix $\psi,\tpsi\in \L^2_{x}$. We apply the integral identity of Corollary~\ref{cor:energy2} to 
$u\coloneqq\Ga(\cdot, 0)\psi$ and $\tilde u\coloneqq \tGa(\cdot, 0)\tpsi$ in the intervals $(0,\infty)$ and $(-\infty,0)$, knowing that 
the integrand vanishes almost everywhere in each interval and that $\angle {u(t)}{\tilde u(t)}\to 0$ in the limit as $t\to \pm\infty$. This gives us
$$
\angle {u(0^+)}{\tilde u(0^+)}=0= \angle {u(0^-)}{\tilde u(0^-)},
$$
that is,
$$
\angle {\P_{0}^+\psi}{\tP_{0}^+\tpsi}=0= \angle {\P_{0}^-\psi}{\tP_{0}^-\tpsi}.
$$
The relations above yield
$$
\angle {\P_{0}^+\psi}{\tpsi}= \angle {\P_{0}^+\psi}{\tP_{0}^-\tpsi}= \angle {\psi}{\tP_{0}^-\tpsi},
$$
which shows that $(\P_{0}^+)^*=\tP_{0}^-$. Since $\P_{0}^+ - \P_{0}^-=\Id$, also $(\P_{0}^-)^*=\tP_{0}^+$ follows.

\medskip
   
\noindent{\emph{Proof of (iii) and (ii) for $t>s$.}} Fix $\psi,\tpsi\in \L^2_{x}$. This time we can apply the integral identity of Corollary~\ref{cor:energy2} to 
$u\coloneqq\Ga(\cdot, s)\psi$ and $\tilde u\coloneqq \tGa(\cdot, t)\tpsi$ in the intervals  $(-\infty,s)$, $(s,t)$ and $(t,\infty)$, knowing that 
the integrand vanishes almost everywhere in each interval and that $\angle {u(\tau)}{\tilde u(\tau)}\to 0$ in the limit as $\tau\to \pm\infty$. We obtain
 \begin{align*}
 \angle {u(t)}{\tilde u(t^+)}&=0,
 \\
\angle {u(t)}{\tilde u(t^-)}&= \angle {u(s^+)}{\tilde u(s)},
\\
0&= \angle {u(s^-)}{\tilde u(s)},
\end{align*}
that is,
\begin{align*}
 \angle {\Ga(t,s)\psi}{\tP_{t}^+\tpsi}&=0,
 \\
\angle {\Ga(t,s)\psi}{\tP_{t}^-\tpsi}&= \angle {\P_{s}^+\psi}{\tGa(s,t)\tpsi},
\\
0&= \angle {\P_{s}^-\psi}{\tGa(s,t)\tpsi}.
\end{align*}
Subtracting the first and third equalities to the second one and using (i), we obtain 
\begin{equation}
\label{eq:propagatorsadjoint}
\angle {\Ga(t,s)\psi}{\tpsi}= \angle {\psi}{\tGa(s,t)\tpsi},
\end{equation}
which proves (ii) in this case. From the adjoint relations in (i) we also see that
$\P_{t}^-\Ga(t,s)\psi=0$ and $ \Ga(t,s)\P_{s}^-\psi=0$. Again by (i) we conclude for the two missing relations in (iii).

\medskip

\noindent{\emph{Proof of (iv) and (ii) for $s>t$.}} The argument is completely symmetric to the previous case.

\medskip

\noindent{\emph{Proof of (v)}.}
Let us first treat the case $s<r<t$. Fix $\psi\in \L^2_{x}$. Let $u\coloneqq\Ga(\cdot, s)\psi$, $v\coloneqq\Ga(\cdot, r)(u(r))$ and define 
\begin{equation*}
w\coloneqq
\begin{cases} v, & \textrm{on}\  (r,\infty)\times\R^n\\
u, & \textrm{on}\ (-\infty, r)\times\R^n
\end{cases}\ .
\end{equation*}
Since the gluing procedure preserves  $ \Deldot^{ r_{1}, q_{1}}$, the equality $w=u$  follows  by uniqueness provided we can show that $w$ is a  $\Deldot^{r_{1},q_{1}}$-solution of 
\begin{align*}
  \partial_{t}w+\Lt w= \delta_{s}\otimes \psi.
\end{align*}
By translation we can assume $r=0$ in order to simplify the exposition. Let $\tphi \in \cD'(\R^{n+1})$. The argument is a reprise of the proof of the jump relation in Theorem~\ref{thm:deltasL2}. Using the same cut-off functions $\theta_{\varepsilon}$ supported outside of $[-\varepsilon, \varepsilon]$, we have, as before,
\begin{align*}
-\Angle {w} {\partial_{t}\tphi} + \Angle {\Lt w}{\tphi } 
 &= \lim_{\varepsilon\to 0} -\Angle {w} {(\partial_{t}\tphi)\, \theta_{\varepsilon})} + \Angle {\Lt w}{\tphi \,\theta_{\varepsilon}} 
 \\
 &=\lim_{\varepsilon\to 0}  - \Angle {w} {\partial_{t}(\tphi\, \theta_{\varepsilon})} + \Angle {\Lt w}{\tphi \,\theta_{\varepsilon}} + \Angle {w} {\tphi\, (\theta_{\varepsilon})'}.
\end{align*}
For the first two terms we see that $\tphi\, \theta_{\varepsilon}$ is the sum of one test function $\tphi_{+}$ supported in $(0,\infty)\times \R^n$ and another one $\tphi_{-}$ supported in $(-\infty,0)\times \R^n$. With the first one we use the equation for $v$ and with the second the equation for $u$, so that we find 
$$
-\Angle {w} {\partial_{t}(\tphi\, \theta_{\varepsilon})} + \Angle {\Lt w}{\tphi \,\theta_{\varepsilon}}= 0+ \angle {\psi}{(\tphi_{-})(s)}= \angle {\psi}{\tphi (s)},
$$
where the second step works provided $\varepsilon$ is small enough in order to guarantee that $\theta_{\varepsilon}(s)=1$.  
For the third term we can argue as in Step 3 of  the proof of Theorem~\ref{thm:deltasL2} to find 
$$
\lim_{\varepsilon\to 0}  \Angle {w} {\tphi\, (\theta_{\varepsilon})'}= \angle {v(0^+)} {\tphi(0)} - \angle {u(0)} {\tphi(0)}.
$$
By definition and (iii), we have 
$$v(0^+)=\P_{0}^+(u(0))= \P_{0}^+\Ga(0,s)\psi= \Ga(0,s)\psi=u(0),$$
so that altogether we have shown that
$$
-\Angle {w} {\partial_{t}\tphi} + \Angle {\Lt w}{\tphi } = \angle {\psi}{\tphi (s)}=\Angle {\delta_{s}\otimes \psi}{\tphi}$$
as desired. The proof when $t<r<s$ is similar, using (iv) instead of (iii).
\end{proof}

\begin{rem} 
\label{rem:weak continuity} The adjointness property implies that $s\mapsto \Ga(t,s)$ is weakly continuous in $\L^2_{x}$ on $\R\setminus \{t\}$. Similarly,  $t\mapsto \tGa(s,t)$ is weakly continuous in $\L^2_{x}$ on $\R\setminus \{s\}$.
 \end{rem}

\subsection{Representation with Green operators}
\label{sec:schwartz}

In this section we shall detail how the Green operators can be seen as operator-valued Schwartz kernels for the inverse of $\cH$. This illustrates nicely how we can re-discover objects of the classical theory for smooth coefficients as part of our `universal' construction. While Proposition~\ref{prop:kernel}  is not used {in other sections}, a key ingredient in its proof implies representations for solutions (Theorem~\ref{thm:representation}) that will be useful  when we deal with Cauchy problems. 

The  operator-valued Schwartz kernels result  is as follows.

\begin{prop}
 \label{prop:kernel} Assume $\mathbf{(H_0)}$. For any $f,\tf\in \cD(\R)$, $\psi,\tpsi\in \cD(\R^{n})$,
\begin{equation}
\label{eq:kernel}
\Angle {\cH^{-1}( f\otimes \psi)}{\tf\otimes \tpsi}= \iint_{\R^2}f(s) \angle {\Ga(t,s)\psi}{\tpsi}\overline\tf(t)\, \d s\d t.
\end{equation}
\end{prop}

Before we come to the proof, let us interpret the result in the context of the Schwartz kernel representation~\cite{Schwartz}. Indeed, there exists a unique $K\in \cD'(\R^{n+1}\times \R^{n+1})$ such that 
$$
\Anglep {\cH^{-1}(f\otimes \psi)}{\tf \otimes \tpsi}= (K_{t,x,s,y}, f_{s}\otimes \psi_{y}\otimes \tf_{t}\otimes \tpsi_{x}).
$$
We indicate the dummy variables for notational simplicity and the bracket with parentheses are the bilinear dualities. For example, building $\cH$ from the heat operator $\partial_{t}-\Delta$, we see that $K_{t,x,s,y}$ can be identified with the heat kernel 
$$1_{\{t>s\}} \frac{1}{(4\pi (t-s))^{n/2}}\ \e^{-\frac{|x-y|^2}{4(t-s)}}.$$
Not all such operators may have kernels with pointwise bounds. In any case, we can also proceed by fixing $\psi,\tpsi\in \cD(\R^n)$ and looking at the bilinear map $(f,\tf)\mapsto \Anglep {\cH^{-1}(f\otimes \psi)}{\tf \otimes \tpsi}$ on $\cD(\R)\times \cD(\R)$. Again, the Schwartz kernel theorem provides a unique distribution $K_{\psi,\tpsi}\in \cD'(\R^2)$ such that 
$$ \Anglep {\cH^{-1}(f\otimes \psi)}{\tf \otimes \tpsi}= (K_{\psi,\tpsi,t,s}, f_{s}\otimes \tf_{t}).$$
Thus, Theorem~\ref{prop:kernel}   establishes that $K_{\psi,\tpsi,t,s}$ can be identified with a locally integrable  function and  we can set 
$$K_{\psi,\tpsi,t,s} \coloneqq \anglep{\Ga(t,s)\psi}{\tpsi} \quad \textrm{on\ }\R^2\setminus\{(t,t):  t\in \R\},$$
the values on the diagonal being irrelevant. In particular, $K_{\psi,\tpsi,t,s}$ agrees with a separately continuous function on $\R^2\setminus\{(t,t):  t\in \R\}$  that vanishes if $|s|$ or  $|t|$ tend to $\infty$.
	
In order to prove Proposition~\ref{prop:kernel}, we begin with a pointwise variant for all $t \in \R$. 

\begin{lem}
\label{lem:representation} Assume $\mathbf{(H_0)}$. For any $f\in \cD(\R)$, $t\in \R$  
and $ \psi,\tpsi \in \cD(\R^n)$, 
\begin{equation}
\label{eq:kernelt}
\angle {\cH^{-1}(f\otimes \psi)(t)}{\tpsi}= \int_{\R} f(s)\angle {\Ga(t,s)\psi}{\tpsi}\, \d s.
\end{equation}
\end{lem}

\begin{proof} 
In order to simplify the exposition, we assume $t=0$. The general case follows by translation as usual. 

Let $(\varphi_{\varepsilon})$ be a standard mollifying sequence in the $t$-variable. Since $\Hstarinverse$ is the adjoint of $\cH^{-1}$ and $f \otimes \psi$, $\varphi_{\varepsilon} \otimes \tpsi$ belong to $\cVdot$, we have
\begin{align*}
	\Angle {\cH^{-1}(f\otimes \psi)}{\varphi_\varepsilon \otimes \tpsi}
	= \Angle {f\otimes \psi}{\Hstarinverse(\varphi_\varepsilon \otimes \tpsi)}.
\end{align*}
Since  $f \otimes \psi$ and $\varphi_{\varepsilon} \otimes \tpsi$ are in $\L^{1}_t \L^{2}_x$, we know from Theorem~\ref{thm:L1L2} (and its proof) that $\cH^{-1}(f \otimes \psi)$  and $\Hstarinverse(\varphi_\varepsilon \otimes \tpsi)$ belong to $ \C_0(\L^2_x)$. Both duality pairings are given by absolutely convergent Lebesgue integrals and we can apply Fubini's theorem on the right-hand side in order to write
\begin{align*}
		\Angle {\cH^{-1}(f\otimes \psi)}{\varphi_\varepsilon \otimes \tpsi} = \int_{\R}{f(s)} \angle {\psi}{( \Hstarinverse(\varphi_\varepsilon \otimes \tpsi)(s)}\, \d s.
\end{align*}
For fixed $s$ we use \eqref{eq:G vs Ts} in order to arrive at
\begin{align}
\label{eq:repofH-1}
	\Angle {\cH^{-1}(f\otimes \psi)}{\varphi_\varepsilon \otimes \tpsi}
	&= \int_{\R}{f(s)} \Angle {\Ga(\cdot,s)\psi}{\varphi_\varepsilon \otimes \tpsi}\, \d s.
\end{align}
Now, we pass to the limit as $\varepsilon \to 0$ as follows.

Since $\cH^{-1}(f\otimes \psi) \in \C_0(\L^2_x)$ as mentioned before, we see that the left-hand side of \eqref{eq:repofH-1} tends to the left-hand side of \eqref{eq:kernelt} as $\varepsilon \to 0$.

On the right-hand side we use the properties of the Green operators from Proposition~\ref{prop:Gamma}. Since $G(\cdot,s)\psi$ is continuous with values in $\L^2_x$ except at $s$, we have pointwise convergence
\begin{align*}
	\lim_{\varepsilon\to 0} \Angle {\Ga(\cdot,s)\psi}{\varphi_\varepsilon \otimes \tpsi} =  \angle {\Ga(0,s)\psi}{\tpsi}
\end{align*}
for all $s \neq 0$. Moreover, $G(\cdot,s)\psi$ is bounded on $\R \setminus \{s\}$ with  values in $\L^2_x$ and the duality pairing on the right-hand side of \eqref{eq:repofH-1} is another convergent Lebesgue integral that obeys the estimate
\begin{align*}
	| \Angle {\Ga(\cdot,s)\psi}{\varphi_\varepsilon \otimes \tpsi} |
	\leq  \|\Ga(\cdot,s)\psi\|_{\L^\infty_t \L^2_x} \|\varphi_\varepsilon \otimes \tpsi\|_{\L^1_t \L^2_x} \leq C \|\psi\|_{\L^2_x} \|\tpsi\|_{\L^2_x}
\end{align*}
with $C$ independent of $s$ and $\varepsilon$. Since $f$ is integrable, the right-hand side of \eqref{eq:repofH-1} tends to the right-hand side of \eqref{eq:kernelt} as $\varepsilon \to 0$ by dominated convergence.
\end{proof}

\begin{proof}[Proof of Proposition~\ref{prop:kernel}]
We use  the continuity of $t\mapsto \cH^{-1}(f\otimes \psi)(t)$ in $\L^2_{x}$ to rewrite the left-hand side of \eqref{eq:kernel} as a Lebesgue integral and use \eqref{eq:kernelt} for the integral in $x$ in order to obtain
\begin{align*}
	\Angle {\cH^{-1}(f\otimes \psi)}{\tf \otimes \tpsi}&= \int_{\R}\angle {\cH^{-1}(f\otimes \psi)(t)}{\tpsi}\overline\tf(t)\, \d t \\
	&= \int_{\R}  \bigg(\int_{\R} f(s)\angle {\Ga(t,s)\psi}{\tpsi}\, \d s\bigg)\overline\tf(t)\, \d t.
\end{align*}
The remaining question is therefore whether the above iterated integral can be taken in the sense of Lebesgue on $\R^2$, so that we can use Fubini's theorem to conclude the proof of \eqref{eq:kernel}. 
	
Since $(s,t)\mapsto \angle {\Ga(t,s)\psi}{\tpsi}$ is separately continuous on $\R^2\setminus\{(t,t)\, : \,  t\in \R\}$, it is a (Borel) measurable function on $\R^2\setminus\{(t,t)\, : \,  t\in \R\}$ and we can consider it as an almost everywhere  defined measurable function on $\R^2$. Finally, uniform boundedness of the Green operators in $\L^2_x$ yields
\begin{equation}
\label{eq:|kernel|}
\iint_{\R^2} \big| f(s)\angle {\Ga(t,s)\psi}{\tpsi}\overline\tf(t) \big|\, \d s\d t \le C \|f\|_{\L^1_{t}}\|\tf\|_{\L^1_{t}}\|\psi\|_{\L^2_{x}}\|\tpsi\|_{\L^2_{x}} . \qedhere
\end{equation}
%and we may apply Fubini's theorem.
\end{proof}

Lemma~\ref{lem:representation}  in turn implies a representation formula for solutions to equations with general right-hand side.

\begin{thm}
\label{thm:representation}
Assume $\mathbf{(H_0)}$.  
Let $(r,q)$ be an admissible pair or $(r,q)=(\infty, 2)$,  $F\in \LL$, $g\in \L^{r'}_{t}\L^{q'}_{\vphantom{t} x}$, $s\in \R$ and $\psi \in \L^2_{x}$. Then the $\L^2_{x}$ value at time $t$ of the unique  $\Deldot^{r_{1},q_{1}}$-solution $u $ of 
\begin{equation}
\partial_{t}u+\Lt u= \delta_{s}\otimes \psi -\div F+g
\end{equation}
obtained by combining Theorems~\ref{thm:Fg}, \ref{thm:L1L2} and \ref{thm:deltasL2},  can be represented by the  equality 
\begin{equation}
\label{eq:representation}
u(t)= \Ga(t,s)\psi- \int_{\R}\Ga(t,\sigma)(\div F(\sigma))\, \d \sigma + \int_{\R}\Ga(t,\sigma)g(\sigma)\, \d \sigma
\end{equation}
when $t\ne s$ for the first term and where the integrals are defined in the weak sense, that is,
\begin{equation}
\label{eq:weaksense}
\angle{u(t)}{\tpsi}= \angle{\Ga(t,s)\psi}{\tpsi}+ \int_{\R}\angle{F(\sigma)}{\nabla\tGa(\sigma,t)\tpsi}\, \d \sigma + \int_{\R}\angle{g(\sigma)}{\tGa(\sigma,t)\tpsi}\, \d \sigma
\end{equation}
for all $\tpsi \in \L^2_{x}$.
\end{thm}

\begin{proof} We prove \eqref{eq:weaksense}. By uniqueness, it suffices to consider the three terms in the right-hand side of the equation individually, assuming the other two vanish. Fix $t\in \R$.

For the term involving $\psi$, this is the definition of $\Ga(t,s)\psi$ if $t\ne s$. 
	
For the term involving $g$, Lemma~\ref{lem:representation}  and the adjoint relation $\tGa(\sigma,t)= \Ga(t,\sigma)^*$   yield the result when $\tpsi$ is a test function and $g = f \otimes \psi$ a tensor product of test functions (or a linear combinations of such tensor products). Hence, $\tpsi$ and $g$ already describe dense subsets of $\L^2_x$ and $\L^{r'}_{t}\L^{q'}_{\vphantom{t} x}$, respectively.  Consider  $g\in \L^{r'}_{t}\L^{q'}_{\vphantom{t} x}$ and $\tpsi \in \L^2_{x}$.  If $(r,q)$ is admissible, then \eqref{eq:uLinftyL2} in Theorem~\ref{thm:Fg}  and  Cauchy-Schwarz yield 
\begin{align*}
	|\angle{(\cH^{-1}g)(t)}{\tpsi}| \leq C(n,q,r)  \|g\|_{\L^{r'}_{t}\L^{q'}_{\vphantom{t} x}} \|\tpsi\|_{\L^2_x}
\end{align*} 
and by  \eqref{eq:udeltasL2toGrq} for the adjoint equation,
\begin{align*}
	\int_{\R} |\angle{g(\sigma)}{\tGa(\sigma,t)\tpsi}| \, \d \sigma 
 \leq C(n,q,r) \|g\|_{\L^{r'}_{t}\L^{q'}_{\vphantom{t} x}} \|\tpsi\|_{\L^2_x}.
\end{align*}
 If $(r,q)=(\infty,2)$,  then Theorem~\ref{thm:L1L2} yields
\begin{align*}
	|\angle{u(t)}{\tpsi}| \leq C(n,q_{1},r_{1})  \|g\|_{\L^{1}_{t}\L^{2}_{x}} \|\tpsi\|_{\L^2_x}
\end{align*} 
and by  \eqref{eq:uLinftydeltas} for the adjoint equation,
\begin{align*}
	\int_{\R} |\angle{g(\sigma)}{\tGa(\sigma,t)\tpsi}| \, \d \sigma 
 \leq C(n,q_{1},r_{1}) \|g\|_{\L^{1}_{t}\L^{2}_{x}} \|\tpsi\|_{\L^2_x}.
\end{align*} 
Under either assumption one can thus pass to the limit by density in \eqref{eq:kernelt}, and \eqref{eq:weaksense} is proved in this case.

For the term involving $F$ the proof is analogous. As before,  for $F\in \LL$ and $\tpsi\in \L^2_{x}$, 
\begin{align*}
	|\angle{\cH^{-1}(\div F)(t)}{\tpsi}| \leq  C \|F\|_{\LL} \|\tpsi\|_{\L^2_x}
\end{align*}
and by construction of $\tGa$ and Theorem~\ref{thm:deltasL2}, 
$\nabla  \tGa(\cdot,t)\tpsi \in \LL$
and $$
\iint_{\R^{n+1}} {|F(\sigma,y)|}{|\nabla\tGa(\sigma,t)\tpsi(y)|}\, \d \sigma\d y \le C(n,q,r) \|F\|_{\LL}\|\tpsi\|_{\L^2_{x}}.
$$
Hence, the integral involving $F$ on the right-hand side of \eqref{eq:weaksense} is defined and, by density,  it remains to check that it agrees with  $-\angle {\cH^{-1}( \div F)(t)}{\tpsi}$  when $\tpsi$ is a test function and $F$ is a tensor product $f\otimes \psi$ of test functions. But in this case we have $\div F= f\otimes \div \psi$ and Lemma~\ref{lem:representation} along with the adjointness relations for $\Ga$ yields
$$-\angle {\cH^{-1}( \div F)(t)}{\tpsi}= -\int_{\R} f(\sigma)\angle {\Ga(t,\sigma)\div\psi}{\tpsi}\, \d \sigma= \int_{\R} \angle {F(\sigma)}{\nabla \tGa(\sigma,t)\tpsi}\, \d \sigma $$
as required.
\end{proof}

\begin{rem} 
We draw the reader's attention to the following regularity result that is implicit from the equality proved in Theorem~\ref{thm:representation}. The integral involving $g$  is   continuous in $\L^2_{x}$  as a function of $t$, while its definition merely yields  continuity for the weak $\L^2_{x}$ topology. The same comment applies to the integral  involving $F$. 
\end{rem}

\subsection{Invertibility, Causality}
\label{sec:invertibility}

So far, we have not addressed sufficient conditions for invertibility of $\cH$ and the question of causality. The true use of the space $\cVdot$ will become transparent here. The first two results require a smallness assumption on the lower order terms, hence are of perturbative nature from the purely second order case. The third one {is non-perturbative and} uses lower bounds.

At this stage, we eventually impose ellipticity to $A$  in the sense of G\aa rding : for some $\lambda>0$ we assume that for all $u \in \cVdot$,
\begin{equation}
\label{eq:Aelliptic}
\Re \int_{\R} \angle{A(t)\nabla u(t)}{\nabla u(t)}\, \d t \ge \lambda  \|\nabla u\|_{\LL}^2.
\end{equation}
Note that this is equivalent to having for almost every $t$ and every $w\in \Hdot^1(\R^n)$ that
$$
\Re \angle{A(t)\nabla w}{\nabla w } \ge \lambda  \|\nabla w\|_{\L^2_{x}}^2.
$$
This lower bound would also be the one to assume for systems. 
When $A(t)$ is a matrix with real measurable entries with respect to $x$, this is also equivalent  to a pointwise lower bound for $A(t,x)$, see~\cite{Hendrik}. However, this last observation is not valid for complex matrices or systems. 

\begin{thm}[Invertibility]
\label{thm:invertible} Assume that $A$ is elliptic and bounded with parameters $\lambda,\Lambda$ as in \eqref{eq:Abdd}, \eqref{eq:Aelliptic}. There is  $\varepsilon_{0}>0$ small enough depending on $\lambda, \Lambda, n, q_{1}, r_{1}$ such that 
$P_{\tilde r_{1},\tilde q_{1}}\le \varepsilon_{0}$ implies that $\cH$ is invertible. 
\end{thm}

\begin{proof} For the sesquilinear form $\beta$ corresponding to the lower order coefficients, we have seen in Lemma~\ref{lem:beta} that
$$
|\Angle {\beta u}{v}| \le P_{\tilde r_{1},\tilde q_{1}} \|u\|_{\Deldot^{ r_{1}, q_{1}}}\|v\|_{\Deldot^{ r_{1}, q_{1}}}.
$$
By the embedding $\cVdot \hookrightarrow  \Deldot^{ r_{1}, q_{1}}$, we have with $C= C(n, r_{1}, q_{1})$, 
$$
|\Angle {\beta u}{v}| \le C P_{\tilde r_{1},\tilde q_{1}} \|u\|_{\cVdot}\|v\|_{\cVdot}.
$$
Next, we let
$$\cH_{0}\colon \cVdot \to \cVdot',  \quad \cH_{0} u= \partial_{t}u-\div(A \nabla u)
$$
be of the same type as $\cH$ but without lower order terms. We use the classical hidden coercivity inequality
$$
\Re \Angle {\cH_{0}u}{(\Id+\delta \H_{t})u} \ge  \delta \|\D_{t}^{1/2}\!u\|_{\LL}^2 + (\lambda-\delta\Lambda) \|\nabla u\|_{\LL}^2,
$$
where $\delta>0$ and $\H_{t}$ is the Hilbert transform in the $t$-variable with symbol $i\tau/|\tau|$, see~\cite{Kaplan}. This inequality follows from the factorization $\partial_{t}= \D_{t}^{1/2}\H_{t}\D_{t}^{1/2}$, so that, using also the  
skew-adjointness of the Hilbert transform and commutation, 
\begin{align*}
\Re \Angle {\partial_{t}u} {(\Id+\delta \H_{t})u}
&= 
\int_{-\infty}^\infty \Re \angle {\H_{t}\D_{t}^{1/2}\!u(t)}{\D_{t}^{1/2}(\Id+\delta \H_{t})u(t)}\, \d t \\
&= \delta\|\H_{t}\D_{t}^{1/2}\! u\|_{\LL}^2
= \delta\|\D_{t}^{1/2}\! u\|_{\LL}^2.
\end{align*}
Altogether, we get
$$
\Re \Angle {\cH u}{(\Id+\delta \H_{t})u} \ge \delta \|\D_{t}^{1/2}\!u\|_{\LL}^2 + (\lambda-\delta\Lambda) \|\nabla u\|_{\LL}^2 - C P_{\tilde r_{1},\tilde q_{1}} \sqrt{1+\delta^2} \|u\|_{\cVdot}^2.
$$
As $\|u\|_{\cVdot}^2= \|\D_{t}^{1/2}\!u\|_{\LL}^2 +  \|\nabla u\|_{\LL}^2$,  we can now set $\delta \coloneqq \lambda/(1+\Lambda)$ and define $\varepsilon_0$ through $C\varepsilon_{0} \sqrt{1+\delta^2}= \delta/2$ in order to conclude that $P_{\tilde r_{1},\tilde q_{1}}\le \varepsilon_{0}$ implies that
\begin{equation}
\label{eq:lowerbound}
\Re \Angle {\cH u}{(\Id+\delta \H_{t})u}  \ge \frac{\delta}{2}  \|u\|_{\cVdot}^2.
\end{equation}
{It follows from the Lax--Milgram lemma that $(\Id+\delta \H_{t})^*\cH$ is invertible from $\cVdot $ to $\cVdot'$.}
As  $\Id+\delta \H_{t}$ is also invertible on $\cVdot$ and its dual, this proves the invertibility of $\cH$. \end{proof}

\begin{thm}[Causality]
\label{thm:causality}
Assume that $A$ is elliptic and bounded with parameters $\lambda,\Lambda$ as in \eqref{eq:Abdd}, \eqref{eq:Aelliptic}. There is  $\varepsilon_{0}>0$ small enough depending on $\lambda, \Lambda,  n, q_{1}, r_{1}$ such that $P_{\tilde r_{1},\tilde q_{1}}\le \varepsilon_{0}$ implies that $\cH$ is causal  in the following sense:
\begin{enumerate}
\item If $u$ is  a  $\Deldot^{r_{1},q_{1}}$-solution of 
$\partial_{t}u+\Lt u=-\div F+g$ as in Theorems~\ref{thm:Fg} or \ref{thm:L1L2} or combination of both, and if $F, g$ vanish  on $(-\infty, s)\times \R^n$ for some $s\in \R$, then $u=0$ in $(-\infty,s]\times \R^n$.
\item  If  $u $ is  a $\Deldot^{r_{1},q_{1}}$-solution of 
$\partial_{t}u+\Lt u=\delta_{s}\otimes \psi$ as in Theorem~\ref{thm:deltasL2}, then $u=0$ in $(-\infty,s)\times \R^n$.
\end{enumerate}
\end{thm}
 
\begin{proof} We begin with the proof of (i). We know that $u\in \C_{0}(\L^2_{x})$. As usual, we may assume $s=0$ to simplify the exposition. We let $S\coloneqq \sup_{t \le 0}\|u(t)\|_{\L^2_{x}}^2$. The integral identities of  Lemma~\ref{lem:energy} {apply to $u$}  and for $ \sigma\le \tau\le 0$ we obtain 
$$
\|u(\tau)\|_{\L^2_{x}}^2-\|u(\sigma)\|_{\L^2_{x}}^2= 2\Re \int_{\sigma}^\tau -\angle{A(t)\nabla u(t)}{\nabla u(t)} - \angle{\beta u(t)}{u(t)}\, \d t
$$
{since $F,g$ vanish in this range.}
We send $\sigma\to -\infty$ and take some $\tau\le 0$ at which the supremum $S$ is attained. Then, we have 
$$
S\le -2\lambda I + 2\int_{-\infty}^\tau | \angle{\beta u(t)}{u(t)}| \, \d t,
$$
where  $I \coloneqq \int_{-\infty}^\tau \|\nabla u(t)\|_{\L^2_{x}}^2\, \d t. $
The standard mixed  estimates, taking into account integration on $(-\infty,\tau)$ in $t$,
give us
$$
\|u\|_{\L^{r_{1}}(-\infty,\tau; \L^{q_{1}}_{x})} \le C S^{\frac{1}{2}-\frac{1}{ r_{1}}}I^{\frac{1}{ r_{1}}},
$$
see 
{Section~\ref{sec:GN} for a very quick proof.}
Using the pair $(\tilde r_{1}, \tilde q_{1})$ for the lower order coefficients and H\"older's inequality as in the proof of Lemma~\ref{lem:beta}, we get
$$
\int_{-\infty}^\tau | \angle{\beta u(t)}{u(t)}|\, \d t \le  C'P_{\tilde r_{1},\tilde q_{1}}(S^{\frac{1}{2}-\frac{1}{ r_{1}}}  I^{\frac{1}{2}+\frac{1}{r_{1}}} + S^{1-\frac{2}{ r_{1}}} I^{\frac{2}{ r_{1}}}).
$$
Altogether,
\begin{align*}
S
&\le -2\lambda I+ 2C'P_{\tilde r_{1},\tilde q_{1}}(S^{\frac{1}{2}-\frac{1}{ r_{1}}}  I^{\frac{1}{2}+\frac{1}{r_{1}}} + S^{1-\frac{2}{ r_{1}}} I^{\frac{2}{ r_{1}}}) \\
&\leq -2\lambda I+ 2C'P_{\tilde r_{1},\tilde q_{1}}((\tfrac{3}{2}-\tfrac{3}{r_{1}})S+ (\tfrac{1}{2}+\tfrac{3}{r_{1}}) I),
\end{align*}
where the second step is by Young's inequality, keeping in mind that $2\le  r_{1}< \infty$. If $P_{\tilde r_{1},\tilde q_{1}}\le \varepsilon_{0}$ is small enough, then we can hide the contribution of $S$ on the left and obtain that $S\le 0$. Hence, $u(t)=0$ for all $t\le 0$. 

For the proof of (ii), we know that $u\in \C_{0}(-\infty,s; \L^2_{x})$ with $\L^2_{x}$ limit when $t\to s^-$. In particular, we can argue with $S= \sup_{t \le s}\|u(t)\|_{\L^2_{x}}^2$ where $u(s)$ means $u(s^-)$ and the proof is the same. 
\end{proof}
 
We turn to lower bounds assumptions. We assume $P_{\tilde r_{1}, \tilde q_{1}}$ finite but not necessarily small. Here, we do not explicitly need the lower bound \eqref{eq:Aelliptic} on $A$ but it is hidden in checking the assumptions.
\begin{thm} [Invertibility through lower bounds]
\label{th:lb} \
\begin{enumerate}
\item Assume that there exists $c>0$ such that
$$\Re \Angle {\Lt  u}{u}  \ge c  \|\nabla u\|_{\LL}^2$$ for all $u\in \cVdot$.
Then $\cH$ is invertible.
\item Assume that  $$ \Re \angle {\Lt  w}{w}  \ge 0$$ almost everywhere for all $w\in \Hdot^1(\R^n)$.
Then $\cH$ is causal  in the sense of   Theorem~\ref{thm:causality}.\end{enumerate}
\end{thm}

\begin{proof} To prove (i),  arguing as before and using the assumption, we have
$$
\Re \Angle {\cH u}{(\Id+\delta \H_{t})u} \ge \delta \|\D_{t}^{1/2}\!u\|_{\LL}^2 + c \|\nabla u\|_{\LL}^2 - C\delta  \|u\|_{\Deldot^{ r_{1},  q_{1}}}\|\H_{t}u\|_{\Deldot^{ r_{1}, q_{1}}}
$$
with $C \coloneqq \|A\|_{\infty}  + P_{ \tilde r_{1}, \tilde q_{1}}.$ 

When $r_{1}=2$ (hence $n\ge 3$ and $q_{1}=\frac {2n}{n-2}$), the Sobolev inequality gives us  $$\|u\|_{\Deldot^{ r_{1},  q_{1}}} \le c(n,q_{1},r_{1})\|\nabla u\|_{\LL}.$$ Since the Hilbert transform is isometric on $\L^2$ and commutes with the gradient, we see that  if  $\delta>0$ is so small that $c-c(n,q_{1},r_{1})C\delta>0$, then $\cH$ is invertible.  

When $r_{1}>2$,  we refine the first embedding of Lemma~\ref{ref:embedding}, by replacing \eqref{eq:convexity} with 
\begin{equation*}
\|\D_{t}^{\theta/2}(-\Delta)^{(1-\theta) /2} \varphi\|_{\LL}^2
\le \varepsilon^{1/\theta} \theta \|\D_{t}^{1/2} \!\varphi\|_{\LL}^2 + \varepsilon^{-1/(1-\theta)}(1-\theta)\|(-\Delta)^{1 /2} \varphi\|_{\LL}^2
\end{equation*}
for $\varepsilon>0$. As $\|(-\Delta)^{1 /2} \varphi\|_{\LL}= \|\nabla \varphi\|_{\LL}$, we get 
\begin{equation*}
\|u\|_{\L^{r_{1}}_{t}\L^{q_{1}}_{x}}^2  \le c(n,q_{1},r_{1}) \Big(\varepsilon^{1/\theta} \theta
\|\D_{t}^{1/2}\! u\|_{\LL}^2 +\varepsilon^{-1/(1-\theta)}(1-\theta)  \|\nabla u\|_{\LL}^2 \Big).
\end{equation*}
Using that $\H_{t}$ is isometric on $\cVdot$,  we obtain
\begin{align*}
 \Re \Angle {&\cH u}{(\Id+\delta \H_{t})u} \\ 
 &\ge \delta \|\D_{t}^{1/2}\! u\|_{\LL}^2 + c \|\nabla u\|_{\LL}^2\\
 & \quad - C'\delta \Big((\varepsilon^{1/\theta} \theta
\|D_{t}^{1/2}\! u\|_{\LL}^2 + \varepsilon^{-1/(1-\theta)}(1-\theta) \|\nabla u\|_{\LL}^2+ \|\nabla u\|_{\LL}^2 \Big)
\end{align*}
with $C'=c'(n,q_{1},r_{1})C$.   One chooses first $\varepsilon$ with $0<C' \varepsilon^{1/\theta} \theta<1$ and then $\delta$ with $ 0<C'((1-\theta) \varepsilon^{-1/(1-\theta)}+1)\delta < c$.  This yields the desired invertibility for $\cH$. 

The proof of (ii) is  easy.  Let $u$ be a $\Deldot^{r_{1},q_{1}}$-solution of  $\partial_{t}u+\Lt u=\delta_{s}\otimes \psi -\div F+g$.  Then we know that for $\tau<s$, $$
\|u(\tau)\|_{\L^2_{x}}^2= -2\Re \int_{-\infty}^\tau \angle {\Lt  u(t)}{u(t)}\, \d t.
$$ We conclude right away that $u(\tau)=0$. 
\end{proof}

\begin{rem}
\label{rem:lb}
{In practice,} the hypothesis in (i) follows from elliptic inequalities of the form 
$\Re\angle{\beta w}{w}\le (1-\gamma)\Re\angle{A\nabla w}{\nabla w}$ with $\gamma<1$, when there is a lower bound $\lambda>0$ for $A$ as in \eqref{eq:Aelliptic}. It is not so much the smallness of $P_{\tilde r_{1},\tilde q_{1}}$ that matters (although its size can be used in proofs). 
\end{rem}

We obtain as a corollary the further identities for the Green operators that have been mentioned in Remark~\ref{rem:no-causality-yet} (i) earlier on.
 
\begin{cor} If the previous results on invertibility and causality under smallness assumptions or lower bounds hold, then $\Ga(t,s)=0$ if $t<s$ and  $\Ga(t,s)\to I$ strongly as $t\to s^+$. 
 \end{cor}

\begin{proof} In both cases, the solution $u$ of \eqref{eq:deltasL2} is given by $\Ga(\cdot,s)\psi$, $t\ne s$. Thus, $\Ga(t,s)=0$ if $t<s$ and $\lim_{t\to s^-}\Ga(t,s)\psi=0$. Hence, $\lim_{t\to s^+}\Ga(t,s)\psi=\psi$ follows from (i) in Theorem~\ref{thm:Gammats}.
 \end{proof}
 
 \subsection{Inhomogeneous assumptions on lower order terms}
\label{sec:inhomogneous}
 
So far, we have put ourselves in the situation where the lower order terms bring a contribution that is homogeneous to the gradient in $\L^2$. 
However, if the size of this contribution is not small enough, then invertibility of $\cH$ is not clear and we have to consider inhomogeneous assumptions by adding a positive constant.  

We define the inhomogeneous versions of $\cVdot$ and $\Deldot^{r, q}$. We set
 $ \cV \coloneqq \cVdot \cap \LL$  with norm
 $$
\|u\|_{\cV} \coloneqq \big(\|u\|_{\LL}^2+ \|\nabla u\|_{\LL}^2+ \|\D_{t}^{1/2}\!u\|^2_{\LL}\Big)^{1/2}
$$
and 
 $\Del^{r, q}: =\Deldot^{r, q} \cap \LL= \L^2_{t}\H^{1}_{x} \cap \L^{r}_{t}\L^{q}_{\vphantom{t} x}$
with norm 
$$
\|u\|_{\Del^{r, q}}: =\|u\|_{\LL}+ \|\nabla u \|_{\LL}+ \|u\|_{\L^{r}_{t}\L^{q}_{\vphantom{t} x}}.
$$
The continuous inclusion $\cV \hookrightarrow \Del^{r, q}$ for admissible pairs {follows from} Lemma~\ref{ref:embedding}. For admissible pairs, we still miss the extreme cases  $r=\infty$  that one obtains when $\L^{\infty}_{t}\L^2_{x}$ replaces $\H^{1/2}_{t}\L^2_{x}$,  and $q=\infty$.  The descriptions of the dual or pre-dual of $ \Del^{r, q}$  are similar. 

We may as well enlarge the class of coefficients  and assume from now on that 
\begin{equation}
\label{eq:sum} |\oa  |^2+|\ob  |^2+|a|\in \L^{\tilde r_{1}}_{t}\L^{\tilde q_{1}}_{\vphantom{t} x}+ \L^\infty_{t}\L^\infty_{x}
\end{equation}
with $(\tilde r_{1},\tilde q_{1})$ being a compatible pair for lower order coefficients, as in Section~\ref{sec:existenceanduniqueness}. Recall that this means $\frac 1 {\tilde r_{1}}+\frac {n}{2\tilde q_{1}}=1$ with $(\tilde r_{1},\tilde q_{1})\in (1,\infty]^2$. 
%As before, the pair  $(\tilde r_{1},\tilde q_{1})$ may vary from coefficient to coefficient as long as it remains   compatible for lower  order coefficients but we avoid making the distinction in the notation here, too. 

\begin{rem}[Subcritical exponents]
\label{rem:subcritical}
Let  $(\tilde r, \tilde q)\in [1,\infty]^2$ satisfy the subcritical compatibility relation $\frac 1 {\tilde r}+\frac {n}{2\tilde q}<1$. Coefficients with
 $$
 |\oa  |^2+|\ob  |^2+|a|\in \L^{\tilde r}_{t}\L^{\tilde q}_{\vphantom{t} x}
 $$
have such a decomposition when, in addition, 
\begin{align*}
\begin{cases}
	\tilde q<\infty & \text{if } n\ge 3 \\
	\max (\tilde r, \tilde q)<\infty & \text{if } n=2 \\
	\frac 1 {\tilde r}  > \frac 1 {2\tilde q}>0 & \text{if } n=1
\end{cases}.
\end{align*}
{Indeed, this is immediately seen from visualizing exponents in a $(\frac{1}{\tilde r}, \frac{1}{\tilde q})$-plane and truncating coefficients at a fixed height.}  In \cite{Ar68}, subcritical compatibility is assumed because the goal is to deal with bounded solutions. {The same condition} appears for Cauchy problems in \cite{LSU}.
See also \cite{KRW}. 
\end{rem}

We can define
$$\cH\colon \cV \to \cV', \quad \cH u= \partial_{t}u+\Lt u
$$
and $$\cH^*\colon \cV \to \cV', \quad \cH^* \tilde u = -\partial_{t}\tilde u+\Ltstar   \tilde u,
$$
where  $\cV'$ is the dual of $\cV$ with respect to $\LL$ duality. We use the same notation $\cH$ as before, {although $\cV$ now is a smaller space}.  These operators are bounded and adjoint to one another. Instead of $\mathbf{(H_{0})}$ we now work under the hypothesis that
$$
\mathbf{(H_{\kappa}) \qquad  there \ is \ \kappa>0 \ such \ that \ \cH +\kappa \ and \ \cH^*+\kappa \ are \ invertible.}
$$ 
Then we naturally work with $\cH +\kappa$, which means that we add the assumption $u\in \LL$ in most statements. Again the estimates will depend on $\Lambda=\|A\|_{\infty}$, {the bound implicit in} \eqref{eq:sum}, with fixed compatible pair $(\tilde r_{1},\tilde q_{1})$ for lower order coefficients, and the norm of the inverse of $\cH +\kappa$. Here is a description of changes in Sections \ref{sec:variationalspace} - \ref{sec:invertibility}:

\begin{enumerate}
\item In the modification of Corollary~\ref{cor:mainreg}  we assume $u\in \L^2_{t}\H^1_{x}$ and $\partial_{t}u\in \L^2_{t}\H^{-1}_{x}+ \L^{r'}_{t}\L^{q'}_{\vphantom{t} x}+\LL$. We obtain $u\in \cV \cap \C_{0}(\L^2_{x})$ when $(r,q)$ is an admissible pair or just  $u\in  \C_{0}(\L^2_{x})$  when $(r,q)=(\infty,2)$. The proofs of the technical lemmas use  estimates for the operator $\partial_{t}- \Delta+1$. (Note that we  still cannot use the Lions embedding theorem to deduce continuity because we do not, and do not want to, assume that $u\in \cV$.)

\item In Lemma~\ref{lem:energy}, we assume $u\in \L^2_{t}\H^1_{x}$ and $\partial_{t}u= -\div F+ g+ h$, where $F\in \L^2_{t}\L^2_{x}$, $g\in \L^{r'}_{t}\L^{q'}_{\vphantom{t} x} $ with $(r,q)$ an admissible pair or $(r,q)=(\infty,2)$ and $h\in \LL$. Then the conclusion is the same (without a constant) with the extra term $\angle{h(t)}{u(t)}$ in \eqref{eq:energy}.  The statements that follow in the same section are adapted similarly.

\item In all statements of Section~\ref{sec:existenceanduniqueness}, Assumption $\mathbf{(H_0)}$ is replaced by $\mathbf{(H_{\kappa})}$ and the equation to solve is for the operator $\partial_{t}+\Lt +\kappa$ in the sense of  $\Del^{r_{1},q_{1}}$-solutions, {which are defined by changing $\Deldot^{r,q}$ to $\Del^{r,q}$ in Definition~\ref{def:Delta-sol}. Such solutions belong to $\C_{0}(\L^2_{x})$.}
Uniqueness is in that class, and existence theorems through the inverse of $\cH +\kappa$ or its adjoint are proved for this operator with possible addition of an extra term  $h\in \LL$ in Theorem~\ref{thm:Fg}.

\item In Section~\ref{sec:propagators}, if we assume $\mathbf{(H_{\kappa})}$ instead of $\mathbf{(H_0)}$, we  obtain exactly the same statements for the Green operators 
$\Ga_{\kappa}(t,s)$ and $\tGa_{\kappa}(s,t)$ of $\partial_{t}+\Lt +\kappa$ and $-\partial_{t}+\Ltstar   +\kappa$, respectively.  In particular, they are uniformly bounded operators on $\L^2_{x}$ provided $t\ne s$.

\item Sections \ref{sec:schwartz}   can be adapted \textit{mutatis mutandis} with the Green operators $\Ga_{\kappa}$ under $\mathbf{(H_{\kappa})}$.  In the statement corresponding to Theorem~\ref{thm:representation}, one can add again an extra forcing term $h\in \LL$. 
\end{enumerate}

In order to check invertibility and causality,  we introduce the following property on the lower order coefficients, {where $\varepsilon\ge 0$ will be chosen appropriately small later}. 

\medskip

\paragraph{\bf Assumption $\mathbf{(D_{\varepsilon})}$}  For some   compatible pair $(\tilde r_{1},\tilde q_{1})$ for lower order  coefficients,   one can find a decomposition 
\begin{align*}
\oa  &= \oa_{0}+\oa_{\infty},
\\
\ob  &= \ob_{0}+\ob_{\infty},
\\
{a}&= {a_{0}}+{a_{\infty}},
\end{align*}
with 
\begin{align}
P_{\tilde r_{1},\tilde q_{1}}:&=\||\oa_{0}|^2\|_{\L^{\tilde r_{1}}_{t}\L^{\tilde q_{1}}_{\vphantom{t} x}}^{1/2}+ \|\ob_{0}|^2\|_{\L^{\tilde r_{1}}_{t}\L^{\tilde q_{1}}_{\vphantom{t} x}}^{1/2}+ \|a_{0}\|_{\L^{\tilde r_{1}}_{t}\L^{\tilde q_{1}}_{\vphantom{t} x}} \le \varepsilon,
\\
P_{\infty}:&= \||\oa_{\infty}|\|_{\infty}+ \||\ob_{\infty}|\|_{\infty}+ \|a_{\infty}\|_{\infty}^{1/2} <\infty.
\end{align}

\begin{rem}
\label{rem:Deps}
{By truncation at large height} one can always do such a decomposition {for any $\eps > 0$} starting from $|\oa  |^2+|\ob  |^2+|a|\in \L^{\tilde r_{1}}_{t}\L^{\tilde q_{1}}_{\vphantom{t} x}+ \L^\infty_{t}\L^\infty_{x}$, except if $\tilde r_{1}=\infty$. In this case, one needs further assumptions, such as that the part in $\L^\infty_{t}\L^{\tilde q_{1}}_{x}$  is uniformly continuous in time. For example,  independence of time is a valid hypothesis. In other words, $\mathbf{(D_{\varepsilon})}$ always holds for all $\varepsilon>0$ except when $\tilde r_{1}=\infty$. 
\end{rem}

The quantities $P_{\tilde r_{1},\tilde q_{1}}, P_{\infty}$ turn out to quantify nicely some estimates.  

\begin{thm}[Invertibility,  inhomogeneous case]
\label{thm:invertiblealpha}  Assume that $A$ has bounds  $\lambda,\Lambda$ as in \eqref{eq:Abdd}, \eqref{eq:Aelliptic} and that $\mathbf{(D_{\varepsilon})}$ holds for $\varepsilon>0$ small enough.  There exists  $\kappa_{0}>0$ large enough so that 
$\cH+\kappa$ is invertible from $\cV$ onto $\cV'$ for any $\kappa\ge \kappa_{0}$. Here, $\eps,\kappa_{0}$ and the lower bound  depend on $\lambda,\Lambda, n, q_{1},r_{1}$ and $\kappa_{0}$ depends additionally on $P_{\infty}$.
\end{thm}

\begin{proof} As in the proof of Theorem~\ref{thm:invertible} we may write with  obvious notation
$$
\cH+\kappa= \cH_{0}+ \beta_{0}+\beta_{\infty}+\kappa,
 $$
 so that 
 $$
|\Angle {\beta_{0}u}{v}| \le C\varepsilon \|u\|_{\cVdot}\|v\|_{\cVdot}
$$
and for the other term $|\Angle {\beta_{\infty}u}{v}| $ we have a simple bound by
$$
  \||\oa_{\infty}|\|_{\infty}\|u\|_{\LL}\|\nabla v\|_{\LL}+  \||\ob_{\infty}|\|_{\infty}\|v\|_{\LL}\|\nabla u\|_{\LL} + \|a_{\infty}\|_{\infty}\|u\|_{\LL}\| v\|_{\LL}.
$$
Hence, with $\delta>0$ and $\varepsilon=\varepsilon_{0}$ chosen as before, we have  
\begin{align*}
 \Re &\Angle {(\cH+\kappa )u}{(\Id+\delta \H_{t})u}  \ge \frac  \delta 2\  \|u\|_{\cVdot}^2
\\& -\sqrt{1+\delta^2} \big(P_{\infty}\|u\|_{\LL}\|\nabla u\|_{\LL} +P_{\infty}^2\|u\|_{\LL}^2\big)  + \kappa\|u\|_{\LL}^2. 
\end{align*}
By Young's inequality $ab \le  (1/4\gamma)a^2+ \gamma b^2$ with $\gamma=\delta/4$,  we see that 
  \begin{align*}
 \Re &\Angle {(\cH+\kappa )u}{(\Id+\delta \H_{t})u}  \ge \frac \delta 4  \|u\|_{\cVdot}^2 
 + \bigg(\kappa -  P_{\infty}^2\bigg(\frac{ (1+\delta^2)}{\delta}+ \sqrt{1+\delta^2}\bigg)\bigg)\|u\|_{\LL}^2.
\end{align*}
Hence, for, say,  $\kappa_{0} \coloneqq \frac \delta 4+ P_{\infty}^2\big(\frac{ (1+\delta^2)}{\delta}+ \sqrt{1+\delta^2}\big)$ and $\kappa\ge \kappa_{0}$, we get
\begin{equation}
\label{eq:lowerboundinh}
\Re \Angle {(\cH+\kappa )u}{(\Id+\delta \H_{t})u} \ge \frac \delta 4\  \|u\|_{\cV}^2
\end{equation}
 and invertibility follows.
\end{proof}

\begin{rem} The proof shows that 
 under the assumptions of Theorem~\ref{thm:invertible}, which corresponds to $P_{\tilde r_{1},\tilde q_{1}}\le \varepsilon_{0}$ and $P_{\infty}=0$, that  $\cH+\kappa\colon \cV\to \cV'$ is invertible for all $\kappa>0$. 
\end{rem}

For causality, Theorem~\ref{thm:causality} becomes the following result.

\begin{thm}[Causality, inhomogeneous case]
\label{thm:causalityalpha} With the assumptions of Theorem~\ref{thm:invertiblealpha} there exists $\kappa_{0}> 0$ such that $\cH+\kappa $ is causal for $\kappa\ge\kappa_{0}$ in the following sense:  
\begin{enumerate}
\item If $u $ is  a  $\Del^{r_{1},q_{1}}$-solution of 
$\partial_{t}u+\Lt u+\kappa u=-\div F+g+h$ as in the modifications of  Theorems~\ref{thm:Fg} or \ref{thm:L1L2} (here, $h\in \LL$, while $g\in \L^{r'}_{t}\L^{q'}_{\vphantom{t} x}$) and $F, g,h$ vanish  on $(-\infty, s)\times \R^n$ for some $s\in \R$, then $u=0$ identically on $(-\infty,s]\times \R^n$.
\item  If $u$ is  a $\Del^{r_{1},q_{1}}$-solution of 
$\partial_{t}u+\Lt u+\kappa u=\delta_{s}\otimes \psi$ as in the modification of Theorem~\ref{thm:deltasL2}, then $u=0$ identically on $(-\infty,s)\times \R^n$.
\end{enumerate}
 \end{thm}

\begin{proof} 
It begins as the proof of Theorem~\ref{thm:causality}. 
In the first case, as $u\in \C_{0}(\L^2_{x})$, we fix $s=0$ to simplify the exposition, let $S \coloneqq \sup_{t \le 0}\|u(t)\|_{\L^2_{x}}^2$, which is attained at some $\tau$, and obtain 
$$
S\le -2\lambda I + 2\int_{-\infty}^\tau | \angle {\beta_{0}u(t)}{u(t)}|+| \angle {\beta_{\infty}u(t)}{u(t)}|- \kappa\|u(t)\|_{\L^2_{x}}^2 \, \d t,
$$
where $I \coloneqq \int_{-\infty}^\tau \|\nabla u(t)\|_{\L^2_{x}}^2\, \d t. $
The treatment of the term $\beta_{0}$ is as before: we have 
$$
\int_{-\infty}^\tau | \angle {\beta_{0}u(t)}{u(t)}|\, \d t \le  C'\varepsilon(S^{\frac{1}{2}-\frac{1}{r_{1}}}  I^{\frac{1}{2}+\frac{1}{r_{1}}} + S^{1-\frac{2}{r_{1}}} I^{\frac{2}{r_{1}}})
$$
and Young's inequality allows us to hide the contribution coming from  $S$ on the left-hand side up to loosing a little on $-2\lambda I$ by choosing $\varepsilon$ small enough. Next the contribution of  $\beta_{\infty}$ is 
$$
\int_{-\infty}^\tau |  \angle {\beta_{\infty}u(t)}{u(t)}|\, \d t \le   \int_{-\infty}^\tau \big(P_{\infty} \|u(t)\|_{\L^2_{x}}\|\nabla u(t)\|_{\L^2_{x}} +P_{\infty}^2\|u(t)\|_{\L^2_{x}}^2\big)\, \d t
$$
and Young's inequality yields again a contribution of 
$\int_{-\infty}^\tau \|u(t)\|_{\L^2_{x}}^2\, \d t$ that is compensated if  $\kappa$ is larger than a constant times $P_{\infty}^2$, up to loosing again a little on $-2\lambda I$. We obtain $S\le0$, and thus $u(t)=0$ for $t\le 0$.
 
 In the second case, we know that $u\in \C_{0}(-\infty,s; \L^2_{x})$ with $\L^2_{x}$ limit when $t\to s^-$. In particular, we can argue with $S= \sup_{t \le s}\|u(t)\|_{\L^2_{x}}^2$, where $u(s)$ means $u(s^-)$, and the proof is the same. 
 \end{proof}
 
 Finally the result with lower bounds (Theorem~\ref{th:lb}) becomes the following and we skip the easy adaptation of the proof. Again, a lower bound on $A$ is implicit in order  to check the assumptions. 
 
\begin{thm}[Invertibility through lower bounds, inhomogeneous case]
 \label{thm:lbi} \
\begin{enumerate}
\item Assume that there exists $c,c'>0$ such that
$$
\Re \Angle {\Lt  u}{u}  \ge c \|\nabla u\|_{\LL}^2- c'\|u\|_{\LL}^2
$$
for all $u\in \cV$.
Then $\cH+\kappa$ is invertible for all $\kappa \ge \kappa_{1}$ with $\kappa_{1}(\Lambda, n, q_{1}, r_{1}, c, c')>0$ large enough.
\item Assume that $$\Re \angle {\Lt  w}{w}  \ge - c'\|w\|_{\L^2_{x}}^2$$ almost everywhere for some $c'>0$ and  for all $w\in \H^1(\R^n)$.
Then $\cH+ \kappa $ is causal in the sense of   Theorem~\ref{thm:causalityalpha} for any $\kappa\ge c'$.
\end{enumerate}
\end{thm}

\begin{cor} Under either smallness assumption { (Theorem~\ref{thm:causalityalpha}) or lower bounds (Theorem~\ref{thm:lbi} (ii))} it follows for all $\kappa$ large enough that $\Ga_{\kappa}(t,s)=0$ if $t<s$ and $\Ga_{\kappa}(t,s)\to I$ strongly if $t\to s^+$. 
 \end{cor}

\begin{proof} The {$\Del^{r_1,q_1}$-solution} $u$ of $\partial_{t}u+\Lt u+\kappa u=\delta_{s}\otimes \psi$ is given by $\Ga_{\kappa}(t,s)\psi$ for $t\ne s$. Thus, $\Ga_{\kappa}(t,s)=0$ if $t<s$ and $\lim_{t\to s^-}\Ga_{\kappa}(t,s)\psi=0$. Hence, $\lim_{t\to s^+}\Ga_{\kappa}(t,s)\psi=\psi$ follows from (i) in {(the modification of)} Theorem~\ref{thm:Gammats}.
 \end{proof}

\subsection{The Cauchy problem and the fundamental solution operator}
\label{sec:CP}
We consider \textit{in fine} the  Cauchy problem on the strip $[0,T]\times \R^n$ with $0<T<\infty$. (We shall say a word concerning $T=\infty$ in Remark \ref{rem:T=infty}) . It is of course sufficient to consider coefficients only on this strip and the foregoing results will allow us to work under the following set of assumptions.
\begin{itemize}
\item[(A1)]$\Lt $ is given as in \eqref{eq:Lt} with coefficients  $A,\oa  , \ob  , a$  defined almost everywhere in $(0,T)\times \R^n$.
\item[(A2)] $A$ has bounds   $\lambda,\Lambda$ as in \eqref{eq:Abdd}, \eqref{eq:Aelliptic} for almost every $t\in (0,T)$. 
\item[(A3)]  The lower order coefficients satisfy  $\mathbf{(D_{\varepsilon})}$ on $(0,T)\times \R^n$   for all $\varepsilon$ small enough with compatible pair $(\tilde r_{1},\tilde q_{1})$ as in Definition \ref{def:compatiblepair}.
\end{itemize}
Recall that we take coefficients with $|\oa  |^2+|\ob  |^2+|a|\in \L^{\tilde r_{1}}_{t}\L^{\tilde q_{1}}_{\vphantom{t} x}+ \L^\infty_{t}\L^\infty_{x}$ in $(0,T)\times \R^n$. As in the case of $\R^{n+1}$, (A3) automatically holds except if $\tilde r_{1}=\infty$ (hence $n\ge 3$ and $\tilde q_{1}=n/2$), in which case we can proceed by imposing it or by assuming (uniform) $t$-continuity on $[0,T]$ instead of mere boundedness, valued in $\L^{n/2}_{x}$, compare with Remark~\ref{rem:Deps}. An alternative in the case $(\tilde r_1, \tilde q_1) = (\infty, n/2)$ is to replace (A3) by the following condition.
\begin{itemize}
 \item[(A3)']  $\tilde r_{1}=\infty$ and  there exist $c>0$ and $c'\ge 0$ such that  $$
\Re \angle {\Lt  v}{v}  \ge c \|\nabla v\|_{\L^2_{x}}^2- c'\|v\|_{\L^2_{x}}^2$$ for almost every $t\in (0,T)$ and all $v \in \cV$.
\end{itemize}

Let $(r_{1},q_{1})$ be the admissible conjugate pair to $(\tilde r_{1},\tilde q_{1})$ defined in \eqref{eq:associatedpair}.

For  $\psi\in \L^2_{x}$, $F\in \L^2(0,T; \L^2_{x})$, $g\in \L^{r'}(0,T; \L^{q'}_{x})$, where $(r,q)$ is an arbitrary admissible pair as in Definition \ref{def:admissible} or $(r,q)=(\infty,2)$, and $h\in \L^2(0,T; \L^2_{x})$,  the Cauchy problem with initial condition $\psi$ and forcing terms $-\div F+g+h$ consists  of finding $$u\in  \Del^{r_{1},q_{1}}_{0,T}\coloneqq\L^2(0,T; \H^1_{x}) \cap \L^{r_{1}}(0,T; \L^{q_{1}}_{x})$$  solving
 \begin{equation}
 \label{eq:Cauchypb}
 \begin{cases}
\partial_{t}u+\Lt u= -\div F+g+h & \textrm{in}\  (0,T)\times \R^n,
\\
\quad \; u(0,\cdot)=\psi & \textrm{on}\  \R^n,
\end{cases}
\end{equation}
in the sense that the first equation is satisfied weakly against test functions $\tphi \in \cD((0,T)\times \R^n)$ as in \eqref{eq:Deltasol} and that the second equation means $u(t,\cdot) \to \psi$ in $\cD'(\R^n)$ as $t\to 0$. 
%If $T=\infty$, we may replace the non-homogeneous Sobolev space $\H^1_{x}$ by its homogeneous version when $h=0$. 
Weak solutions for the Cauchy problem \eqref{eq:Cauchypb} on $[0,T]\times \R^n$ are those solutions in the class $\L^2(0,T; \H^1(\R^n)) \cap \L^{\infty}(0,T; \L^{2}(\R^n))$. This space 
 embeds  into $ \L^{r_{1}}(0,T; \L^{\vphantom{r_{1}} q_{1}}(\R^n))$, 
see Proposition~\ref{prop:GN} for a quick proof. So we have defined an {\it a priori} larger class of solutions. We next show that in fact a $\Deldot^{r_{1},q_{1}}$-solution is continuous in time valued in $\L^2_{x}$, 
so that in the end it is a weak solution. 
This continuity can be obtained as a regularity result in the inhomogeneous variational setting.

\begin{lem}
\label{lem:energyfiniteinterval} 
Let $\I=(0,T)$ and  $ u\in \L^2(\I;\H^1_{x})$ with $\partial_{t}u= -\div F+ g+h$,  where $F\in \L^2(\I;\L^2_{x})$, $g\in \L^{r'}(\I;\L^{q'}_{x} )$ for $(r,q)$ an admissible pair  or $(r,q)=(\infty,2)$, and $h\in \L^2(\I;\L^2_{x})$. 
Then   $u\in \C([0,T]; \L^2_{x})$ and we have the integral equalities \eqref{eq:energy} for $0 \leq \sigma < \tau \leq T$.
\end{lem}

\begin{proof}  
According to the discussion (ii) in Section  \ref{sec:inhomogneous}, the result holds when $\I=(0,\infty)$.
We proceed as in Corollary~\ref{cor:energy} by constructing an extension $v$ of $u$ to which this applies. The situation here is easier since $u$ is locally integrable; still the language of distributions is convenient.  

More precisely, for $f\in \cD(0,\infty)$ and $\psi\in \cD(\R^n)$, we first note that $t\mapsto \angle{u(t)}{\psi}$ is absolutely continuous on $[0,T]$ with derivative equal to $\angle{F(t)}{\nabla \psi}+ \angle{g(t)}{\psi}+ \angle{h(t)}{\psi}$ almost everywhere on $(0,T)$. Hence,  if $\chi$ denotes a smooth function  that is 1 for $t\le T$ and 0 for $t\ge 2T$, the expression
$$
\Angle{v}{f\otimes \psi}=\int_{0}^\infty (\angle{u(t)}{\psi}1_{(0,T]}(t)+ \angle{u(2T-t)}{\psi}1_{(T,2T)]}(t))\chi(t)\overline f(t)\, \d t
$$
defines a distribution in $(0,\infty)\times \R^n$, which  equals $u$ when restricted to $(0,T)\times \R^n$. Calculations show that 
$\partial_{t}v= (-\div F_{o}+ g_{o}+h_{o})\chi+ u(2T-\cdot)\chi'$
and $\nabla v= (\nabla u)_{e}\chi$, where the subscripts ${\tiny o,e}$ denote odd and even extensions at $t=T$, respectively. Hence, all the required assumptions can be verified to conclude that $v\in \C_{0}([0,\infty), \L^2_{x})$.
\end{proof}

As usual one can add to $g$ several other terms of the same type with different admissible pairs. 

\begin{cor}\label{cor:continuous}  Any $\Del^{r_{1},q_{1}}_{0,T}$-solution  of the Cauchy problem \eqref{eq:Cauchypb}  
%belongs to $\C([0,T]; \L^2_{x})$.  In particular it 
is  a weak solution.
 \end{cor}

{Our main result on the Cauchy problem is a synthesis of most of the theory that we developed so far.}

\begin{thm}
\label{thm:Cauchy} Assume  (A1), (A2) and one of (A3) or (A3)'. 
There is a   solution $u$ of \eqref{eq:Cauchypb}, unique in the class $  \L^2(0,T; \H^1_{x}) \cap \L^{r_{1}}(0,T; \L^{q_{1}}_{x})$, {with the following properties.}
\begin{enumerate}
	\item $u$ belongs to $\C([0,T]; \L^2_{x})$ and is the restriction to  ${(0,T)}$ of a function in $\Hdot^{1/2}_{t}\L^2_{x}$.  
	
	\item  Changing the origin of time,  for $0\le s\le t\le T$ there exists a  unique $\L^2_{x}$-bounded operator $\Gamma(t,s)$,  called \emph{fundamental solution operator} for the Cauchy problem on $(s,T)\times \R^n$ with no forcing terms and initial condition in $\L^2_{x}$, that sends the initial condition at time $s$ to the value of the unique solution at time $t$. 
	\item For $t\in [0,T]$ the solution $u$ above is then given by
	\begin{equation}
		\label{eq:representationCauchy}
		u(t)= \Gamma(t,0)\psi- \int_{0}^t\Gamma(t,s)\div F(s)\, \d s + \int_{0}^t\Gamma(t,s)g(s)\, \d s+ \int_{0}^t\Gamma(t,s)h(s)\, \d s,
	\end{equation}
	where the first two integrals are weakly defined in $\L^2_{x}$, while the last one converges strongly (i.e. in the Bochner sense).
	\item {Define $\cH$ on extending $A$ by the identity and the lower order coefficients by $0$ on $(\R\setminus (0,T))\times \R^n$.}\footnote{There are many ways to do this extension, but this one does not increase the various constants on the coefficients and {could be called canonical.} Alternately, if the coefficients are already defined on full space-time, then one only needs their restrictions anyway.} The fundamental solution operators themselves are given by 
	\begin{equation}
		\label{eq:GammaGreen}
		\Gamma(t,s)= \e^{\kappa(t-s)}\Ga_{\kappa}(t,s), \quad { 0\le s \le t\le T,}
	\end{equation}
	for all $\kappa\ge 0$ for which $\cH+\kappa$  is invertible and causal and   $\Ga_{\kappa}(t,s)$ is  the Green operator obtained  under this assumption.
\end{enumerate}
 
\end{thm}          

\begin{proof} 
  We extend the forcing terms   $F, g, h$ by 0,  keeping the same notation for the extensions and $\Lt$: they satisfy the same conditions on full space-time. We fix $\kappa> 0$ for which $\cH+\kappa$ is invertible and causal (Theorems \ref{thm:invertiblealpha} and  \ref{thm:causalityalpha} or \ref{thm:lbi}). 

First, we can {use the inhomogeneous version of Theorems~\ref{thm:Fg}}, \ref{thm:L1L2} and \ref{thm:deltasL2} to build a (unique)  $\Del^{r_{1},q_{1}}$-solution $v $ to
$$
\partial_{t}v+\Lt v+\kappa v= \delta_{0}\otimes \psi -\div (F \e^{-\kappa t})+g \e^{-\kappa t}+h \e^{-\kappa t}
$$
and take $u \coloneqq v\e^{\kappa t}$ restricted to $[0,T]\times \R^n$. {The assumption of causality implies indeed that $u(t) \to \Gamma(t,0)\psi = \psi$ in $\L^2_{x}$ as $t \to 0$.} Applying the inhomogeneous version of Theorem~\ref{thm:representation} to the Green operators $\Ga_{\kappa}(t,s)$ of $\cH+\kappa$ gives us \eqref{eq:representationCauchy} with $\Gamma(t,s)= \e^{\kappa(t-s)}\Ga_{\kappa}(t,s)$. We refer in particular to (v) in Section~\ref{sec:inhomogneous}. 

Next, we check uniqueness in the class $  \L^2(0,T; \H^1_{x}) \cap \L^{r_{1}}(0,T; \L^{q_{1}}_{x})$. Assume that $u$ solves the Cauchy problem with $\psi=0$,   $F=0$, $g=0$, $h=0$. By Corollary~\ref{cor:continuous} we know that $u\in \C([0,T]; \L^2_{x})$. With  $\kappa$ as above, $v \coloneqq u\e^{-\kappa t}$ solves the Cauchy problem with $0$-data for $\partial_{t}v+\Lt v+\kappa v=0$ in $(0,T)\times \R^n$. By restriction, as before, we can use the global parabolic operator also to build a solution $u^T \in \L^2(T,\infty; \H^1_{x}) \cap \L^{ r_{1}}(T,\infty; \L^{q_{1}}_{x})\cap \C_0([T,\infty); \L^2_{x})$ to the same equation with initial data $u^T(T) = u(T)$ and in $(-\infty,0]\times \R^n$ we set $u^0(t) \coloneqq 0$. By continuity valued in $\L^{2}_{x}$, we can glue $u^0, u, u^T$ together to a $\Del^{r_{1},q_{1}}$-solution $w$ of $\partial_{t}w+\Lt w+\kappa w=0 $ in $\R^{n+1}$, which vanishes identically by the inhomogeneous version of Theorem~\ref{thm:uniqueness}. 
Hence, we have $u=0$. 

The rest of the statement follows easily. By construction, $u$ is the restriction of a function which belongs to $\cVdot$. Next,  uniqueness implies that the formula \eqref{eq:GammaGreen} is valid for all $\kappa> 0$ for which $\cH+\kappa$ is invertible and causal. 

It remains to include the case $\kappa=0$ for \eqref{eq:GammaGreen}   when $\cH$ is invertible and causal. To this end, we can apply the homogeneous versions in Sections \ref{sec:existenceanduniqueness}, \ref{sec:schwartz} and \ref{sec:invertibility} {with $F,g,h$ being all zero}.  In that case, we consider a  $\Deldot^{r_{1},q_{1}}$-solution on $\R^{n+1}$, which produces  a solution by restriction.  This solution has a representation using the Green operators $G_{0}(t,s)$. Already established uniqueness shows that   $\Gamma(t,s)=G_{0}(t,s)$.    
\end{proof}

\begin{rem} 
\label{rem:T=infty}
In  case that $T=\infty$ and $h=0$, inspection of the {proof above with $\kappa=0$  shows that if}  $\cH$ is invertible and causal,  then existence and uniqueness hold  in the class $\L^2(0,\infty; \Hdot^1(\R^n)) \cap \L^{ r_{1}}(0,\infty; \L^{ q_{1}}(\R^n))$, with  regularity in $\C_{0}([0,\infty);\L^2_{x})$.  
\end{rem}

\begin{rem}
Having forcing terms in $\L^{ r'}(0,T; \L^{ q'}(\R^n))+\L^{2}(0,T; \L^{2}(\R^n))$ allows us to cover for example supercritical forcing terms $g\in \L^{\rho'}(0,T; \L^{\eta'}_{x})$, {by which we mean $2< \rho,\eta< \infty$ with $\frac 1 \rho +\frac n {2\eta} > \frac n 4$ and, when $n=1$, additionally  $\frac 1 \rho -\frac 1 {2\eta}{<} \frac 1 4$.}
{Indeed, visualizing the exponents in a $(\frac{1}{\rho},\frac{1}{\eta})$-plane immediately reveals that they can be decomposed as required.}
\end{rem}

\begin{rem}
\label{rem:independence} 
{If the coefficients for $\cH$ are defined on $\R^{n+1}$, then by \eqref{eq:GammaGreen}} we have $$\e^{\kappa (t-s)}\Ga_{\kappa}(t,s)= \e^{\kappa_{0} (t-s)}\Ga_{\kappa_{0}}(t,s)$$ for all $\kappa\ge \kappa_{0}\ge 0$  and $t,s\in \R$, where $\kappa_{0}$ is such that $\cH+\kappa_{0}$ is invertible and causal (setting also $\Ga_{0}\coloneqq\Ga$). It is interesting to note that this relation between Green operators cannot be seen directly in $\R^{n+1}$ because the conjugation by the exponentials does not preserve the spaces of solutions. Another interesting consequence is that it implies exponential decay estimates for the operator norm: 
$$
\| \Ga_{\kappa}(t,s)\|_{\L^2_{x}\to \L^2_{x}} \le  \e^{(\kappa_{0}-\kappa) (t-s)}
\sup_{\sigma\le \tau} \| \Ga_{\kappa_{0}}(\tau,\sigma)\|_{\L^2_{x}\to \L^2_{x}},
 $$
recalling that $t-s\ge 0$ for the Green operators to be non-zero.
\end{rem}
  
\begin{rem}
\label{rem:dependanceonT} 
If (A3) is used, then constants in the implicit estimates for $u$ depend on the choice of $\varepsilon_{0}$ for invertibility and causality, 
which was seen to depend on $\lambda,\Lambda, n,q_{1},r_{1}$, and $P_{\infty}$ in the decomposition $\mathbf{(D_{\varepsilon_{0}})}$. They do not depend  on $T$ unless $P_{\infty}$ does. 
\end{rem}

\begin{cor}
\label{cor:adjoint}  
Under the same assumptions as in Theorem~\ref{thm:Cauchy} we have for all $t>s$ the equality 
\begin{equation}
\Gamma(t,s)=\widetilde\Gamma(s,t)^*,
\end{equation}
where $^*$ is the complex adjoint  and $\widetilde\Gamma(s,t)$ is the generalized  fundamental solution of the adjoint problem.
\end{cor}

\begin{proof} We know that the Green operators $\Ga_{\kappa}(t,s)$ and $\tGa_{\kappa}(s,t)$ are adjoint operators. If we adapt the proof above to the adjoint backward operator $-\partial_{t}+\Ltstar   $, we produce solutions in $(0,T)\times \R^n$  on restricting the ones in $\R^{n+1}$ for $-\partial_{t}+\Ltstar   +\kappa$ multiplied by  $\e^{\kappa(T-s)}$ (in the variable $s$). Changing the initial time $T$ to $t$, this yields that the fundamental solution operator $\widetilde \Gamma(s,t)$ for the adjoint problem agrees with $ \e^{\kappa(t-s)}\tGa_{\kappa}(s,t)= \Gamma(t,s)^*$, where the last equality follows from \eqref{eq:GammaGreen}.
 \end{proof}

\begin{cor}
\label{cor:agree} The  solution $u$ of Theorem~\ref{thm:Cauchy} agrees with the ones build in \cite{LSU} and \cite{Ar68} under assumptions in these references.  
 \end{cor}

\begin{proof} 
{By mixed embeddings (Proposition~\ref{prop:GN}), the space for uniqueness in Theorem~\ref{thm:Cauchy} contains the standard} energy space  $\L^2(0,T; \H^1_{x}) \cap \L^{\infty}(0,T; \L^{2}_{x})$. In Chapter~3 of  \cite{LSU}, weak solutions in the latter class are constructed exactly under the same assumptions (A1), (A2) and subcritical (Remark \ref{rem:subcritical}) or critical (a.k.a compatible) conditions on the coefficients (which are {even} assumed real there). In \cite{Ar68} this is being done under the more restrictive conditions of subcritical and real coefficients with the structural conditions (A1) and (A2).
\end{proof}

\subsection{$\mathbf{\L^2}$ off-diagonal estimates}
\label{sec:ode}

Aronson further proved pointwise Gaussian estimates of the generalized fundamental solution when the coefficients are real-valued \cite{Ar67,Ar68}. As already mentioned,  assumptions on lower order coefficients in \cite{Ar68} amount to what we called subcritical compatibility (Remark~\ref{rem:subcritical}), used in an essential way together with the fact that the coefficients are real, to obtain local boundedness of weak solutions. Already in the elliptic case with leading term the Laplacian on the unit ball, explicit examples show existence of unbounded weak solutions for some first order coefficients in $\L^n$ or some zero order coefficients  in $\L^{n/2}$, see \cite{KS}. 

We know from Corollary~\ref{cor:agree} that our solutions agree with the ones of Aronson under his assumptions; in particular his generalized fundamental solution operator and ours are identical.
Hence, pointwise bounds under (critical) compatibility assumptions are not to be expected. Still, under this assumption, we will be able to show $\L^2$ off-diagonal estimates (or Gaffney estimates) for the fundamental solution operator, that is, decay of localized $\L^2$ norms. 

When there are no lower order terms, the method of Aronson has been streamlined with the exponential trick of   Davies \cite{Da}  for time independent $A$ and this has been adapted by Fabes--Stroock \cite{FaSt}   when $A$ is time-dependent, see also Hofmann--Kim \cite{HK04} for a nice presentation, using the Gronwall lemma as a starting point. The same ideas go through with bounded lower order coefficients but when they are allowed to be unbounded, it is not clear how to set up the arguments properly. In \cite{AMP}, a construction is proposed in absence of lower order terms, starting from the semigroup case. This approach is not possible when using mixed norms on  lower order coefficients because there is no  semigroup to begin with.  Our  approach allows us to overcome these difficulties by extending Davies' ideas  to the context of variational parabolic forms.

\begin{thm}
\label{thm:ODE} Assume the conditions  of Theorem~\ref{thm:Cauchy}. 
Then there are constants $0<C,c_{0},\omega <\infty$ such that for all $0\le s<t\le T$, all closed sets $E,F \subset \R^n$ and  all $\psi\in \L^2_{x}$ with support in $F$, we have 
\begin{equation}
\label{eq:ODE}
\|\Gamma(t,s)\psi\|_{\L^2(E)}\le C\e^{- \tfrac{d(E,F)^2}{4c_{0}(t-s)}+\omega (t-s)} \|\psi\|_{\L^2(F)}.
\end{equation} 
\end{thm}

Let us comment  on the three constants. If (A3) is used, then $\omega=c_{0}P_{\infty}^2$ with $P_{\infty}$ from the decomposition  given by $\mathbf{(D_{\varepsilon_{0}})}$ and $C, c_{0}$ depend only on  $\lambda,\Lambda, n, q_{1}, r_{1}$, where $\varepsilon_{0}$ is such that the arguments for invertibility and causality apply. As in Remark~\ref{rem:dependanceonT}, they may depend on $T$ but only through $P_{\infty}$. If (A3)' is used, then $C, \omega= c_{0}$ depend  on  $\Lambda, n$ and $c, c'$ in (A3)'.

\begin{proof} 
We extend the coefficients to full space-time as in the proof of Theorem~\ref{thm:Cauchy} and use the same notation. Henceforth, we work in $\R^{n+1}$ and prove \eqref{eq:ODE} for all $ s<t$.

For a function $h:\R^n\to [0,\infty[$ bounded and Lipschitz, consider the operator obtained in $\R^{n+1}$ on conjugating $\cH\, (=\partial_{t}+\Lt )$ with the multiplication by $\e^h$.
A calculation (in the weak sense) shows that 
\begin{equation}
\label{eq:conjugation} 
\e^{h}
\cH 
%(\partial_{t}+\Lt )
\e^{-h}= 
\cH
%\partial_{t}+\Lt 
+\beta_{h},
\end{equation}where 
\begin{align}
\label{eq:beta}
	\Angle{\beta_{h} u}{v}= \Angle{\oa_{h} u}{  \nabla v} + \Angle{\ob_{h}\cdot \nabla u}{  v} + \Angle{a_{h} u}{ v}
\end{align}
and with  $A^t$ being the real transpose of $A$,
\begin{equation}
\label{eq:coefh}
\begin{split}
\oa_{h}&= -A \nabla h,	 
 \\
 \ob_{h}&= \ \ A^t \nabla h,
 \\
{a_{h}}&= -A \nabla h	\cdot \nabla h + (\oa  -\ob  )\cdot 	\nabla h.
\end{split}
 \end{equation}
The coefficients $\oa_{h}$ and $ \ob_{h}$ are bounded by $\|A\|_{\infty}\|\nabla h\|_{\infty}$. In ${a_{h}}$, the first term is bounded by $\|A\|_{\infty}\|\nabla h\|_{\infty}^2$. To handle the second term, we distinguish the two assumptions. 

\medskip

\noindent \emph{Proof under (A3)}.  The number $\varepsilon_{0}$ is chosen in particular such that \eqref{eq:lowerboundinh} for $\cH+\kappa$ holds with $\kappa\ge \kappa_{0}$ where $\kappa_{0} =\delta/4+ c_{\delta}P_{\infty}^2$. Our first goal is to check that  \eqref{eq:lowerboundinh}  for $\cH+\kappa+\beta_{h}$ holds with $\delta/4$ replaced by, say, $\delta/8$ for large enough $\kappa$ that will also depend on $\|\nabla h\|^2_{\infty}$. To this end, it will suffice to revisit the proof of that inequality after adding the contribution of the coefficients in  \eqref{eq:coefh}. 

\

\noindent \emph{Step 1: Proof of the lower bound for the perturbed $\cH+\kappa+\beta_{h}$.} We decompose $\oa  -\ob=(\oa_{0}  -\ob_{0})+(\oa_{\infty}  -\ob_{\infty}) $ as in the assumption $\mathbf{(D_{\varepsilon})}$ with $\varepsilon=\varepsilon_{0}$. The term coming from $(\oa_{\infty}-\ob_{\infty})\cdot 	\nabla h$ brings a bounded contribution of size $P_{\infty}\|\nabla h\|_{\infty}$. 
For the other term, we observe that $\oa_{0}-\ob_{0}$ belongs to $\L^{2\tilde r_{1}}_{t}\L^{2\tilde q_{1}}_{\vphantom{t} x}$ with norm not exceeding $2\varepsilon_{0}$ and $(2\tilde r_{1}, 2\tilde q_{1})$ is a subcritically compatible pair for  coefficients of order $0$. 
We decompose further this term as suggested in Remark~\ref{rem:subcritical}. To this end, call $\Lt_{0}$ the elliptic operator with coefficients $A, \oa_{0}, \ob_{0}, a_{0}$ and $\cH_{0}$ the corresponding parabolic operator. Through the choice of $\eps_0$, we can make sure that \eqref{eq:lowerbound} holds for $\cH_{0}$.  We also know that the multiplication by $V\in \L^{\tilde r_{1}}_{t}\L^{\tilde q_{1}}_{\vphantom{t} x}$ is a bounded operator $\cVdot \to \cVdot'$. 
Thus, we can choose $\eta>0$  (depending on $n, q_{1},r_{1}, \delta$) so small that  $\|V\|_{\L^{\tilde r_{1}}_{t}\L^{\tilde q_{1}}_{\vphantom{t} x}}\le \eta$ implies 
$$
\Re \Angle {(\cH_{0}+V) u}{(1+\delta \H_{t})u}  \ge \frac{\delta}{4}  \|u\|_{\cVdot}^2.
$$
For $m>0$, the truncation  $V_{0} \coloneqq 1_{|\oa_{0}-\ob_{0}|>m}(\oa_{0}-\ob_{0})\cdot \nabla h$ satisfies $$\|V_{0}\|_{\L^{\tilde r_{1}}_{t}\L^{\tilde q_{1}}_{\vphantom{t} x}} \le 4\varepsilon_{0}^2 m^{-1}\|\nabla h\|_{\infty}.$$ We choose $m$  so that $4\varepsilon_{0}^2m^{-1}\|\nabla h\|_{\infty}=\eta$. On the other hand, $V_{\infty} \coloneqq 1_{|\oa_{0}-\ob_{0}|\le m}(\oa_{0}-\ob_{0})\cdot \nabla h$ satisfies $$\|V_{\infty}\|_{\infty} \le m\|\nabla h\|_{\infty}=4\varepsilon_{0}^2\eta^{-1}\|\nabla h\|_{\infty}^2.$$  Setting $\beta_{\infty}= \cH -\cH_{0}$ and $\tilde \beta_{h}=\beta_{h}-V_{0}$, we have established the decomposition
\begin{align}
\label{eq:ODE-decomp}
\cH
%\partial_{t}+\Lt 
+\beta_{h}+\kappa= (\cH_{0}+V_{0})+(\tilde  \beta_{h}+\beta_{\infty})+\kappa
\end{align}
with $\tilde \beta_{h}+\beta_{\infty}$ having first order coefficients bounded by  $\|\nabla h\|_{\infty}+P_{\infty}$ and zero order coefficients bounded by  $\|\nabla h\|_{\infty}^2+P_{\infty}^2$ up to multiplicative constants that depend only on $\lambda, \Lambda, n, q_{1},r_{1}$. {This was the key point.}  

Applying the same {simple absorption argument} as in Theorem~\ref{thm:invertiblealpha} to this decomposition reveals that for some constant $c_{0}$ with the same dependency and  $\kappa=1+ c_{0}(\|\nabla h\|_{\infty}^2+P_{\infty}^2)$, the operator in \eqref{eq:ODE-decomp} is invertible from $\cV$ onto $\cV'$ with a lower bound $\delta/8$ in \eqref{eq:lowerboundinh}. 

Our next goal is to transfer such lower lower bounds to operator norms for the perturbed fundamental solution operator, following the dependency in $h$. 

\

\noindent \emph{Step 2: Norm bounds for the perturbed fundamental solution operator.}
With the constraints on $\kappa$ and $h$ above, the norm of the  inverse of $\cH+\kappa+\beta_{h}$ depends on $\lambda, \Lambda, n, q_{1},r_{1}$ but not on $h$. Altogether, it follows that the Green operators $G_{h,\kappa}(t,s)$ associated to $ \e^{h}\cH \e^{-h} + \kappa$ are uniformly bounded on $\L^2_{x}$ with respect to $(t,s)$ with a bound $C_{0}$ depending only on $\lambda, \Lambda, n, q_{1},r_{1}$.

Now, by construction we have $G_{h,\kappa}(t,s)=\e^h G_{\kappa}(t,s)\e^{-h}$ and by Theorem~\ref{thm:Cauchy} we have $\Gamma(t,s)=\e^{\kappa(t-s)}G_{\kappa}(t,s)$. Hence, $\e^h\Gamma(t,s)\e^{-h}= \e^{\kappa(t-s)} G_{h,\kappa}(t,s)$. This infers that  for all $t-s=1$ and $\psi\in \L^2_{x}$,
$$
\|\e^h\Gamma(t,s)\psi\|_{\L^2_{x}}\le (C_{0}\e) \e^{c_{0}(\|\nabla h\|_{\infty}^2+P_{\infty}^2)} \|\e^h\psi\|_{\L^2_{x}}.
$$ 
A scaling argument will now provide us with the right dependence of $\omega$. 

Fix $s=0$ to simplify matters by time translation invariance of the assumptions. Set  $u(t,\cdot)\coloneqq \Gamma(t,0)\psi$. Recall that $u$ solves $\partial_{t}u+\Lt u= \delta_{0}\otimes \psi$, so that if $R>0$,  then $u^R(t,x) \coloneqq  u(R^2t,Rx)$ solves $\partial_{t}u^R+\Lt ^Ru^R= \delta_{0}\otimes \psi^R$, with $\psi^R(x)=\psi(Rx)$ and $\Lt ^R$ has coefficients $A(R^2t,Rx)$, $R\, \oa  (R^2t,Rx)$, $R\, \ob  (R^2t,Rx)$, $R^2a(R^2t,Rx)$. The quantity $P_{\tilde r_{1},\tilde q_{1}}$ is scale invariant and therefore does not depend on $R$. The same applies to the ellipticity constants $\lambda, \Lambda$, while $P_{\infty}$ becomes $P_{\infty}R$. Applying  the above conclusion to the Green operator  of $\partial_{t}+\Lt ^R$ at $t=1$ with $h^R(x)=h(Rx)$, and changing variables in space, yields 
$$
\|\e^h\Gamma(R^2,0)\psi\|_{\L^2_{x}}\le (C_{0}\e) \e^{c_{0}(\|\nabla h\|_{\infty}^2+P_{\infty}^2) R^2} \|\e^h\psi\|_{\L^2_{x}}.
$$
Altogether, this shows for all $t>s$ and $\psi\in \L^2_{x}$,
\begin{equation}
\label{eq:Daviestrick}
\|\e^h\Gamma(t,s)\psi\|_{\L^2_{x}}\le (C_{0}\e) \e^{c_{0}(\|\nabla h\|_{\infty}^2+P_{\infty}^2)(t-s)} \|\e^h\psi\|_{\L^2_{x}}.
\end{equation}
It remains to optimize $h$ appropriately.

\medskip

\noindent \emph{Step 3: Choice of $h$.}
 Fix $E,F$ closed sets, let $t>s$  and assume $d(E,F)^2>t-s$, since otherwise there is nothing to prove. Let $h(x) \coloneqq \inf( \frac{d(E,F)d(x,F)}{2c_{0}(t-s)}, N)$ with $N>\frac{d(E,F)^2}{2c_{0}(t-s)}$. We see that $h\ge \frac{d(E,F)^2}{2c_{0}(t-s)}$ on $E$, $h=0$ on $F$, and $\|\nabla h\|_{\infty}=\frac{d(E,F)}{2c_{0}(t-s)}$. Thus, if $\psi $ has support in $F$, we obtain 
\eqref{eq:ODE} with $C=C_{0}\e$ and $\omega=c_{0}P_{\infty}^2$. 

\medskip

\noindent \emph{Proof under (A3)'}.
 We modify the argument, explaining how to adapt the proof of Theorem~\ref{thm:lbi} ({or Theorem~\ref{th:lb} in the inhomogeneous setting, to be precise}). As $\tilde r_{1}=\infty$, we have $n\ge 3$ and $\tilde q_{1}=n/2$.  

There are two key observations. First, if we add  lower order terms with bounded coefficients to $\Lt$, then we still have the lower bound in (A3)' up to taking $c$ smaller and $c'$ larger. Second, if $V\in \L^{\infty}_{t}\L^{n/2}_{x}$, then 
$$
|\Angle{Vu}{v}|\le  \|V\|_{\L^{\infty}_{t}\L^{n/2}_{\vphantom{t} x}}\|u\|_{\L^{2}_{t}\L^{2^*}_{\vphantom{t} x}}\|v\|_{\L^{2}_{t}\L^{2^*}_{\vphantom{t} x}} \le c(n) \|V\|_{\L^{\infty}_{t}\L^{n/2}_{\vphantom{t} x}}\|\nabla u\|_{\LL}\|\nabla v\|_{\LL},
$$
so that in particular if $\eta= c(n)^{-1}c/2$ and $ \|V\|_{\L^{\infty}_{t}\L^{n/2}_{x}}\le \eta$, then we  preserve the lower bound assumption of  Theorem~\ref{th:lb} on adding $V$. 

In order to make use of these two observations, we recall that $\oa  -\ob  \in \L^{\infty}_{t}\L^{n}_{x}$ and decompose $(\oa-\ob)\cdot \nabla h$ further in $\L^{\infty}_{t}\L^{n/2}_{x}+ \L^{\infty}_{t}\L^{\infty}_{x}$ 
as $V_{0}+V_{\infty}$, where as usual we take $V_{0} \coloneqq 1_{|\oa-\ob|>m}(\oa-\ob)\cdot \nabla h$ and $V_{\infty} \coloneqq 1_{|\oa-\ob|\le m}(\oa-\ob)\cdot \nabla h$. We have  $$\|V_{0}\|_{\L^{\infty}_{t}\L^{n/2}_{x}} \le \||\oa-\ob|\|_{\L^{\infty}_{t}\L^{n}_{x}}^2\  m^{-1}\|\nabla h\|_{\infty},$$ and we choose $m$  so that this bound equals $\eta$. Thus, $$\|V_{\infty}\|_{\infty}\le \eta^{-1} \||\oa-\ob|\|_{\L^{\infty}_{t}\L^{n}_{x}}^2\ \|\nabla h \|_{\infty}^2.$$  The decomposition replacing \eqref{eq:ODE-decomp} is 
$$
%\partial_{t}+\Lt 
\cH +\beta_{h}+\kappa= 
%\partial_{t}+\Lt
\cH+ \tilde   \beta_{h}+ V_{0}+\kappa,
$$
where $\tilde \beta_{h}$ has first order coefficients bounded by $C\|\nabla h\|_{\infty}$ and zeroth order coefficients bounded by $C(1+\|\nabla h\|_{\infty}^2)$. 

Applying the two introductory observations and choosing $\kappa=c_{0}(1+\|\nabla h\|_{\infty}^2)$ for an appropriate constant $c_{0}$, we see that the inverse of 
$\cH+\tilde  \beta_{h}+ V_{0}+\kappa$ has a norm that is bounded by a constant  independent of $h$. The rest of the proof is as in the first case but the scaling argument is not needed: we first obtain 
$$
\|\e^h\Gamma(t,s)\psi\|_{\L^2_{x}}\le C_{0} \e^{c_{0}(1+\|\nabla h\|_{\infty}^2)(t-s)} \|\e^h\psi\|_{\L^2_{x}}
$$ 
for all $t>s$ and then {then same choice of $h$ as before} leads to \eqref{eq:ODE} with $\omega=c_{0}$ and $C=C_{0}$. 
\end{proof}

\subsection{Pointwise Gaussian bounds}
\label{sec:Gaussian_pointwise}

We prove that pointwise Gaussian bounds for the fundamental solution operator follow from an assumption of local boundedness on weak solutions of both the parabolic equation and its adjoint. To this end, we extend the argument presented in \cite{HK04} without lower order coefficients.  This argument adapts once we have \eqref{eq:Daviestrick} at hand. As said before, we do not know how to modify the argument in \cite{HK04} for this inequality directly in the presence of lower order coefficients.

We recall that a weak solution of $\partial_{t}u+\Lt u=0$ in an open set $\I\times \Omega$ is a function $u$ that is  in the class $\L^\infty( \I; (\L^2(\Omega))$ with $\nabla u$  in $\L^2(\I; (\L^2(\Omega))$ which satisfies the  equation weakly against test functions $\tphi \in \cD(\I\times \Omega)$ as in \eqref{eq:Deltasol}. It is well-known that $u$ is continuous in time locally in $\L^2$, see also  Lemma~\ref{lem:energyfiniteinterval}.  {The following definition introduces quantitative boundedness in the two variables.}

For $(t,x)\in \R^{n+1}$ and $r>0$, we let $Q_{r}(t,x)=(t-r^2, t]\times B(x,r)$ and  $Q_{r}^*(t,x)=[t,t+r^2)\times B(x,r)$ {be the usual forward and backward in time parabolic cylinders}.

\begin{defn} 
 We say that $\partial_{t}+\Lt $ and $-\partial_{t}+\Ltstar   $ have the local boundedness property if there are  $\rho\in (0,\infty]$  and $0<B<\infty$ such that for all $(t,x)\in \R^{n+1}$ and $0<r<\rho$, any weak solution of 
$\partial_{t}u+\Lt u=0$ and $-\partial_{t}\tilde u+\Ltstar   \tilde u=0$
on neighborhoods of $Q_{2r}(t,x)$ and $Q_{2r}^*(t,x)$, respectively, has local bounds of the form 
\begin{align}\label{eq:localbound}
\esssup_{B(x,r)}|u(t,\cdot)|^2 &\le \frac {B^2} {r^{n+2}} \iint_{Q_{2r}(t,x)} |u|^2,
 \\
\label{eq:localbound*}  \esssup_{B(x,r)}|\tilde u(t,\cdot)|^2 &\le \frac {B^2} {r^{n+2}} \iint_{Q_{2r}^*(t,x)} |\tilde u|^2.
\end{align}
\end{defn}

\begin{rem}
If $\rho=\infty$, the condition is scale invariant; here we will also encounter non-scale invariant situations, in which we need to consider $\rho<\infty$.
\end{rem}

Note that these conditions are usually presented by taking suprema on $Q_{r}(t,x)$, $Q_{r}^*(t,x)$ respectively, which means that one needs to know that solutions have pointwise values. Our weaker formulation suffices.

\begin{thm}
\label{thm:GUB} Assume the conditions of Theorem~\ref{thm:Cauchy} and that $\partial_{t}+\Lt $ and $-\partial_{t}+\Ltstar $ have the local boundedness property for some $\rho\in (0,\infty]$.  
Then, for all $t>s$, the fundamental solution operator $\Gamma(t,s)$ has  a kernel  $\Gamma(t,x,s,y)$, called generalized fundamental solution, with almost everywhere pointwise Gaussian upper bound
\begin{equation}
\label{eq:GUB}
|\Gamma(t,x,s,y)|\le \frac {\mu^{k+1}} {(16\pi c_{0}(t-s))^{n/2}}\e^{-\tfrac{ |x-y|^2}{16c_{0}(t-s)}+\omega (t-s)},
\end{equation}
whenever $k\rho^2\le t-s<(k+1)\rho^2$ for some $k\in \IN$. (If $\rho=\infty$, the only non-void case is $k=0$.)  Here,
$$
\mu= (32\pi c_{0})^{n/2}2^{n/2} \e^{2/c_{0}}(2^{1+n/2}BC)^2,
$$
where the constants $0< C, \omega, c_{0}<\infty$ are the ones explicated in  Theorem~\ref{thm:ODE}.
 \end{thm}

\begin{proof} 
Under the hypotheses of Theorem~\ref{thm:ODE} we have proved \eqref{eq:Daviestrick}, which we rewrite for all $t>s$, $\psi\in \L^2_{x}$ and real, Lipschitz and bounded $h$ as
\begin{equation}
\label{eq:Daviestrick2}
\|\Gamma^h(t,s)\psi\|_{\L^2_{x}}\le C\e^{\omega (t-s)}\e^{c_{0} \gamma^2 (t-s)} \|\psi\|_{\L^2_{x}}
\end{equation}
with $\Gamma^h(t,s)\coloneqq\e^h\Gamma(t,s)e^{-h}$ and $\|\nabla h\|_{\infty}=\gamma$.  {By duality} this inequality holds also for $\widetilde \Gamma^{h}(s,t)=\Gamma^{-h} (t,s)^*$. Let $$u^h(t,\cdot) \coloneqq \e^{-h}\Gamma^h(t,s)\psi= \Gamma(t,s)(e^{-h}\psi).$$ We may apply \eqref{eq:localbound} to $u^h$ and obtain for $0<t-s<\rho^2/2$ and $x\in \R^n$ that for almost every $z\in B(x,\sqrt {t-s}/2)$,
\begin{align*}
 |u^h(t,z)|^2  &\le \frac {2^{2+n}B^2} {(t-s)^{1+n/2}} \int_{s}^t\int_{B(x,\sqrt{t-s})} |u^h(\tau,y)|^2 \, \d y\d \tau
 \\
 &\le \frac {2^{2+n} B^2} {(t-s)^{1+n/2}} \int_{s}^t\int_{B(x,\sqrt{t-s})} \e^{-2h(y)}|\Gamma^h(\tau,s)\psi(y)|^2 \, \d y\d \tau,
\end{align*}
hence 
\begin{align*}
 |\Gamma^h(t,s)\psi(z)|^2  &\le \frac {2^{2+n}B^2} {(t-s)^{1+n/2}} \int_{s}^t\int_{B(x,\sqrt{t-s})} \e^{2h(z)-2h(y)}|\Gamma^h(\tau,s)\psi(y)|^2 \, \d y\d \tau
 \\
 &\le \frac {2^{2+n}B^2} {(t-s)^{1+n/2}} \int_{s}^t\int_{B(x,\sqrt{t-s})} \e^{2\gamma|z-y|}|\Gamma^h(\tau,s)\psi(y)|^2 \, \d y\d \tau
 \\
 &\le \frac {2^{2+n}B^2 \e^{4\gamma\sqrt{t-s}}} {(t-s)^{1+n/2}} \int_{s}^t\int_{B(x,\sqrt{t-s})} |\Gamma^h(\tau,s)\psi(y)|^2 \, \d y\d \tau
\\
&\le \frac {2^{2+n}B^2 \e^{4\gamma\sqrt{t-s}}} {(t-s)^{1+n/2}} \int_{s}^t C^2\e^{2\omega (\tau-s)}\e^{2c_{0} \gamma^2 (\tau-s)}  \|\psi\|_{\L^2_{x}}^2 \, \d \tau.
\end{align*}
Note that the right-hand side does not depend on the space variable. As $\tau-s\le t-s$, this implies
\begin{equation}
\label{eq:Gammah2infty}
\|\Gamma^h(t,s)\psi\|_{\L^\infty_{x}}\le \frac {2^{1+n/2}BC \e^{2\gamma\sqrt{t-s}}\e^{\omega (t-s)}\e^{c_{0} \gamma^2 (t-s)}} {(t-s)^{n/4}} \ \|\psi\|_{\L^2_{x}}.\end{equation}
Using \eqref{eq:localbound*} and \eqref{eq:Daviestrick2} for the adjoint of $\Gamma^h(t,s)$ and duality, this yields
\begin{equation}
\label{eq:Gammah12}
\|\Gamma^h(t,s)\psi\|_{\L^2_{x}}\le \frac {2^{1+n/2}BC \e^{2\gamma\sqrt{t-s}}\e^{\omega (t-s)}\e^{c_{0} \gamma^2 (t-s)}} {(t-s)^{n/4}} \ \|\psi\|_{\L^1_{x}}.\end{equation}
Let us momentarily assume $0<t-s<\rho^2$, that is $k=0$. We shall remove this in the final step. By the Chapman-Kolmogorov identity of Theorem~\ref{thm:Gammats}, which implies  
$\Gamma(t,s)=\Gamma(t,r)\Gamma(r,s)$ with $r=\frac{t+s}2$,  we obtain
%if $0<t-s<\rho^2$, 
\begin{equation}
\label{eq:Gammah1infty}
\|\Gamma^h(t,s)\psi\|_{\L^\infty_{x}}\le \frac {2^{n/2}(2^{1+n/2}BC)^2 \e^{2\sqrt 2\gamma \sqrt{t-s}}\e^{\omega (t-s)}\e^{c_{0} \gamma^2 (t-s)}} {(t-s)^{n/2}} \ \|\psi\|_{\L^1_{x}}.\end{equation}
By the Dunford--Pettis theorem (Theorem 1.3 in \cite{Arendt-Bukhvalov}),  this amounts to the fact that for all $t>s$,  $\Gamma(t,s)= \e^{-h}\Gamma^h(t,s)\e^h $ is an integral operator with  measurable kernel that we denote by $\Gamma(t,x,s,y)$, having an almost everywhere bound
\begin{equation}
\label{eq:Gamma(t,x,s,y)h}
|\Gamma(t,x,s,y)|\le \e^{h(y)-h(x)} \frac {2^{n/2}(2^{1+n/2}BC)^2 \e^{2\sqrt 2 \gamma \sqrt{t-s}}\e^{\omega (t-s)}\e^{c_{0} \gamma^2 (t-s)}} {(t-s)^{n/2}} .\end{equation}
Taking $h=0$ already gives us a uniform almost everywhere bound
\begin{equation}
\label{eq:Gamma(t,x,s,y)uniform}
|\Gamma(t,x,s,y)|\le  \frac {2^{n/2}(2^{1+n/2}BC)^2 \e^{\omega (t-s)}} {(t-s)^{n/2}} .\end{equation} 
In order to prove \eqref{eq:GUB}, we fix $x,y,t,s$ and assume $\frac {|x-y|}{2\sqrt {2}\sqrt{t-s}}\ge 2$; otherwise we can simply use \eqref{eq:Gamma(t,x,s,y)uniform} since $1 \le \e^{\frac {2}{c_{0}}}\e^{-\frac{ |x-y|^2}{16 c_{0}(t-s)}}$. 
We pick $h(z)= \inf(\gamma|z-y|, N)$ with $\gamma= \tfrac {|x-y|}{{4c_{0}(t-s)}}$ and $N>\gamma|x-y|$. Thus, $h$ is bounded and Lipschitz with $\|\nabla h\|_{\infty}=\gamma$ and $h(x)=\gamma|x-y|$, $h(y)=0$. 
Observe that  $2\sqrt 2 \gamma \sqrt{t-s} \le \frac {h(x)}{2}$   and 
$-\frac {h(x)}{2}+ {c_{0}\gamma^2}(t-s)= - \frac {|x-y|^2}{16c_{0}(t-s)}$. Hence, 
$$
|\Gamma(t,x,s,y)|\le  \frac {2^{n/2}(2^{1+n/2}BC)^2 \e^{\omega (t-s)}} {(t-s)^{n/2}}\  \e^{ - \tfrac {|x-y|^2}{16c_{0}(t-s)}}.
$$
This concludes the argument when $0<t-s<\rho^2$.

We are of course done when $\rho=\infty$. 
To conclude the proof when $\rho<\infty$, we iteratively apply  the Chapman-Kolmogorov formula for $\Gamma(t,s)$ together with the upper bound just found and the convolution rule $g_{\alpha}\star g_{\beta}=g_{\alpha+\beta}$, where $g_{\alpha}(x)= (4\pi \alpha)^{-n/2} \e^{-|x|^2/4\alpha}$ for $\alpha,\beta>0$.
\end{proof}

\begin{cor} 
Under the same assumptions as in Theorem~\ref{thm:GUB} we have for all $t>s$ the equality 
\begin{equation}
\Gamma(t,x,s,y)=\widetilde\Gamma(s,y,t,x)^*
\end{equation}
for almost every $x,y\in \R^{n+1}$, where $^*$ is the complex adjoint (here the conjugation as the kernels are complex-valued) and $\widetilde\Gamma(s,y,t,x)$ is the generalized  fundamental solution of the adjoint problem.
\end{cor}

\begin{proof} We know that $\Gamma(t,s)=\widetilde \Gamma(s,t)^*$ and both have integral kernels.
\end{proof}

\begin{rem}
 Aronson's prerequisite to obtaining Gaussian upper bounds for their generalized fundamental solution (which we now know agree with ours) is a condition on coefficients that insures the local boundedness property with the supremum, see Theorem~B in \cite{Ar68}. Thus,  Theorem~\ref{thm:GUB} reproves Aronson's upper bound 
 in a constructive way through identification of the general fundamental solution operators with integral kernels.
\end{rem}

\begin{rem}
The stability result in Proposition~2.1 of \cite{HK04} for pure second-order $\Lt$ {could} be adapted but not with full lower order terms. Although formulated as a perturbation result for local bounds, it proves more, {namely}: if weak solutions of $\partial_{t}- \div A \nabla  + \ob  \cdot \nabla$ satisfy local H\"older bounds with proper scaling, one preserves this regularity up to changing the H\"older exponent, on 
 perturbing of $A$ in $\L^\infty$ and $\ob  $ in the compatible mixed Lebesgue space. It is not clear what happens when adding the other terms with $\oa  $ or $a$. 
\end{rem}

\subsection{Pure second order elliptic part}
\label{sec:puresecondorder}

When the lower order coefficients are zero, that is, the elliptic part is the pure second order operator $\Lt_{0}\coloneqq -\div A \nabla$, we see that there is no need to introduce the compatible pair $(\tilde r_{1},\tilde q_{1})$ to define $\cH_{0}=\partial_{t}+\Lt_{0}\colon \cVdot \to \cVdot'$ in Proposition~\ref{prop:Lt} and the information that $\nabla u \in \LL$ suffices. Thus, we can introduce the (larger) class of $\L^2_{t}\Hdot^{1}_{\vphantom{t} x}$-solutions of $\partial_{t}u+\Lt_{0}u=f$ in $\R^{n+1}$, which we define as the class of distributions $u$ with $\nabla u \in \LL$ such that $ \partial_{t}u+\Lt_{0} u= f$ in $\cD'(\R^{n+1})$.  

Inspection of the arguments in Section~\ref{sec:existenceanduniqueness} reveals that if $\cH_{0}$ is invertible, then the statements extend by replacing systematically $\cH$, $\Deldot^{r_{1},q_{1}}$  and $\|u\|_{\Deldot^{r_{1}, q_{1}}}$ by  $\cH_{0}$, $\L^2_{t}\Hdot^{1}_{\vphantom{t} x}$ and $\|\nabla u\|_{\LL}$, respectively. In particular, uniqueness up to a constant (assuming invertibility) is obtained  in the larger class $\L^2_{t}\Hdot^{1}_{\vphantom{t} x}$.

From there on, the theory develops analogously in this special case. The Cauchy problem for $\partial_{t}+\Lt_{0}$ can be posed and solved uniquely in $\L^2(0,T; \H^1(\R^n))$ when $T<\infty$ for arbitrary data $\psi, F,g,h$ in appropriate spaces, or in $\L^2(0,\infty; \Hdot^1(\R^n))$ when $T=\infty$ and $h=0$ (recovering and extending the result in \cite{AMP}). The elimination of the constant comes from the initial data being in $\L^2_{x}$. The $\L^2$ off-diagonal decay was already known in this case (see the beginning of Section~\ref{sec:ode}) but we still offer a different proof.   

\subsection{Lower order coefficients in Lorentz spaces} 
\label{sec:Lorentz}

We have developed our variational approach under control of mixed Lebesgue norms on the lower order coefficients.  {We shall now explain why} these conditions can be relaxed {with hardly any effort}, using the Lorentz spaces  $\L^{p,\infty}$. 
Recall that on a measure space $(M,\mu)$,  a measurable function $f$ belongs to $\L^{p,q}$ in the case $1\le p,q<\infty$ if
$$
\|f\|_{\L^{p,q}} \coloneqq \bigg(\frac q p \int_{0}^\infty \big( t^{1/p} f^*(t)\big)^q \, \frac {\d t}t \bigg)^{1/q} <\infty
$$
and in the case $1\le p<\infty, q=\infty$ if 
$$
\|f\|_{\L^{p,q}} \coloneqq \sup_{t>0} t^{1/p} f^*(t) <\infty.
$$
Here, $f^*$ is the non-increasing rearrangement of $f$. It is known that $\L^{p,p}=\L^p$ and that $\|f\|_{\L^{p,q}}$ is non-increasing as a function of $q$, so that $\L^{p,q}\subset \L^{p,p} \subset  \L^{p,r}$ if $q\le p\le r$. Details are found in Chapter~5 of \cite{SW}. {Mixed Lorentz spaces in $(t,x)$ have been introduced by Fernandez~\cite{Fe}, who also proved that they behave in the same way as Lebesgue spaces concerning duality and multiplication (H\"older's inequality). Simple functions are dense in spaces for which all exponents are finite.}

The extension mainly relies on  the following lemma.

\begin{lem}
Let $(\tilde r_{1}, \tilde q_{1})$ be a compatible pair for lower order coefficients with admissible conjugate $(r_{1}, q_{1})$. Then $\cVdot \hookrightarrow \L^{r_{1},2}_{t} \L^{r_{2},2}_{x}$ with continuous inclusion. Consequently, if 
\begin{equation}
\label{eq:LorentzLorentz}
|\oa  |^2+|\ob  |^2+|a|\in \L^{\tilde r_{1},\infty}_{t}\L^{\tilde q_{1},\infty}_{\vphantom{t} x},
\end{equation} 
then $\cH:\cVdot \to \cVdot'$ is well-defined  and bounded and if
\begin{equation}
\label{eq:LorentzLorentzinfty}
|\oa  |^2+|\ob  |^2+|a|\in \L^{\tilde r_{1},\infty}_{t}\L^{\tilde q_{1},\infty}_{\vphantom{t} x} + \L^\infty_{t}\L^\infty_{x},
\end{equation} 
then $\cH:\cV \to \cV'$ is well-defined  and bounded.
\end{lem}

\begin{proof}  
Sobolev embeddings  are equivalent to  $\L^p-\L^q$ boundedness of Riesz potentials with $p<q$. However, it was observed by O'Neil \cite{ONeil} that such Riesz potentials also have $\L^{p,s}-\L^{q,s}$ boundedness for the same $p,q$ and all $1\le s\le \infty$. In particular, they are $\L^{p}-\L^{q,p}$ bounded as $\L^p=\L^{p,p}$. Thus, with the same relations between  $q,r$ and $\theta$ as in Lemma~ \ref{ref:embedding} but with different constants, 
\begin{equation*}
\label{eq:embedLorentz}
\|\varphi\|_{\L^{r,2}_{t}\L^{q,2}_{\vphantom{t} x}}  \le c(n,q) \|(-\Delta)^{(1-\theta) /2} \varphi\|_{\L^{r,2}_{t}\L^2_{\vphantom{t} x}}\le c(n,q) c(1,r)\|\D_{t}^{\theta/2}(-\Delta)^{(1-\theta) /2} \varphi\|_{\LL}
\end{equation*}
and the continuous inclusion for $\cVdot$ follows from \eqref{eq:embed}.

Now, we assume \eqref{eq:LorentzLorentz}. A modification of Lemma~\ref{lem:beta}, using H\"older's inequality in Lorentz spaces to guarantee that a product of three functions in $\L^{p_{i},s_{i}}$  belongs to $\L^1$ if $1 = \frac 1 {p_{1}} +\frac 1 {p_{2}}+\frac 1 {p_{3}}$ and $1 = \frac 1 {s_{1}} +\frac 1 {s_{2}}+ \frac 1 {s_{3}}$,  yields
$$
|\Angle{\beta u}{v}| \lesssim   \|u\|_{\cVdot}\|v\|_{\cVdot}.
$$
With this at hand, the boundedness of  $\cH$  from $\cVdot$ to its dual follows exactly as in  Proposition~\ref{prop:Lt}.  

Likewise, if we assume \eqref{eq:LorentzLorentzinfty}, then we proceed with the modifications as in Section~\ref{sec:inhomogneous}.
\end{proof}

{Assuming that    \eqref{eq:LorentzLorentz} holds for the compatible pair $(\tilde r_{1}, \tilde q_{1})$, one can define $\cH$ and develop the variational theory upon replacing in the definition of the space  $\Deldot^{r_{1},q_{1}}$, where $(r_{1},q_{1})$ is  the conjugate admissible pair, the mixed Lebesgue space $\L^{r_{1}}_{t}\L^{q_{1}}_{x}$  by the mixed Lorentz space $\L^{r_{1},2}_{t}\L^{q_{1},2}_{x}$.  With this precaution and these changes,  the estimates in Corollary~\ref{cor:mainreg} and the integral equalities in Lemma~\ref{lem:energy} hold. (When $(r,q)=(\infty,2)$, {there is no weakening of assumptions} and we keep working with the space $ \L^{1}_{t}\L^{2}_{x}$.)
We may proceed with the regularity Proposition~\ref{prop:regularityofDeldotr1q1solutions},  the uniqueness  Theorem~ \ref{thm:uniqueness}, the well-posedness Theorem~\ref{thm:Fg} with $g\in \L^{r',2}_{t}\L^{q',2}_{x}$ on the right-hand side,  and so on up until Theorem~\ref{thm:invertible}. It is only for Theorem~\ref{thm:causality} that we need a stronger assumption on the coefficients to guarantee causality, as we have used inequalities in the spirit of Gagliardo-Nirenberg.  It follows from Proposition~\ref{prop:GN} and H\"older inequalities that it is enough to  impose
\begin{equation}
\label{eq:LebesgueLorentz}
|\oa  |^2+|\ob  |^2+|a|\in \L^{\tilde r_{1}}_{t}\L^{\tilde q_{1},\infty}_{\vphantom{t} x}.
\end{equation}

One can also develop the corresponding inhomogeneous theory {with coefficients as in \eqref{eq:LorentzLorentzinfty}, working mainly under the Lorentz-Lorentz analogue of Assumption~$(\mathbf{D_\eps)}$. While this amounts to the same symbolic changes from Lebesgue to Lorentz spaces in $(\mathbf{D_\eps)}$ itself, the succeeding Remark~\ref{rem:Deps} has to be interpreted correctly: it says that by truncation a decomposition as in $(\mathbf{D_\eps)}$ for arbitrarily small  $\eps>0$ can be achieved starting from $|\oa  |^2+|\ob  |^2+|a|\in \L^{\tilde r_{1}, \tilde r_{2}}_{t}\L^{\tilde q_{1}, \tilde q_{2}}_{\vphantom{t} x}+\L^{\infty}_{t}\L^{\infty}_{x}$ with $1 \leq \tilde q_{2}, \tilde r_{2} < \infty$, but not when one of $\tilde q_{2}, \tilde r_{2}$ is infinite. Hence, the lower bounds assumption (A3)' becomes more interesting here.} In particular there is a statement corresponding to Theorem~\ref{thm:Cauchy} in which mixed Lebesgue norms are replaced with mixed Lebesgue-Lorentz norms on the lower order coefficients with the same pairs $(\tilde r_{1}, \tilde q_{1})$ and $$|\oa  |^2+|\ob  |^2+|a|\in \L^{\tilde r_{1}}_{t}\L^{\tilde q_{1},\infty}_{\vphantom{t} x}+\L^{\infty}_{t}\L^{\infty}_{x},$$ and in the equation the forcing term  $g$ can be taken in $ \L^{r',2}(0,T; \L^{q',2}_{\vphantom{t} x})$ when $(r,q)$ is admissible (but not when $(r,q)=(\infty,2)$, where we take $g\in  \L^{1}(0,T; \L^{2}_{x})$ as before).

All the direct consequences of this result also extend: Corollary~\ref{cor:adjoint}, Theorem~\ref{thm:ODE} and Theorem~\ref{thm:GUB}. In the latter theorem it depends on whether the local boundedness assumption is true for the particular $\Lt$ and its adjoint. Note that neither \cite{LSU} nor \cite{Ar68} consider coefficients in mixed Lebesgue-Lorentz spaces. Hence this extension is quite a new observation.

Let us give an example in the case $(\tilde r_{1},\tilde q_{1})=(\infty,n/2)$, when $n\ge 3$.   Consider  parabolic Schr\"odinger operators $$\cH=\partial_{t}-\Delta+ c(t,x)|x|^{-2}:\cVdot\to \cVdot'$$ with $c$ a complex-valued measurable and bounded function. One cannot use the assumption $\mathbf{(D_{\varepsilon})}$ here. But the classical Hardy inequality
 $$
 \int_{\R^n} \frac {|f(x)|^2}{|x|^2}\, \d x \le \bigg(\frac {2}{n-2}\bigg)^2 \int_{\R^n} |\nabla f(x)|^2\, \d x,
 $$
which follows from Hardy's one dimensional inequality \cite[Appendix~A]{St} using polar coordinates, allows one to apply Theorem~\ref{th:lb} when $\operatorname{ess\,inf} \Re c> -(\frac {n-2}{2})^2 \eqqcolon c_{n}$. Thus, $\cH$ is invertible and causal (for causality, $\Re c \ge   c_{n}$ works). One can therefore solve the Cauchy problem as above and obtain $\L^2$ off-diagonal Gaussian decay of its fundamental solution operator. In \cite{BG}, the slightly different but related question of  existence of a distributional non-negative solution to the Cauchy problem for $\partial_{t}-\Delta+ c|x|^{-2}$ with non-negative initial $\L^1$ or measure data  and $c$ a constant with $c\in [c_{n}, 0]$ is considered.

\subsection {Adding a skew-symmetric real $\BMO$ matrix to higher order coefficients}
\label{sec:BMO}
Motivated by fluid dynamics, it has become interesting to add to the usual elliptic matrix $A$  a skew-symmetric term with boundedness replaced by a BMO condition. Indeed, formally,
 pointwise lower ellipticity of the matrix $A$ does not change if one adds to it a real and skew-symmetric matrix $D(t,x)$  as 
$$
\Re \angle {D(t)\nabla u(t) }{\nabla u(t)}=0
$$
and, if $D(t,x)$ has finite $\BMO$ norm in the $x$-variable, uniformly for each $t$, then for $u,v\in \cVdot$,
$$
|\angle {D(t)\nabla u(t) }{\nabla v(t)}| \le C(n)\|D(t)\|_{\BMO_{x}}  \|\nabla u(t)\|_{\L^2_{x}}\|\nabla v(t)\|_{\L^2_{x}}
$$ 
using the $\BMO_{x}-\cH^1_{x}$ duality and compensated compactness~\cite{CLMS}.
Integrating this in time guarantees boundedness and ellipticity of the second order term in $\Lt $ if $A$ is changed to $A+D$ with $\|D\|_{\L^\infty_{t}\BMO_{x}}<\infty$.  We shall make this precise below.

All the results obtained up to this point extend with $A$ replaced by $A+D$ under this assumption on $D$. {Indeed, the extension only affects the second order term, which has been treated via bounds for the pairing $\angle{A \nabla u}{\nabla v}$ at each occurrence rather than concrete bounds on $A$, with one sole exception that we address next.}

The only subtle thing to handle is the proof of the $\L^2$ off-diagonal estimates  \eqref{eq:ODE} as in Theorem~\ref{thm:ODE} (with a less precise control on the constants $C, \omega, c_{0}$), {the difficulty being that {$D$} re-appears in lower order coefficients when using Davies'~exponential trick in \eqref{eq:conjugation}.} We first give rigorous definitions of the bracket terms to justify computations. 

We would like to  set
$$
\Angle {D\nabla u }{\nabla v} =\int_{\R}\angle {D(t)\nabla u(t) }{\nabla v(t)}\, \d t,
$$
but the inner term is usually not an honest Lebesgue integral for arbitrary $u,v\in \cVdot$. 

We introduce the set $\cE$ of functions in $\cVdot$ that are in $\cS(\R^{n+1})$  with bounded support in the $x$-variable, which is dense in $\cVdot$ (resp.\ $\cV$). Indeed, we know that $\cS(\R^{n+1})$ is dense in $\cVdot$ and from there, we can use smooth truncations. Consider $u,v\in \cE$. Let $Q$ be a cube containing their support. Set for $i,j\in \{1,\ldots, n\}^2$,  
 \begin{equation}
\label{eq:gij}
g_{i,j} (t,x) \coloneqq \partial_{x_{j}}u(t,x) \partial_{x_{i}}\overline v(t,x) -\partial_{x_{i}}u(t,x) \partial_{x_{j}}\overline v(t,x). 
\end{equation}
For each $t$, this is a bounded function with support in $Q$ and mean value zero. Hence, it is a constant multiple of an atom in $\cH^1_{x}$, the real Hardy space on $\R^n$, and the $\BMO_{x}-\cH^1_{x}$ duality is realized in this case as a Lebesgue integral 
$$
\angle {d_{i,j}(t)}{\overline{g_{i,j}}(t)}= \int_{\R^n} d_{i,j}(t,x)g_{i,j} (t,x)\, \d x.$$
As we know from \cite{CLMS} that 
$$
\|g_{i,j}(t)\|_{\cH^1_{x}} \le C(n) \|\nabla u(t)\|_{\L^2_{x}}\|\nabla v(t)\|_{\L^2_{x}},
$$
we deduce 
$$
 \frac 1 2 \int_{\R}  |\angle {d_{i,j}(t)}{\overline{g_{i,j}}(t)}|\, \d t \le C(n)\|D\|_{\L^\infty_{t}\BMO_{x}} \|\nabla u\|_{\LL}\|\nabla v\|_{\LL}.
 $$
Using the skew-symmetry of $D$, that is, $d_{i,j}=-d_{j,i}$, we can set
\begin{equation}
\label{eq:defDform}
\Angle {D\nabla u }{\nabla v}  \coloneqq \frac 1 2 \sum_{i,j} \int_{\R}  \angle {d_{i,j}(t)}{\overline{g_{i,j}}(t)}\, \d t
\end{equation}
and this form extends boundedly to $\cVdot\times \cVdot$. We now explain the necessary modifications. 

\begin{proof}[Proof of Theorem~\ref{thm:ODE}, $\BMO$-case]

To check the invertibility, it suffices as before  to look for lower bounds of 
$\e^{h}(\cH +\kappa )\e^{-h}u$.
Thus, we study again 
$\e^{h}\cH\e^{-h}$ with $h$ Lipschitz. We do not want to assume (qualitative) boundedness of $h$ this time. Hence, we first restrict the operator to $\cE$ but it extends to $\cV$ through the right-hand side of \eqref{eq:conjugation}. This allows us to take $h$ an affine real-valued function given by $h(x)=x\cdot \zeta+ c$, with $\zeta\in \R^n$ and $c\in \R$. It will be important that the gradient of $h$ is constant (as in \cite{QX, EH}). Thus, we compute
%$ \Angle {\e^{h}(\partial_{t}+\Lt +\kappa )\e^{-h}u}{v}$
 $ \Angle {\e^{h}(\cH +\kappa )\e^{-h}u}{v}$ with $u,v\in \cE$ and $h$ affine. 

\medskip

\noindent \emph{Step 1: New error estimate.} Compared to \eqref{eq:conjugation}, we get an extra term coming from the presence of $D$. A calculation yields, with $g_{i,j}$ defined in \eqref{eq:gij},
$$
\partial_{x_{j}}(\e^{-h}u) \partial_{x_{i}}(\e^{h}\overline {v}) -\partial_{x_{i}}(\e^{-h}u) \partial_{x_{j}}(\e^{h}\overline {v})= g_{i,j}+ \zeta_{i}\partial_{x_{j}}(u\overline{v})-\zeta_{j}\partial_{x_{i}}(u\overline{v}).
$$
 Next, we claim that for $f,g \in H^1_{x}$ and each $i\in \{1, \ldots, n\}$ the function $\partial_{x_{i}}(fg)$ belongs to $\cH^1_{x}$  with the estimate
$$
\|\partial_{x_{i}}(fg)\|_{\cH^1_{x}} \le C(n)(\|f\|_{\L^2_{x}}\|\nabla g\|_{\L^2_{x}}+ \|g\|_{\L^2_{x}}\|\nabla f\|_{\L^2_{x}}).
$$  
{For $f=g$ this is Proposition~3.2 in \cite{QX} and the argument applies \emph{mutadis mutandis} in the general case.} Moreover, if  $f,g$ are smooth with bounded support, then $\partial_{x_{i}}(fg)$ is a multiple of an atom in $\cH^{1}_{x}$, so that for any $b\in \BMO_{x}$,   $$\angle{b}{\partial_{x_{i}}(fg)}=\int_{\R^n}b(x) \overline{\partial_{x_{i}}(fg)(x)}\, dx,$$ 
and  the $\BMO_{x}-\cH^1_{x}$ duality gives us a bound 
$$|\angle{b}{\partial_{x_{i}}(fg)}| \leq C(n)\|b\|_{BMO_{x}}(\|f\|_{\L^2_{x}}\|\nabla g\|_{\L^2_{x}}+ \|g\|_{\L^2_{x}}\|\nabla f\|_{\L^2_{x}}).$$  Hence, for each fixed $t$, this applies to $f=u(t), g=\overline{v}(t)$ and, using again the skew-symmetry of $D$, we arrive at
$$ \Angle {D\nabla (\e^{-h}u) }{\nabla (\e^{h} {v})}= \Angle {D\nabla u }{\nabla v}+ \Angle {\beta_{D,\zeta}u}{v}$$ with 
$$
  \Angle {\beta_{D,\zeta}u}{v} \coloneqq  \sum_{i,j} \int_{\R}  \angle {d_{i,j}(t)}{\partial_{x_{j}}(\overline{u}v)(t) } \zeta_{i}\, \d t.
$$
Using the above estimate
and  Young's inequality, we see that for any $\varepsilon>0$, 
$$
|\Angle {\beta_{D,\zeta}u}{u}| \le \frac{C'_{n}|\zeta|^2\|D\|_{\L^\infty_{t}\BMO_{x}}^2}{\varepsilon} \|u\|_{\LL}^2+  \varepsilon \|\nabla u\|_{\LL}^2.
$$
{This is the required estimate for the additional error term in the presence of $D$.}

\medskip

\noindent \emph{Step 2: Off-diagonal estimate with affine perturbation.}
Now, it follows in the case (A3) that  if $\mathbf{(D_{\varepsilon_{0}})}$ holds for $\varepsilon_{0}$ small enough, then 
$\e^{h}(\cH +\kappa )\e^{-h}:\cV \to \cV'$ is invertible for $\kappa \ge 1+c_{0}(| \zeta|^2+P_{\infty}^2)$. In the case of lower bounds assumptions for $\Lt$, this is for $\kappa \ge c_{0}(1+| \zeta|^2)$. {Of course, $c_0$ now also depends on $\|D\|_{\L^\infty_{t} \BMO_{x}}$.} Moreover, in both cases, the operator norm of the inverse is bounded independently of $|\zeta|$. In conclusion, we obtain an estimate of the form
$$
\|\e^h\Gamma(t,s)\psi\|_{\L^2_{x}}\le C \e^{(\omega+c_{0}|\zeta|^2)(t-s)} \|\e^h\psi\|_{\L^2_{x}}
$$
for all $t>s$ with positive constants $C, \omega, c_{0}$. 

\medskip

\noindent \emph{Step 3: Proof of \eqref{eq:ODE}.} Let us first treat the case that $E,F$ are convex and compact sets with $d(E,F)^2>4n(t-s)$.
In this case, take  $e\in E, f\in F$  such that 
$|e-f|=d(E,F)$ and set 
\begin{align*}
 h(x)\coloneqq\frac{(f-x)\cdot (f-e)}{2c_{0}(t-s)} \quad \text{and} \quad \zeta\coloneqq\frac{e-f}{2c_{0}(t-s)}.
\end{align*}
Note that  $e$ is the orthogonal projection of $f$ onto $E$ and vice-versa. Hence, 
$$h(x)= \frac{|f-e|^2}{2c_{0}(t-s)} +\frac{(e-x)\cdot (f-e)}{2c_{0}(t-s)} \ge \frac{|f-e|^2}{2c_{0}(t-s)}$$ for $x\in E$ and $h(y)\le 0$ for $y\in F$, from which we obtain  \eqref{eq:ODE}. For the general situation where $E, F$ are arbitrary closed sets, we can assume  $d(E,F)^2>8 n (t-s)$; otherwise, we are done {with the uniform $\L^2_{x}$ bound for $\Gamma(t,s)$}. Let $Q_{k} \coloneqq [0,\sqrt{t-s}]^n +\{\sqrt{t-s}\, k\}$, $k\in \IZ^n$. Cover  $E$  with the cubes $Q_{k}$ that intersect $E$, and  $F$ with the cubes $Q_{\ell}$ that intersect $F$. We have $d(Q_{k},Q_{\ell})^2> 4 n (t-s)$. We apply the estimate just obtained for each pair $Q_{k},Q_{\ell}$ and sum in order to conclude (of course the constants change), using that the cubes form a partition of $\R^n$ up to a null set and simple discrete convolution inequalities.
\end{proof}

\begin{rem}
When $A$ is also a real matrix,  pointwise upper and lower bounds were obtained for the fundamental solution of the parabolic operator with pure second order term and matrix coefficient $A+D$ in \cite{QX}. Here, we allow complex $A$ and unbounded lower order terms and limit ourselves to an $\L^2-\L^2$ upper bound.  Some similar estimates are obtained for time-independent matrix coefficients of the form $A+D$ without lower order terms in \cite{EH}. {In principle, we could re-discover pointwise upper bounds from (the extension of) Theorem~\ref{thm:GUB}, were we able to verify the local boundedness property without resorting to itself \cite{QX}. This is yet another example that illustrates how the order of classical arguments is reversed in our work.}
\end{rem}

\subsection{Systems} The theory and its previous extensions do not change for  systems of $N$ equations, $N\ge 2$. The results are the same with pointwise ellipticity in the $x$-variable replaced by ellipticity in the G\aa rding sense (uniformly in $t$): The matrix $A(t)$ has entries being $N\times N$ matrices of bounded measurable coefficients in $(t,x)$ and 
$$
\Re \angle {A(t)\nabla \otimes  \mathbf{u}(t) }{\nabla \otimes  \mathbf{u} (t)} \ge \lambda \|\nabla \otimes  \mathbf{u}(t)\|_{\L^2_{x}}^2
$$
holds for all $t$. {Indeed, we have never used pointwise bounds and ellipticity on $A$ for means other than bounding $\angle{A \nabla u}{\nabla v}$ from above and below.}

If one wants to add a matrix of $\BMO$-type, it should be block diagonal, that is $D=(\delta_{\alpha,\beta}D^\alpha)_{1\le \alpha,\beta\le N}$, where $\delta_{\alpha,\beta}$ is the Kronecker symbol, with each $D^\alpha$ as in the previous section.

If the G\aa rding inequality comes with a negative $\L^2$ norm on $ \mathbf{u}(t)$, then one should apply the inhomogeneous theory. We leave details to the reader.

\section{Higher order problems on full space}
\label{sec:HO}

It is mainly a matter to fix algebraic notation as the analysis done for second order parabolic operators goes through almost \textit{verbatim} for higher order problems on full space. We give details of the setup and  sketch the main points, following faithfully what was done for second order problems. Given our omission of proofs, this section should be considered as an announcement of results, the verification of which is left to the interest readers. Results in this section also provide the generalization of the theory for second order elliptic parts when the compatible pairs are allowed to vary with the coefficients as mentionned earlier.

\subsection{The elliptic operator}
 \label{sec:hooperator}
The elliptic part $\Lt $ is now $2m$-th order, $m\ge 2$, given formally by
\begin{equation}
\label{eq:higherorderL}
\Lt u= \sum (-1)^{|\alpha|}\partial^\alpha (a_{\alpha,\beta}(t,x)\partial^\beta u),
\end{equation}
where the sum is taken over pairs $(\alpha,\beta) $ of multi-indices with $0\le |\alpha|, |\beta| \le m$ and $\partial^\alpha$ are partial derivatives in the $x$-variable of order $\alpha$. We have set  $|\alpha|=\alpha_{1}+\cdots +\alpha_{n}$ for $\alpha=(\alpha_{1},\ldots, \alpha_{n})$.

\subsection{Variational space}
 \label{sec:hovarspace}
For the homogeneous theory, the space $\cVdot$ becomes the space of tempered distributions $u$ having Fourier transforms  $(|\xi|^{2m}+|\tau|)^{-1/2}g$ for some (unique) $g\in \LL$, equipped with the norm $\|u\|_{\cVdot} \coloneqq (2\pi)^{-(n+1)/2} \|g\|_{\LL}$. As in the case of order $2$, this space realizes $\L^2_{t}\Hdot^m_{x}\cap \Hdot^{1/2}_{t}\L^2_{x}$ defined within tempered distributions modulo polynomials with norm 
$$\big(\|(-\Delta)^{m/2}u\|^2_{\LL}+ \|\D_{t}^{1/2}\! u\|_{\LL}^2\big)^{1/2}\sim
\sum_{|\alpha|=m} \|\partial^{\alpha}u\|_{\LL}+ \|\D_{t}^{1/2}\! u\|_{\LL}.$$

\subsection{Embeddings}
 \label{sec:hoembeddings} 
 
For an arbitrary collection $({\bf r,q})$ of pairs of exponents $(r^{\alpha,\beta}, q^{\alpha,\beta})$ in $[1,\infty]^2$ indexed by multi-indices $(\alpha,\beta) $ with $0\le |\alpha|, |\beta| \le m$,  we set
$$
\Deldot^{{\bf r,q}} \coloneqq \bigcap_{0\le |\alpha|,|\beta|\le m} \{u\in \cD'(\R^{n+1})\, :\, \partial^\alpha u \in \L_{t}^{r^{\alpha,\beta}}\L_{\vphantom{t} x}^{q^{\alpha,\beta}}\}.
$$
For each $\alpha$, there could be several mixed spaces involved to which $\partial^\alpha u$ belongs, parametrized by the multi-indices $\beta$.
If all pairs  of exponents belong to  $[1,\infty)^2$, then the dual space of  $\Deldot^{{\bf r,q}}$ in the duality extending the $\LL$ inner product can be identified with
\begin{equation}
\label{eq:dual}
(\Deldot^{{\bf r,q}})' \coloneqq \sum_{0\le |\alpha|,|\beta|\le m}  \partial^\alpha \L_{t}^{(r^{\alpha,\beta})'}\L^{(q^{\alpha,\beta})'}_{\vphantom{t} x} = \sum_{0\le |\alpha|,|\beta|\le m}   \L_{t}^{(r^{\alpha,\beta})'}\big(\partial^\alpha\L_{\vphantom{t}  x}^{(q^{\alpha,\beta})'}\big) =: \dot\Sigma^{{\bf r',q'}},
\end{equation}
with the same interpretation as in the case $m=1$ in Section~\ref{sec:preliminaries} and $({\bf r',q'})$ is the collection of pairs of H\"older conjugates obtained from $({\bf r,q})$. When all pairs of exponents in $({\bf r,q})$ belong to   $(1,\infty]^2$, then the dual space of  $\Sigma^{{\bf r',q'}}$  can be identified with $\Deldot^{{\bf r,q}}$ for the same duality.  In particular, $\Deldot^{{\bf r,q}}$ is reflexive when all pairs belong to $(1,\infty)^2$.

Sobolev embeddings for partial derivatives $\partial^\alpha$ in the spirit of Lemma~\ref{ref:embedding} are as follows: If $u\in \cVdot$ and $0\le |\alpha|\le m$, then $\partial^\alpha u\in \L^{r}_{t}\L^{q}_{\vphantom{t} x}$ with $\|\partial^\alpha u \|_{\L^{r}_{t}\L^{q}_{\vphantom{t} x}} \lesssim \|u\|_{\cVdot}$ provided  
\begin{equation}
\label{eq:highorderadmissiblepartialalpha}
(r,q) \in [2,\infty)^2, \quad \frac 1{r}+ \frac n{2mq}=  \frac {n+2|\alpha|}{4m}. 
\end{equation}
We say that  pairs $(r,q)$ with the condition \eqref{eq:highorderadmissiblepartialalpha} are {\bf admissible for 
$\partial^\alpha$}.  When $|\alpha|=m$,  the only admissible pair  for  $\partial^\alpha$ is $(2,2)$. If $|\alpha|<m$, then there is more flexibility. A collection    $({\bf r,q})$    of pairs  
$(r^{\alpha,\beta}, q^{\alpha,\beta})$ indexed by multi-indices $(\alpha,\beta)$ with  $0\le |\alpha|,|\beta| \le m$ is {\bf admissible} (resp.\ {\bf super admissible})  if  each pair $(r^{\alpha,\beta}, q^{\alpha,\beta})$ is admissible for $\partial^\alpha$ (resp.\ admissible for $\partial^\alpha$ when $\alpha\ne 0$ and admissible for $\partial^\alpha$  or equal to $(\infty, 2)$ when $\alpha =0$). In particular, the continuous inclusion $\cVdot \hookrightarrow \Deldot^{{\bf r,q}}$ holds for all admissible collections.

\subsection{Variational approach}
 \label{sec:hoapproach}
 
When $0\le |\alpha|,|\beta| \le m$,  critical mixed Lebesgue  spaces  $\L_{t}^{r({\alpha,\beta})}\L^{q({\alpha,\beta})}_{\vphantom{t}  x}$ for  the coefficients $a_{\alpha,\beta}$ are given by the relations
\begin{equation}
\label{eq:highordercompatible}
(r({\alpha,\beta}), q({\alpha,\beta}))\in (1,\infty]^2, \quad  \frac 1{r({\alpha,\beta})}+ \frac n{2mq({\alpha,\beta})}= 1- \frac {|\alpha|+|\beta|}{2m}. 
\end{equation}
We say that $(r({\alpha,\beta}), q({\alpha,\beta}))$ is a compatible pair for $(\alpha,\beta)$. If such a pair  is given, any choice  of  admissible pairs $
(r^{\alpha,\beta}, q^{\alpha,\beta}) $  and $
(r_{\beta, \alpha}, q_{\beta, \alpha}) $ for  $\partial^\alpha$ and
$\partial^\beta$,  respectively, yields
\begin{equation}
\label{eq:higherorderHolder}
\|a_{\alpha,\beta}\,\partial^\beta u\, \partial^\alpha v\|_{ \L^1_{t}\L^1_{x}} \le  \|a_{\alpha,\beta}\|_{\L_{t}^{r({\alpha,\beta})}\L_{\vphantom{t}  x}^{q({\alpha,\beta})}} \|u\|_{\L_{t}^{r_{\beta, \alpha}}\L_{\vphantom{t}  x}^{q_{\beta, \alpha}}} \|v\|_{\L_{t}^{r^{\alpha,\beta}}\L_{\vphantom{t}  x}^{q^{\alpha,\beta}}}
\end{equation}
provided that
\begin{equation}
\label{eq:highorderadmissiblecompatible}
\frac 1{q({\alpha,\beta})}+ \frac 1{q_{\beta, \alpha}} + \frac 1{q^{\alpha,\beta}}=1 \quad \& \quad \ \frac 1{r({\alpha,\beta})}+ \frac 1{r_{\beta, \alpha}}+ \frac 1{r^{\alpha,\beta}}=1. 
\end{equation}
Note that this covers the higher order derivatives when $|\alpha|=|\beta|=m$, where $a_{\alpha,\beta}$ are bounded and the admissible pairs  for  $\partial^\alpha$ and $\partial^\beta$ are $(2,2)$. If $|\alpha|+|\beta|<2m$, then we have several choices.

We come to the definition of $\cH$ on $\cVdot$. First, we fix {once and for all} a collection
$({\bf \tilde r_{1},\tilde q_{1}})$  of compatible pairs $(r({\alpha,\beta}), q({\alpha,\beta}))$ for $(\alpha,\beta)$ with   $0\le |\alpha|,|\beta| \le m$.  We assume $a_{\alpha,\beta}\in \L_{t}^{r({\alpha,\beta})}\L^{q({\alpha,\beta})}_{\vphantom{t}  x}$ and set\footnote{With this parametrization, when $m=1$, we have considered the compatible collection consisting of $(\tilde r_{1},\tilde q_{1})$ when $|\alpha|=|\beta|=0$,  $(2\tilde r_{1}, 2 \tilde q_{1})$ when $|\alpha| +|\beta|=1$ and $(\infty,\infty)$ when $ |\alpha|=|\beta|=1$.} 
$$
\Lambda\coloneqq \sup_{|\alpha|=|\beta|=m}\|a_{\alpha,\beta}\|_{\infty}\quad \& \quad P_{\bf \tilde r_{1}, \tilde q_{1}} \coloneqq \sum_{ |\alpha|+|\beta|<2 m} \|a_{\alpha,\beta}\|_{\L_{t}^{r({\alpha,\beta})}\L_{\vphantom{t}  x}^{q({\alpha,\beta})}}.
$$
Secondly, we need to work with two collections of admissible pairs, one for $\partial^\alpha$ denoted by  $({\bf r_{1},q_{1}})=
(r^{\alpha,\beta}, q^{\alpha,\beta})_{\alpha,\beta} $, the other one for $\partial^\beta$ denoted by $({\bf \bar r_{1},\bar q_{1}})=(r_{\beta, \alpha}, q_{\beta, \alpha})$, {both} satisfying in addition 
\eqref{eq:highorderadmissiblecompatible}.\footnote{With this parametrization, there is no notion of {unambiguously} defined conjugate collection  associated to the compatible collection for the coefficients. Besides, there are many  possible choices of collections $({\bf r_{1} , q_{1}})$ and $({\bf \bar r_{1},\bar q_{1}})$ and we fix one once and for all.} 
 We define accordingly  the space $\Deldot^{{\bf r_{1},q_{1}}}$ as above, and, taking into account the symmetric roles of multi-indices $\alpha,\beta$, we set
$$
\Deldot^{{\bf \bar r_{1},\bar q_{1}}} \coloneqq \bigcap_{0\le |\alpha|,|\beta|\le m} \{u\in \cD'(\R^{n+1})\, :\, \partial^\beta u \in \L_{t}^{r_{\beta,\alpha}}\L_{\vphantom{t}  x}^{q_{\beta,\alpha}}\}.
$$
{With this notation and using \eqref{eq:higherorderHolder} and \eqref{eq:highorderadmissiblecompatible}, we see that $\Lt$ in \eqref{eq:higherorderL} satisfies}
$$|\Angle {\Lt u} {v}|
  \le (\Lambda  + P_{\bf \tilde r_{1}, \tilde q_{1}}) \|u\|_{\Deldot^{{\bf \bar r_{1},\bar q_{1}}}}\|v\|_{\Deldot^{{\bf r_{1},q_{1}}}},
 $$ 
so that $\Lt$ acts boundedly from $\Deldot^{{\bf \bar r_{1},\bar q_{1}}}$ into $(\Deldot^{{\bf r_{1},q_{1}}})'$. From the continuous inclusions $\cVdot \hookrightarrow \Deldot^{{\bf \bar r_{1},\bar q_{1}}}$ and $(\Deldot^{{\bf r_{1},q_{1}}})' \hookrightarrow \cVdot'$ for admissible collections, we obtain that $\cH=\partial_{t}+\Lt: \cVdot \to \cVdot'$ is well-defined and bounded. 

\subsection{Main regularity estimates}
 \label{sec:homainreg}
   
We can now state the main regularity lemma. We set $\nabla^mu=(\partial^\alpha u)_{ |\alpha|= m}$ for simplicity. 

\begin{lem}
\label{lem:mainreghigh} Let $u\in \cD'(\R^{n+1})$. Assume $\nabla^m u\in \LL$ and $\partial_{t}u\in\dot\Sigma^{{\bf r',q'}}$, where $({\bf r,q})$ is  a super admissible collection. Then, there is a polynomial $P$ in the $x$-variable with degree not exceeding $m-1$, such that 
$u-P\in \C_{0}(\L^2_{x})$ and   
$$
 \sup_{t\in \R}\|u(t)-P\|_{\L^2_{x}} \le C(\|\nabla^m u\|_{\LL}+ \|\partial_{t}u\|_{\dot\Sigma^{\bf r', q'} })
$$
with some constant $C$ independent of $u$ and $P$. Moreover, if the collection $({\bf r,q})$ is admissible, then $u-P\in \cVdot$ with the same estimate on $\|u-P\|_{\cVdot}$.
\end{lem}

The proofs are similar to that of Section~\ref{sec:mainregularityestimates}, replacing $-\Delta$ by $(-\Delta)^m$: for example, one uses
$$\Hmtheta \coloneqq \{\D_{t}^{\theta/2}(-\Delta)^{m(1-\theta) /2} g \, : \,  g\in \LL\}$$ as each $\L_{t}^{(r^{\alpha,\beta})'}\big(\partial^\alpha\L_{\vphantom{t}  x}^{(q^{\alpha,\beta})'}\big)$ embeds into one $\Hmtheta$ for some $\theta\in [0,1)$ when $(r^{\alpha,\beta},q^{\alpha,\beta})$ is an admissible pair for $\partial^\alpha$. 

The integral equalities of Section~\ref{sec:integralequalities} are also proved similarly.

\subsection{The resulting theory}
 \label{sec:hotheory}

The invertibility of $\cH$ is again enough to develop the uniqueness and existence  of $\Deldot^{{\bf  \bar r_{1}, \bar q_{1}}}$-solutions and  to produce Green operators in order to obtain representations. For example, the uniqueness statement    corresponding to Theorem~\ref{thm:uniqueness} becomes that whenever $\cH$ is invertible, then any  $u\in\Deldot^{{\bf  \bar r_{1}, \bar q_{1}}}$  such that  $\partial_{t}u+\Lt u=0$ vanishes. 

The invertibility for $\cH$ can be checked provided there is a G\aa rding inequality in the spirit of \eqref{eq:Aelliptic} for the leading coefficients, that is,
 \begin{equation}
 \label{eq:hogarding}
 \Re \sum_{ |\alpha|=|\beta|= m} \int_{\R} \angle{a_{\alpha,\beta}(t)\partial^\beta u(t)}{\partial^\alpha u(t)}\, \d t \ge \lambda  \|\nabla^m u\|_{\LL}^2,
\end{equation}
and for the lower order coefficients, smallness  of $P_{\bf \tilde r_{1}, \tilde q_{1}}$ is needed. Alternatively, invertibility can also follow from lower bounds on $\Lt$ as in Theorem~\ref{th:lb}.

If \eqref{eq:hogarding} holds and the leading part of $\cH$ is a pure $2m$-order operator, then one can work with the uniqueness class of  $\L^2_{t}\Hdot^{m}_{\vphantom{t} x}$-solutions, which is defined analogously to Section~\ref{sec:puresecondorder}. 

If the G\aa rding inequality comes with a negative $\LL$ norm on $u$, or $P_{\bf \tilde r_{1}, \tilde q_{1}}$ is not small enough, or  bounded coefficients are added to the lower order coefficients while $P_{\bf \tilde r_{1}, \tilde q_{1}}$ remains small, or again that a lower bound is assumed on $\Lt$, then one uses inhomogeneous spaces to prove invertibility of $\cH+\kappa:\cV \to \cV'$ for large enough $\kappa$. 

Using the improvement of \eqref{eq:higherorderHolder} with the mixed Lorentz spaces $\L_{t}^{r({\alpha,\beta}),\infty}\L_{\vphantom{t}  x}^{q({\alpha,\beta}),\infty}$ replacing the mixed Lebesgue spaces $\L_{t}^{r({\alpha,\beta})}\L_{\vphantom{t}  x}^{q({\alpha,\beta})}$ for the lower order coefficients is possible and  $P_{\bf \tilde r_{1},\tilde q_{1}}$ is modified accordingly. This covers, for example, power weights $c(t,x)|x|^{-n/q(\alpha,\beta)}$ with $c$ bounded above and below, when $r(\alpha,\beta)=\infty$. For forcing terms and solutions, the mixed Lorentz spaces $\L_{t}^{r,2}\L_{\vphantom{t}  x}^{q,2}$ may replace the mixed Lebesgue spaces $\L_{t}^{r}\L_{\vphantom{t}  x}^{q}$ with the same collections of pairs.

The proof of causality uses  a variant of Gagliardo--Nirenberg inequalities and requires mixed Lebesgue-Lorentz norms. A quick proof of this variant can be found in  Proposition~\ref{prop:GN}. 

The Cauchy problem can be stated and proved in a similar fashion. The fundamental solution operator can be identified with exponentially weighted Green operators as before. Under the same assumptions guaranteeing invertibility and causality of $\cH+\kappa$, the fundamental solution operator enjoys $\L^2$ off-diagonal estimates.   Lipschitz bounded functions of the $x$-variable are replaced by the regular functions  considered by Davies in \cite{Da1} for the case of time-independent parabolic operators with bounded lower terms.  This is more complicated here, because we take unbounded coefficients. But we can obtain lower bounds of perturbed operators 
$\e^h (\cH+\kappa)\e^{-h}$
using successive and tedious decompositions of the perturbed coefficients as in the condition $\mathbf{(D_{\varepsilon})}$, where $\kappa$ is chosen on the order of $c+c\|\nabla h\|_{\infty}^{2m}$ and optimization in $h$ gives  exponential decay in  $(d(E,F)^{2m}/|t-s|)^{1/(2m-1)}$. 

Extensions to systems work without difficulty.

\section{Second order problems with lateral boundary conditions}
\label{sec:lateral}

In this short section we describe an extension of our theory to second order parabolic problems on cylinders with lateral boundary conditions. As the previous section, {this should be considered an announcement of results}. Working out the details along our sketch and extending the results to  systems is again left to interested readers. Adaptation to higher order problems is likely to hold but would require further work. 

\subsection{The geometric setup}
 \label{sec:bcsetup}

We work on $\R\times \Omega$, where $\Omega$ is an open set in $\R^n$, and encode lateral boundary conditions through the choice of a variational space $V$ with 
$$
\W^{1,2}_{0}(\Omega) \subset V \subset \W^{1,2}(\Omega),
$$
equipped with the Hilbertian norm 
$$
\|\psi\|_{V}\coloneqq \big(\|\psi\|_{\L^2(\Omega)}^2+ \|\nabla \psi\|_{\L^2(\Omega)}^2\big)^{1/2}. 
$$
The cases $V = \W^{1,2}_{0}(\Omega)$ and $V = \W^{1,2}(\Omega)$ correspond to (pure) lateral Dirichlet and Neumann boundary conditions. Spaces in between can be used to model for instance a mix of the two.

The only geometric assumption that we make on $\Omega$ are (fractional) Sobolev embeddings for $V$. We write $[\cdot\,,\cdot]_\theta$ for the complex interpolation bracket, see for example Section~1.9 in~\cite{Tr}.

\medskip

{
\paragraph{\bf Assumption $\mathbf{(V)}$}
We assume that there exists an \emph{embedding dimension} $d \in [1,\infty)$ with the following property: For all $\theta \in [0,1]$ and $2 \leq q < \infty$ such that $\frac{1}{2} - \frac{1-\theta}{d} = \frac{1}{q}$, we have 
\begin{equation}
\label{eq:Sobolev_frac}
	[\L^2(\Omega), V]_{1-\theta} \hookrightarrow \L^q(\Omega)
\end{equation}
with continuous inclusion.

%\medskip

\begin{rem}
 If $\mathbf{(V)}$ holds for one choice of $d$, then it holds for all larger choices. Hence, it will be advantageous to take $d$ as small as possible. The primary example we have in mind  is { when} $\theta = 0$ is allowed above (hence $d >2$) and therefore $V$ itself satisfies the Sobolev embedding
\begin{align}
	\label{eq:Sobolev}
	V \hookrightarrow \L^{2d/(d-2)}(\Omega).
\end{align}
In this case, the other embeddings required in $\mathbf{(V)}$ follow by complex interpolation. However, already for $\Omega = \R^2$ the optimal choice is $d=2$ and by fractional Sobolev embeddings we have indeed $\mathbf{(V)}$ with $d=2$ and that \eqref{eq:Sobolev_frac} is satisfied when $\theta\in (0,1]$, even though we do not have \eqref{eq:Sobolev}. In ambient dimension $n=1$ {and when $\Omega$ is an interval}, $\mathbf{(V)}$ holds with embedding dimension  $d=1$ no matter what the boundary conditions are and \eqref{eq:Sobolev_frac} is satisfied in the limited range $\theta \in (\frac 12, 1]$ due to the constraint $2\le q<\infty$.
\end{rem}

\begin{rem}
Testing \eqref{eq:Sobolev_frac} with cut-off functions $\psi$ for arbitrarily small balls contained in $\Omega$, reveals that $d$ cannot be smaller than the ambient dimension $n$. In principle, $d$ can be  larger than $n$. When $V=\W^{1,2}_{0}(\Omega)$ or when $\Omega$ is sufficiently regular, the value $d=n$ is obtained. For a discussion of irregular sets that satisfy $\mathbf{(V)}$, we refer to the introduction of \cite{Be} or \cite[Ch.~4]{AF} for the case $d>n$ and to \cite[Sec.~3]{Eg} for mixed Dirichlet-Neumann boundary conditions.
\end{rem}
}

\subsection{Variational space}
 \label{sec:bcvarspace}
The variational space is now $\cV \coloneqq  \L^2_{t}V \cap \H^{1/2}_{t}\L^2_{x}$, equipped with the Hilbertian norm $\|u\|_{\cV}$ given by
$$
\|u\|_{\cV}\coloneqq \big(\|u\|_{\LL}^2+ \|\nabla u\|_{\LL}^2+ \|\D_{t}^{1/2}\!u\|^2_{\LL}\, \big)^{1/2},
$$
where in this section we use the notation $\L^p_{x} \coloneqq \L^p(\Omega)$. {Let $-\Delta_{V}$ be the positive self-adjoint operator built from the sesquilinear form $(\psi, \tpsi) \mapsto \angle{\nabla \psi}{\nabla \tpsi}$ on $V\times V$. We let $S=(1-\Delta_{V})^{1/2}$, so that by Kato's second representation theorem~\cite{Kato} the domain of $S$ is equal to $V$ with $\|S\psi\|_{\L^2_{x}}= \|\psi\|_{V}$ for all $\psi\in V$. It is also known that the domains of the  powers $S^\alpha$, $\alpha\in \R$,  interpolate by the complex method~\cite{AMcN}.

\subsection{Embeddings}
 \label{sec:bcembeddings}	
We begin by developing the theory along the lines of Section~\ref{sec: full space}. As the reader may have already observed, we have used the full strength of distribution theory only in the $t$-variable, whereas in the $x$-variable distributions and test functions have mostly appeared for the sake of simple arguments but they could have been replaced by spectral theory for the Laplacian and functions in less regular spaces such as $\Deldot^{r,q}$. This is our general guideline.

Our first task is to identify the pairs $(r,q)$ for which we have the embedding
\begin{equation}
\label{eq:embedding}
\cV \hookrightarrow \Del^{r, q}: = \L^2_{t}V \cap \L^{r}_{t}\L^{q}_{\vphantom{t} x},
\end{equation}
where $\Del^{r, q}$ is equipped with the  norm 
$
\|u\|_{\Del^{r, q}}: =\|u\|_{\L^2_{t}V}
+ \|u\|_{\L^{r}_{t}\L^{q}_{\vphantom{t} x}}.
$
We set $\Sigma^{r',q'}= \L^2_{t}V'+ \L^{r'}_{t}\L^{q'}_{\vphantom{t} x}$ with the usual infimum norm.

\begin{lem}
\label{ref:embeddingV}
Under Assumption~$\mathbf{(V)}$, the embedding \eqref{eq:embedding}, and by duality $\Sigma^{r',q'} \hookrightarrow \cV'$, hold if $\frac 1 {r}+\frac {d}{2q}= \frac d4$ with $2\le r,q < \infty$. 
\end{lem}

\begin{proof}
We modify the proof of Lemma~\ref{ref:embedding}. In order to prove \eqref{eq:convexity} we have previously used the Fourier transform on $\L^2(\R^n)$ to obtain unitary equivalence of $-\Delta$ to a multiplication operator $m(\xi) = |\xi|^2$.  Here, we use the spectral theorem for $(1-\Delta_{V})$ and the same argument applies. As for \eqref{eq:embed}, the required Sobolev inequality in the spatial variable is precisely our Assumption~$\mathbf{(V)}$ and now $d$ instead of $n$ plays the role of the dimension. Hence, \eqref{eq:embedding} holds under the given conditions on $(r,q)$.
\end{proof}

\subsection{Variational approach}
 \label{sec:bcapproach}
Pairs that satisfy the relation in Lemma~\ref{ref:embeddingV} will be called {\bf admissible pairs} (for the boundary value problems under assumption $\mathbf{(V)}$). Once again, admissible pairs $(r,q)$  are conjugates of pairs $(\tilde r, \tilde q)$, called {\bf compatible pairs for lower order coefficients}, which are defined by
$$\frac 1 {\tilde r}+\frac {d}{2\tilde q}= 1\quad  \&  \quad  1<\tilde r, \tilde q \le \infty.$$
The conjugation rule is  $(r, q)=(2(\tilde r)',2(\tilde q)')
$ as in \eqref{eq:associatedpair}.
Fixing once and for all a compatible pair $(\tilde r_{1}, \tilde q_{1})$ for lower order coefficients,
we define the parabolic operator $\cH$ on $\cV$ by the sesquilinear form
$$
\Angle {\cH u}{v} = \Angle {\partial_{t}u} {v} +  \int_{-\infty}^\infty   \angle{A(t)\nabla u(t)}{\nabla v(t)} +\angle{ \beta u(t)}{v(t)}\, \d t,
$$
where as before, $\beta$  includes the lower order terms,  $\angle{\cdot \,}{\cdot}$ is now the inner product on $\L^2_{x}=\L^2(\Omega)$ and $\Angle{\cdot \,}{\cdot}$ the sesquilinear duality extending the $\L^2_{t}\L^2_{x}$ inner product.  
As $\cD(\R; V)$ is a dense subspace of $\cV$, we have
$$
\Angle {\cH u}{v}= -\Angle {u}{\partial_{t} v} + \Angle {\Lt  u}{v}
$$
for all $u\in \cV$ and $v\in \cD(\R; V)$, where $\Lt $ is defined by the integral above. 
H\"older's inequality, which is dimensionless in terms of exponents, yields
\begin{align*}
 |\Angle {\Lt u} {v}| \le \|A\|_{\infty} \|\nabla u\|_{\LL}\|\nabla v\|_{\LL} + P_{\tilde  r_{1},\tilde q_{1}} \|u\|_{\Del^{ r_{1}, q_{1}}}\|v\|_{\Del^{ r_{1}, q_{1}}}
\end{align*}
with $P_{\tilde  r_{1},\tilde q_{1}}$ as in   \eqref{eq:P}, so that 
$$
|\Angle {\cH u}{v}|\le C \|u\|_{\cV}\|v\|_{\cV}.
$$
Hence, using $\cH$ gives access to weak solutions in $\L^2_{t}V$ of $\partial_{t} u+\Lt u=w$ with lateral boundary conditions prescribed by $V$.

\subsection{Main regularity estimates}
 \label{sec:bcmainreg}
 
Modifying the proofs of Section~\ref{sec:mainregularityestimates} on replacing  $-\Delta$ systematically by $(1-\Delta_{V})$, the main regularity lemma becomes the following statement.

\begin{lem}
\label{lem:mainregV} Let $u\in \cD'(\R; V')$  with $u\in \L^2_{t}V$  and $\partial_{t}u\in\Sigma^{r', q'} $ for $(r,q)$ an admissible pair  or  $(r,q)=(\infty,2)$ under Assumption $\mathbf{(V)}$. Then
 $u\in  \C_{0}(\L^2_{x})$ and for some   constant $C<\infty$ independent of $u$, 
$$
 \sup_{t\in \R}\|u(t)\|_{\L^2_{x}} \le C(\|u\|_{\L^2_{t}V}+ \|\partial_{t}u\|_{\Sigma^{r', q'} }).
$$
Moreover, if $(r,q)$ is admissible, then $u\in \cV$ with the same estimate on {$\|u\|_{\cV}$}.
\end{lem}

\begin{proof} We indicate the main changes. 

\medskip

\paragraph{\itshape Modification of the uniqueness Lemma~ \ref{lem:uniqueness}}

This is now stated for $u\in \cD'(\R; V')$ such that $\partial_{t}u + (1-\Delta_{V}) u=0$ in $\cD'(\R; V')$: if   $u\in \L^2_{t}V$, then $u=0$. Indeed, we see that $\partial_{t}u \in \L^2_{t}V'$. By Lions' embedding theorem we have $u\in \C_{0}(\L^2_{x})$, and testing the equation against $u$ yields that $\Angle {(1-\Delta_{V})u}{u}=0$, which implies $u=0$.

\medskip

\paragraph{\itshape Modification of the embedding in Lemma~\ref{lem:H-theta} }

Here, we have to show that with $\theta=1-\frac 2 r$ we have the continuous inclusion
\begin{align*}
	(\L^{r}_{t}\L^{q}_{\vphantom{t} x})' \hookrightarrow \HthetaV,
\end{align*}
 where 
\begin{align*}
 \HthetaV \coloneqq\{\D_{t}^{\theta/2}S^{1-\theta} g \, : \,  g\in \LL\} \quad \text{with} \quad \|w\|_{\HthetaV} \coloneqq \|g\|_{\LL}.
\end{align*}
In the definition of this space, $S^{1-\theta}$ is extended dy duality to a map from $\L^2$ into the dual of $V$ with respect to the $\L^2_x$ duality. This uses that $V$ is the domain of $S$ and that $0 \leq 1-\theta\leq1$. Hence, we are working with  a subspace of $ \cD'(\R; V')$. The embedding itself is a repetition of the proof of Lemma~\ref{lem:H-theta} except that now we take $G = \cS(\R; V)$ as dense subset. This is where we use assumption $\mathbf{(V)}$.

\medskip

\paragraph{\itshape Modification of the stronger regularity statement in Lemma~ \ref{lem:maxreg}}

We need a new dense subspace $G_0$, which we can take as $G_0 \coloneqq \cS_{00}(\R; \dom(\Delta_V^2))$ here. 

Step~1 then goes through \emph{mutadis mutandis} if we understand Fourier in the $x$-variable as a special case of the spectral theorem for the Laplacian, compare with the proof of Lemma~\ref{ref:embeddingV}.

Step~2 remains unchanged.

For Step~3, we obtain $v'(t)+ (1-\Delta_{V})v(t)=w(t)$ in $\L^2_x$ for all $t \in \R$, whenever $g\in G_{0}$. The equation can also be interpreted in $\cD'(\R; V')$ for the test functions $\phi\in \cD(\R; V)$: this interpretation passes to the limit for $g\in \LL$, thanks to Step~1. 

Lastly, Step~4, again for $g\in G_{0}$, has been a Fourier transform argument and now its use in the $x$-variable should be replaced by the spectral theorem. From this perspective, the proof is the same as before.

\medskip

\paragraph{\itshape End of proof} Modifications of Proposition~\ref{prop:maxregL1L2} and of the corollaries that follow are proved similarly with constant $c=0$.
\end{proof}

\subsection{The resulting theory}
 \label{sec:bctheory}
From this point on, the theory can be developed similar to the \emph{inhomogeneous setting} of Section~\ref{sec:inhomogneous}, assuming  $\mathbf{(V)}$. The two exceptional topics are pointwise Gaussian upper bounds (Section~\ref{sec:Gaussian_pointwise}) and $\BMO$-coefficients in the principal part (Section~\ref{sec:BMO}), the extension of which will require finer geometrical properties of the underlying domain $\Omega$ and should be considered open at this point.

The rest works out smoothly, as long as we assume uniform ellipticity in the sense of G\aa rding: There should exist $\lambda>0$ and $c_0 \in \R$ such that for almost every $t$ and every $w \in V$ we have
\begin{equation}
\label{eq:gardingV}
\Re \angle{A(t)\nabla w}{\nabla w } \ge \lambda  \|\nabla w\|_{\L^2_{x}}^2 - c_0 \|w\|_{\L^2_{x}}^2.
\end{equation}
Then, we can work with lower order coefficients in Lebesgue-Lebesgue mixed spaces with the assumption $\mathbf{(D_{\varepsilon})}$. This gives access to representation by Green operators  for the inverse of $\cH+\kappa$ for appropriate $\kappa\ge 0$, causality and fundamental solution operators for the Cauchy problem. The proof of $\L^2$ off-diagonal estimates can be adapted if $V$ is invariant under multiplication with bounded Lipschitz functions. For example, the variational spaces for mixed Dirichlet-Neumann boundary conditions have this property~\cite[Lem.~4]{Eg}. 

If \eqref{eq:gardingV} comes with $c_{0}=0$ and the leading part of $\cH$ is a pure second order operator, then one can also develop the theory in the class  $\L^2_{t}V$ similar to Section~\ref{sec:puresecondorder}.

Finally, for the extension of the definition of $\cH$ when coefficients belong to   mixed Lorentz spaces, we can use the following self-improvement property to treat all compatible pairs $(\tilde r_{1},\tilde q_{1})$ with $\tilde r_{1}<\infty$.

\begin{lem}
\label{lem:self-improvement V} If Assumption~{$\mathbf{(V)}$} holds, then 
$\L^q(\Omega)$ can be replaced with $\L^{q,2}(\Omega)$ in  \eqref{eq:Sobolev_frac} when $\theta>0$.
\end{lem}

\begin{proof}
Let us fix $\theta \in (0,1]$ with corresponding Lebesgue exponent $q$. By open-endedness, we pick $\vartheta \in (0,\theta)$ with corresponding larger exponent $r$. By \eqref{eq:Sobolev_frac} we have a continuous inclusion
\begin{align*}
 [\L^2_x, V]_{1-\vartheta} \hookrightarrow \L^r_x.
\end{align*}
For Hilbert spaces the complex method agrees with the $(\cdot\,, \cdot)_{\theta,2}$-real method~\cite[Sec.~6]{AMcN}. With $\sigma \coloneqq \frac{1-\theta}{1-\vartheta} \in (0,1)$ we obtain the required continuous inclusion
\begin{align*}
	[\L^2_x, V]_{1-\theta} 
	= (\L^2_x,V)_{1-\theta,2} 
	= (\L^2_x, (\L^2_x,V)_{1-\vartheta,2})_{\sigma,2}
	\hookrightarrow (\L^2_x, \L^r_{\vphantom{x} x})_{\sigma,2}
	= \L^{q,2}_{x}.
\end{align*}
The second equality is the reiteration theorem~\cite[Sec.~1.10.2]{Tr} and the final equality follows from the real interpolation property for Lebesgue spaces~\cite[Sec.~1.18.6]{Tr} and the relation $\frac{1-\sigma}{2} + \frac{\sigma}{r} = \frac{1}{q}$.
\end{proof}

With the previous lemma at hand, the extension of the definition of $\cH$ with coefficients in mixed Lorentz spaces can be carried out as before for compatible pairs for the coefficients with $\tilde r_{1}<\infty$. The case $r_{1}=2$  for the conjugate admissible pair $(r_{1},q_{1})$ is not covered by this statement. Invertibility can be shown under Lorentz-Lorentz mixed norms for the  lower order coefficients and causality follows under Lebesgue-Lorentz mixed spaces. 

In order to include Lorentz spaces for compatible pairs with $\tilde r_{1}=\infty$, which is probably the most interesting case in applications, the improvement in Lemma~\ref{lem:self-improvement V} for $\theta = 0$   is needed, that is, $d \geq 3$ and the embedding $V \hookrightarrow \L^{2d/(d-2),2}_{x}$ holds.  One simple way to guarantee this embedding  for $d=n \geq 3$ is to assume that there is a bounded  Sobolev extension operator $V \to \W^{1,2}(\R^n)$ since then one can use the O'Neil's Sobolev embedding on $\R^n$ and restrict back to $\Omega$. Hence, this always works for pure lateral Dirichlet conditions ($V = \W^{1,2}_0(\Omega)$), using the extension by zero. For the existence of an extension operator in the case of mixed lateral boundary conditions, the most general geometric assumptions as far we are aware can be found in \cite{BBHT}.

\section{Gagliardo--Nirenberg inequalities}
\label{sec:GN}

We prove here a version of Gagliardo--Nirenberg inequalities including Lorentz norms. We work on $\R^n$.

\begin{prop}
\label{prop:GN} Let $m\ge 1$ be an integer, $\I$ be an interval and $u\in \L^\infty(\I; \L^2_{x})$ with $\nabla^m u \in \L^2(\I; \L^2_{x})$. Let $\alpha$ be a multi-index such that $0\le |\alpha|\le m$ and 
\begin{equation*}
(r,q) \in [2,\infty)^2, \quad \frac 1{r}+ \frac n{2mq}=  \frac {n+2|\alpha|}{4m}. 
\end{equation*}
Then 
$$
\|\partial^\alpha u\|_{\L^{r}(\I; \L^{q,2}_{x})} \le C(n, m,|\alpha|, r) \|\nabla^m u\|_{\L^{2}(\I; \L^{2}_{x})}^{2/r} \| u\|_{\L^{\infty}(\I; \L^{2}_{x})}^{1-2/r}.
$$
\end{prop}

\begin{proof} Let $a=|\alpha|$.  Define $b\geq 0$ by   $b-\frac n 2= a- \frac n q$. Then, for (almost) every $t\in \I$, we invoke the boundedness of $\partial^\alpha(-\Delta)^{-a /2}$ on Lorentz spaces and  the O'Neil's Sobolev embedding theorem \cite{ONeil},  
\begin{equation*}
\|\partial^\alpha u(t) \|_{ \L^{q,2}_{x}}  \lesssim  \|(-\Delta)^{a /2} u(t)\|_{\L^{q,2}_{x}}\lesssim   \|(-\Delta)^{b /2} u(t)\|_{\L^{2}_{x}}.
\end{equation*}
We next use the classical interpolation inequality 
$$
\|(-\Delta)^{b /2} u(t)\|_{\L^{2}_{x}} \lesssim \|\nabla^m u(t)\|_{\L^2_{x}}^{b/m} \|u(t)\|_{\L^2_{x}}^{(m-b)/m}.
$$
%which holds for almost every $t$. 
Using $u\in \L^\infty(\I; \L^2_{x})$, we can conclude by integrating the $r$-th power and working out the exponents. The interpolation inequality itself is easily seen by using the Fourier transform in $\L^2_{x}$. Indeed, writing $|\xi|^{2b} |\hat{u}(\xi)|^2 = (|\xi|^{2m} |\hat{u}(\xi)|^2)^{b/m}(|\hat{u}(\xi)|^2)^{(m-b)/m}$, it boils down to H\"older's inequality with exponent $\frac{m}{b} \in [1,\infty)$.
 This finishes the proof.  
\end{proof}

{\small
\bibliographystyle{amsplain}
%\addcontentsline{toc}{section}{References}

 \end{document}